\documentclass[11pt]{amsart}
\usepackage{amsthm}
\usepackage{mathrsfs}
\usepackage{graphicx, psfrag, epstopdf}
\newtheorem{Main}{\bg Theorem}

\numberwithin{equation}{section}
\renewcommand{\vec}[1]{\mathbf{#1}}
\def\vv{\vec{v}}
\usepackage{xcolor}
\usepackage{bbm}
\newcommand{\cK}{\mathcal K}
\newcommand{\cN}{\mathcal N}
\newcommand{\cQ}{\mathcal Q}
\def\tcP{\widetilde{\cP}}

\newcommand{\norm}[1]{\bigl\| #1 \bigr\|}

\def\bg{\color{bluegray}}
\definecolor{dgreen}{rgb}{0.1,0.6,0.1}

\definecolor{bluegreen}{rgb}{0.1,0.5,0.2}
\definecolor{bpurple}{rgb}{0.74,0.2,0.64}

\def\black{\color{black}}

\def\bpurple{\color{bpurple}}

\newtheorem{theo}{\bg Theorem}[section]
\newtheorem{lemma}[theo]{\bg Lemma}
\newtheorem{lemmata}[theo]{\bg Sublemma}
\newtheorem{coro}[theo]{\bg Corollary}
\newtheorem{prop}[theo]{\bg Proposition}

\theoremstyle{definition}
\newtheorem{define}[theo]{\bg Definition}

\newtheorem{remark}[theo]{\bg Remark}

\def\bPi{\mathcal P}
\def\cM{\mathcal M}

\def\hlambda{\hat{\l}}

\newcommand{\Q}{{\mathbb Q}}

\newcommand{\T}{{\mathbb T}}
\newcommand{\R}{{\mathbb R}}
\newcommand{\N}{{\mathbb N}}

\usepackage{hyperref}

\def\xia{{\vartheta}}
\def\Ka{{K}}
\def\tKa{{\tilde{\Ka}}}

\def\one{{\mathbbm{1}}}

\def\eps{{\epsilon}}
\def\breps{{\bar\epsilon}}
\def\DOO{D}
\def\holako{D}

 \newcommand{\Z}{{\mathbb Z}}

\def\tC{\tilde{C}}
\def\tf{\tilde{f}}
\def\tF{\tilde{F}}
\def\tcF{\tilde{\cF}}

\def\tK{\tilde{K}}

\def\tQ{\tilde{Q}}
\def\tcQ{\tilde{\cQ}}

\def\tY{\tilde{Y}}

\def\bdelta{\bar{\delta}}
\def\teps{\tilde{\eps}}
\def\tLambda{{\tilde\Lambda}}
\def\tmu{{\tilde\mu}}
\def\tPi{{\tilde\Pi}}

\def\tTheta{\tilde{\Theta}}

\def\MA{U}

\def\th{{\theta}}
\def\E{{\mathbb{E}}}
\def\Prob{{\mathbb{P}}}

\def\Pois{{\mathfrak{P}}}

\def\ba{{\mathbf{a}}}
\def\bc{{\mathbf{c}}}
\def\bbc{{\mathbf{c}}}
\def\bD{{\mathbf{D}}}

\def\br{{\mathbf{r}}}
\def\bs{{\mathbf{s}}}
\def\bt{{\mathbf{t}}}
\def\bT{{\mathbf{T}}}

\def\brho{{\boldsymbol{\rho}}}
\def\bPhi{{\boldsymbol{\Phi}}}

\def\brC{{\bar C}}
\def\brF{{\bar F}}
\def\brj{{\bar j}}
\def\brk{{\bar k}}

\def\brcL{{\bar \cL}}

\def\brs{{\bar s}}
\def\brt{{\bar t}}
\def\brfU{{\bar\fU}}

\def\brdelta{{\bar \delta}}
\def\brmu{{\bar\mu}}

\def\th{\underline{\theta}}
\def\cA{\mathcal A}
\def\cC{\mathcal C}
\def\cF{\mathcal F}
\def\cD{\mathcal D}
\def\cI{\mathcal  I}
\def\cK{\mathcal  K}
\def\cL{\mathcal L}
\def\cO{\mathcal O}
\def\cP{\mathcal P}
\def\cR{\mathcal R}
\def\cS{\mathcal S}

\def\fg{\mathfrak{g}}
\def\fp{\mathfrak{p}}
\def\fU{\mathfrak{U}}

\def\hK{{\hat K}}
\def\hcP{{\hat\cP}}

\def\hY{{\hat Y}}
\def\Ht{{\hat{t}}}

\def\hPhi{{\hat\Phi}}

\def\hxi{{\hat\xi}}

\def\trho{\tilde{\rho}}
\def\tE{\tilde{E}}

\def\Card{{\rm Card}}
\def\Ker{{\rm Ker}}
\def\mes{{\rm mes}}

\def\N{{\mathbb{N}}}

\def\Vol{{\rm Vol}}

\def\integers{{\mathbb{Z}}}

\def\reals{{\mathbb{R}}}

\def\a{\alpha}
\def\th{\theta}

\def\l{\lambda}

\usepackage[latin1]{inputenc}

\def\fC{{\mathfrak{C}}}
\def\fh{{\mathfrak{h}}}
\def\fm{{\mathfrak{m}}}
\def\tfm{{\tilde{\fm}}}
\def\fX{{\mathfrak{X}}}
\def\tfX{{\tilde{\fX}}}
\def\fY{{\mathfrak{Y}}}

\def\cfourier{{\mathfrak{a}}}
\def\carre{\hfill $\Box$}

\author{Dmitry Dolgopyat and Bassam Fayad}
\title[Deviations of ergodic sums]{Deviations of ergodic sums for toral translations\\II. Boxes.}

\usepackage{etoolbox}
\patchcmd{\section}{\normalfont}{\normalfont\color{bleu1}}{}{}
\patchcmd{\subsection}{\normalfont}{\normalfont\color{bleu1}}{}{}
\usepackage{lipsum}

\def\bg{\color{bleu1}}
\def\bk{\color{black}}

\definecolor{bleu1}{RGB}{0,57,128}
\begin{document}

\bg \begin{abstract}  \bk 
We study the Kronecker sequence $\{n\alpha\}_{n\leq N}$ on the torus $\T^d$ when $\alpha$ is uniformly distributed on
$\T^d.$ We show that the discrepancy of the number of
visits of this sequence to a random box, normalized by $\ln^d N$, converges as $N\to\infty$ to 
a Cauchy distribution.  The key ingredient of the proof is a Poisson limit theorem for the Cartan action
on the space of $d+1$ dimensional lattices. 
\end{abstract}

\maketitle


\tableofcontents
\addtocontents{toc}{\protect\setcounter{tocdepth}{1}}

 \section{Introduction} 
\subsection{Equidistribution of Kronecker sequences on $\T^d$}  \label{subsec.intro}
It is well known that the orbits of a non resonant translation on the torus $\T^d=\R^d/\Z^d$ are uniformly distributed. 
A quantitative measure of uniform distribution is given by the discrepancy function:  for a set 
$\cC \subset \T^d$ let 
$$D(\alpha,x,\cC, N)=\sum_{n=0}^{N-1} \one_{\cC}(x+n \alpha)-N \nu(\cC)$$
where $(\a,x) \in \T^d \times \T^d,$  $\one_\cC$ is the 
characteristic function of the set $\cC$ and $\nu$ is the Haar measure on the torus.
(We will sometimes write $\nu_d$ if we want to emphasize the dimension of the torus). 
Uniform distribution of the sequence $x+n\a $ on $\T^d$ is equivalent to the fact that, 
for regular sets $\cC,$ 
$D(\alpha,x,\cC, N)/N \to 0$ as $N \to \infty$. A step further  is the study of the 
rate of convergence to $0$ of $D(\alpha,x,\cC, N)/N$.

Already with $d=1$, it is clear that if $\a \in \T-\Q$ is fixed, the discrepancy $D(\alpha,x,\cC, N)$ displays an oscillatory 
behavior according to the position of $N$ with respect to the denominators of the best rational approximations of  $\a$. 
A great deal of work in Diophantine approximation has been done on estimating the oscillations of the discrepancy function in relation with 
the arithmetic properties of $\a \in \T$, and more generally for $\a \in \T^d$. It is of  common knowledge that  in  studying  
the discrepancies in dimension $1$    the continued fraction algorithm provides crucial help, and that the absence of an 
analogue in higher dimensions makes the study of discrepancies much harder.

 In particular, let
$$\overline{D}(\a, N)=\sup_{\cC \in \mathbb{B}} D(\alpha,0,\cC,N)$$
where the supremum is taken over all sets $\cC$ in some natural class of sets $\mathbb{B}$, for example balls or boxes.
The case of (straight) boxes was extensively studied, and  growth properties of the sequence $\overline{D}(\a, N)$ were obtained  with a special emphasis on their relations with the Diophantine approximation properties of $\a.$ 
In particular, following earlier advances of \cite{Kok, HL, O, Kh, schmidt} and others, \cite{beck} proves that 
for arbitrary positive increasing function $\phi(n)$ 
\begin{equation} \sum_{n} \frac{1}{\phi(n)}<\infty  \iff \frac{\overline{D}(\a, N)}{(\ln N)^d \phi(\ln\ln N)}   \begin{array}{l}
\text{ is  bounded for}\\
\text{ almost every  }\a \in \T^d.
\end{array} 
   \label{beckbound}  \end{equation}


In dimension $d=1$, this result 
is the content of Khinchine theorems obtained in the early 1920's \cite{Kh},
and  follows easily from well-known  results from 
the metrical theory of continued fractions (see for example the introduction of \cite{beck}).
The higher dimensional case is significantly more difficult andÊ \eqref{beckbound} was only obtained in the 1990s.

 The bound in (\ref{beckbound}) 
 focuses on the worst case scenario, that is, on how bad can the discrepancy become along a subsequence of $N$, for a fixed $\a$ in a full measure set.  

The restriction on $\a$ is necessary, since given any $\eps_n \to 0$ it is easy to see that for $\a \in \T$ {\it sufficiently   Liouville}, the discrepancy (relative to intervals) can be as bad as $N_n \eps_n$ along a suitable sequence $N_n$ (large multiples of denominators of very good rational approximations).
It is conjectured that for any $\a$ the discrepancy will be as bad as $(\ln N)^d$  
along a suitable subsequence
but not much is known better than the general lower bound $(\ln N)^{d/2}$ that holds for every sequence on $\T^d$ (\cite{Roth}). Here again, due to the use of continued fractions the latter conjecture can be easily verified in dimension $1$ (cf. discussion in \cite{beck}).

In another direction, but still studying the discrepancy for a fixed $\a$ and along subsequences of $N$, \cite{H, CILB} obtain a Central Limit Theorem in the one dimensional case of circle rotations. 
The results of \cite{H, CILB} apply either for a set of $\a$ of zero measure (so called badly approximable numbers) and a set of times of large density, or for all $\a$ but for a small 
set of times (in both cases, the time sets depend on $\a$).

By contrast, if  one lets both $\a$ and $x$ be random then it is possible to obtain asymptotic distributions of the adequately normalized discrepancy for {\it all} $N$.


This is the approach adopted by Kesten in \cite{K1,K2} (see also \cite{bellman}) where he studied the distribution of the discrepancies related to circular rotations  as $\a$ and $x$ are randomly distributed over the circle.
He proved the following result.


\medskip 

\noindent {\bf \bg Theorem \cite{K1, K2}.} 
{\it Let $0<a<b<1$ and define $$D(\alpha,x,[a,b],N)=\sum_{n=0}^{N-1} 
\one_{[a,b]}(x+k\alpha)-N (b-a).$$ There is a number
$\rho=\rho(b-a)$ such that if $(\alpha,x)$ is uniformly distributed on $\T^2$ then
$\frac{D(\alpha,x,[a,b],N)}{\rho \ln N}$ converges to the standard Cauchy distribution, that is,
$$ \nu_2 \left\{ (\alpha,x): \frac{D(\alpha,x,[a,b],N)}{\rho \ln N}\leq z \right\} \to \fC(z) $$ 
where $\nu_2$ is the Lebesgue measure on $\T^2$ and
\begin{equation}
\label{DefCauchy}
 \fC(z)=\frac{\tan^{-1} z}{\pi} +\frac{1}{2}. 
 \end{equation}
Moreover $\rho(b-a)\equiv \rho_0$ is independent of $b-a$ if $b-a\not\in \Q$ and 
it has non-trivial dependence on $b-a$ if $b-a\in \Q.$}

\medskip 


Our goal is to extend this result to higher dimensions. As in the case of other results related to discrepancies of Kronecker sequences, the main difficulty comes from the absence of a continued  fraction algorithm that was also the main tool in Kesten's proof. 

Before we describe our approach, let us mention that there  are two natural counterparts to intervals in higher dimension: balls and boxes. 
In \cite{dbconvex} we considered the case where $\cC$ is analytic and strictly convex
and showed that $D(\alpha,x,\cC, N) /N^{(d-1)/2d}$ has a limiting distribution (which however depends on $\cC$ and is not a standard stable law).

Here we address the case where $\cC$ is a box and show that
$\frac{D(\a,x,\cC, N)}{(\ln N)^d}$ converges to a Cauchy distribution. To avoid the irregular behavior of the limiting distribution 
on the size of the considered box, as is the case in Kesten's result for example, 
we introduce an additional randomness to the  parameters, by letting the lengths of the box's sides fluctuate. 
For a reason that will be explained in  the sequel we also have to apply (arbitrarily small) random linear deformations on the boxes. 

More precisely, for $u=(u_1,\ldots,u_d)$ with $0<u_i<1/2$  for every $i$, we  define a {\it box} on the $d$-torus by 
$C_{u}=[-u_1,u_1]\times \ldots [-u_d,u_d]$. Fix a small $\eta>0$ let
\begin{equation}
\label{GEta}
G_\eta =\{ \MA=(a_{ij}) \in  {\rm SL}_d(\R) :  |a_{i,i}-1|<\eta, \forall i \text{ and }|a_{i,j}|<\eta \ \forall j\neq i\}. 
\end{equation}
We denote by  $U \cC_{u}$
the image of $\cC_{u}$ by a matrix $U \in G_\eta$. 
Next, each length $u_i$ is assumed to be uniformly distributed in a segment $[v_i,w_i]$ where $v_i,w_i$ are fixed such that $0<v_i<w_i<1/2$ for every  $i$.

Let $$X=\left\{(\a,x, u,\MA ) \in \T^{d} \times \T^d \times ([v_1,w_1] \times \dots [v_d,w_d]) \times G_\eta \right\}$$
and denote by $\lambda$ the normalized Lebesgue measure on $X$.  
For $\xia=(\a,x,u,U) \in X$, define the following discrepancy function 
\begin{equation}\label{def.Vn} D(\xia,N) = \# \{1\leq m \leq N : (x+
m\a)  {\rm  \  mod  \ } 1 \in U C_{u}  \} - 2^d \left(\Pi_i u_i\right) N.\end{equation}

\begin{Main} \label{dimd} 
For 
any $z \in \R$ we have
\begin{equation} \label{cauchy} \lim_{N \to \infty} \lambda \{\xia \in X  \ / \ \frac{ D(\xia,N)}{(\ln N)^d} \leq z \} 
 = \fC(\brho z) \end{equation}
where $\fC$ is defined by the Cauchy distribution function \eqref{DefCauchy}
and
{\footnotesize
\begin{equation}
\label{DefBRho}
\brho=\frac{1}{\zeta(d+1) d!}
\left(\frac{2}{\pi}\right)^{d} 
\int\dots \int \left|
\sum_{j=1}^{\infty} \frac{\left[\Pi_{i=1}^d \sin(2 \pi j \eta_i)\right]  \sin(\pi j \eta_{d+1}) \cos (2\pi j \eta_{d+2})}{j^{d+1}} \right|
d\eta_1\dots d\eta_{d+2}
\end{equation}}
where $\displaystyle \zeta(d+1)=\sum_{n=1}^\infty \frac{1}{n^{d+1}}$ is the Riemann zeta function.
\end{Main}


 As it will be clear from the proof, the same statement holds if $\lambda$ is replaced by any probability measure on $X$ with smooth density. Actually, we could replace the two perturbations of the box, the fluctuation of the sides' lengths and the application of an ${\rm SL}_d(\R)$ matrix, by a single random linear perturbation, or by $rU$ with $r$ smoothly distributed in a neighborhood of $1$ and $U \in G_\eta$. We prefer to keep the perturbations split because their roles in the proof are quite different.  

As it is alluded in the title, the discrepancy is a special case of ergodic sums $\sum_{n=0}^{N-1} A(x+n\a).$
We refer the reader to a recent survey \cite{DFsurvey} for more results and open questions on this subject.

Our proof of Theorem \ref{dimd} shows that
for typical $\a$, a {\it quenched} limit (that is, with fixed $\a,$ and $x$ uniformly distributed on $\T^d$) 
of $D(\alpha,x,\cC, N)$ does not exist even if we would allow the normalizing sequence to depend on $\a.$
The reason is that the main contribution to 
the discrepancy comes from a small set of so called {\it small denominators} 
and, at different scales, 
different small denominators become important. Also, the number of the small denominators 
of a given size fluctuates. Therefore there is a sequence of times when the discrepancy is dominated by
a single small denominator, so, after a proper normalization we get limiting distribution of  compact support.
On the other hand, we can consider a sequence of times when there are many small denominators of approximately
equal strength, in which case the limiting distribution will be Gaussian. Since we can obtain different
limit distributions along different sequences, no limit exists as $N\to\infty$
(we refer the reader to \cite{DG, DS} for a more detailed version of this argument). 
We note that the absence of quenched limits is often observed
in zero entropy systems \cite{Buf, BF, BS, dbconvex, FK, GM, Mar2}. 
{To finish, we mention  that the paper \cite{DFsurvey} can be consulted for 
an introduction to problems and results related to limit theorems 
for toral translations, an active domain of research in the recent years.}

 \subsection{Plan of the paper}   
We now give a description of the paper's content and of  the main ingredients in the proofs. 

Section \ref{sec2} contains preliminaries and reminders.
In \S \ref{SSPP} we recall the representation of the Cauchy distribution in terms of a Poisson process. 
In \S \ref{SSRI} we present Rogers formulas that allow to compute the average and higher moments
for the number of points of a random
lattice in a given domain. 

In Section \ref{sec3}, harmonic analysis of the discrepancy's Fourier series allows us to 
isolate the frequencies that make essential contributions to the discrepancy 
and to show that they must be resonant with $\a.$
After eliminating a small measure set of vectors $\a,$ for which the resonances are too strong  we obtain that the good normalization for the discrepancy is 
$(\ln N)^d.$ 
The main outcome of Section \ref{sec3} is to reduce the proof of Theorem \ref{dimd} to that of Theorem \ref{ThPoisResd} 
establishing a Poisson limit theorem for the distribution of the small denominators that appear in the (resonant) Fourier terms that contribute to the discrepancy.

Namely, for each $\breps>0$ we need to prove the Poisson limit theorem  for the sequence
$$(\star)\quad \left\{(\ln N)^d  \prod_i \bar{k}_i   \norm{\langle k,\a \rangle}, N\langle k,\a \rangle \text{ mod } 2 , \{ \bar{k}_1 u_1\},\ldots,\{ \bar{k}_d u_d\},  \{\langle k,x \rangle\}  \right\}_{k \in Z(\xia, N)} $$
where $\langle \cdot,\cdot \rangle$, $\|\cdot\|$ and $\{ \cdot\}$ denote respectively the Eucledian scalar product, the closest distance to integers and the fractional part, and where 
$$\brk_i=a_{i,1} k_1+\dots+a_{i,d}k_d, $$
\begin{multline*} \label{defZ} Z(\xia,N) = \left\{ k \in W(\xia,N): \bar k_1>0 \text{ and } \exists m \in \Z  {\rm \ such \ that \ }  \right. \\ \left.  k_1 \wedge \ldots \wedge k_d  \wedge m =1  {\rm \ and \ }  \norm{\langle k,\a \rangle}= |\langle k,\a \rangle+m| \right\} . \end{multline*} 
and 
\begin{multline*} W(\xia,N) := \left\{k\in \Z^d : \left|\prod_{i=1}^d \bar{k}_i\right|<N,   \right. \\ 
\left. \forall i=1,\ldots,d, \quad  |\bar{k}_i|\geq 1, \;\;
\left|\prod_{i=1}^d \bar{k}_i \right|\, \norm{\langle k,\a \rangle} \leq \frac{1}{\breps (\ln N)^d} \right\}  . 
\end{multline*}

In Section \ref{sec.reduction}, we reduce the Poisson limit of the first two coordinates 
of $(\star)$ to a Poisson limit theorem 
(Theorem \ref{ThPLat}) for the number
of visits to a cusp by orbits of the Cartan action on   the space of $d+1$ dimensional 
unimodular lattices
$\cM=SL_{d+1}(\reals)/SL_{d+1}(\integers)$. 
To prove the Poisson limit for all components of $(\star)$, we need to show that the remaining components are asymptotically independent of the first two. This requires an extra work that is done in Proposition \ref{ThLacun}, where the argument is similar to the original analysis of Kesten.

The proof of Theorem \ref{ThPLat} occupies  Sections \ref{sec.poisson.abstract}, 
\ref{SSSLMix},  and
\ref{SSDenom}. 
In Section  \ref{sec.poisson.abstract}, using martingale methods, we  establish an abstract Poisson limit theorem 
that is well adapted to variables coming from dynamical systems.  Establishing Poisson limit theorems for dynamical systems
is a subject with rich literature (see \cite{AV, CC, DGS, D, FHN, Hir1, Hir2, Pit} and the references therein). 
 The most relevant work for our purposes is the paper \cite{D} where a Poisson Limit Theorem is 
proven for partially hyperbolic systems assuming that the images of local unstable manifolds become 
equidistributed at sufficiently fast rate.

In the present setting there are two new difficulties. First, the geometry of the cusp is quite complicated
(especially for large $d$), in the sense that we do not know what is the order of $\brk_i$s that contribute 
to the resonances in ($\star$). However Rogers 
identities (\cite{Siegel, Rogers}) provide sufficiently strong control to handle
this issue.  
Secondly, 
we need to consider the action
of the full diagonal subgroup of $SL_{d+1}(\R)$ because, for a typical resonance,
$\brk_1,\brk_2\dots,$and $\brk_d$ have very different sizes.
For such higher rank actions  there is no notion of "unstable manifold" because
there is no notions of "future" and "past" and going to infinity in different Weyl chambers gives different 
expanding and contracting directions. In the present setting, we are able to prove a Poisson limit
theorem using the fact that the long leaves of the Lyapunov foliations 
become uniformly distributed at a polynomial rate, except, possibly,
for a small measure set. The fact that the we need to prove the Poisson Limit Theorem for higher rank 
subgroups constitutes the main novelty of Sections 
\ref{sec.poisson.abstract}--\ref{SSDenom}. 

The relevant equidistribution results for unipotent subgroups of $SL_{d+1}(\reals)$ acting on $\cM$ are presented in Section~\ref{SSSLMix}. 
To exploit these equidistribution results in the proof of $(\star)$, we 
introduce additional parameters in the form of small affine deformations of
the box. Indeed, if we work with the straight boxes we would have to establish a Poisson Limit Theorem for lattices having a smooth distribution on a positive codimension submanifold of $\cM$; while with the randomly slightly tilted boxes we have to establish a Poisson Limit Theorem for lattices having a smooth density
on $\cM$.

In Section \ref{SSDenom} the conditions of the abstract theorem of Section \ref{sec.poisson.abstract}
are verified for the Cartan action on $\cM$ using the equidistribution results of Section \ref{SSSLMix}.

In section \ref{ScSmall} we discuss the discrepancy for the number of visits
to boxes of small size $N^{-\gamma}$, $\gamma<1/d$, and we obtain a similar result to the case $\gamma=0$ that corresponds to our main Theorem \ref{dimd}.   The case $\gamma=1/d$ was studied in \cite{Mar1} where a limit distribution was obtained without any normalization. 
In the case $\gamma>1/d$, the problem 
is vacuous since most orbits do not visit a ball of size 
$N^{-\gamma}$ before time $N$ (by the Borel Cantelli Lemma).

In Section \ref{ScCont}
we discuss the continuous time case, that is, we study the discrepancies corresponding to linear flows on the torus. We show that in the case of 
boxes the discrepancy is bounded in probability. Namely, the indicator function of a box is a coboundary
with probability one. We actually get convergence in distribution of the discrepancies without any normalization.  However, the method used to prove Theorem \ref{dimd} gives a Cauchy limit theorem for continuous discrepancies relative to {\it balls,} and this only in dimension $d=3$. Indeed, the latter  is in sharp and curious contrast with the higher dimension case obtained in  \cite{dbconvex} that states that for $d\geq 4$ the continuous discrepancies relative to balls  converge in distribution after normalization by a factor  $T^{(d-3)/2(d-1)}$. 

Finally, 
some technical estimates are collected in the appendices.

 \section{Preliminaries} \label{sec2} 
 \subsection{Poisson processes.} 
\label{SSPP}

Recall that a random variable $N$ has Poisson distribution with parameter $\lambda$ if
$\Prob(N=k)=e^{-\lambda} \frac{\lambda^k}{k!}.$ Now easy combinatorics shows the following facts 

(I) If $N_1, N_2\dots N_m$ are independent random variables and each $N_j$ has Poisson distribution with parameter
$\lambda_j$, then $\displaystyle N=\sum_{j=1}^m N_j$ has 
Poisson distribution with parameter $\displaystyle \sum_{j=1}^m \lambda_j.$

(II) Conversely, take $N$ points distributed according to a Poisson distribution with parameter $\lambda$ and color each
point independently with one of $m$ colors where color $j$ is chosen with probability $p_j.$ Let $N_j$ be the
number of points of color $j.$ Then $N_j$ are independent and $N_j$ has Poisson distribution with parameter
$\lambda_j=p_j\lambda.$

Now let $(\fX, \fm)$ be a measure space. By a Poisson process on this space we mean a random point process on $\fX$
such that if $\fX_1, \fX_2\dots \fX_m$ are disjoint sets and $N_j$ is the number of points in $\fX_j$
then $N_j$ are independent Poisson random variables with parameters $\fm(\fX_j)$ (note that this definition is
consistent due to (I)). We will write $\{x_j\}\sim \Pois(\fX, \fm)$ to indicate that $\{x_j\}$ is a Poisson process with parameters
$(\fX, \fm).$ If $\fX\subset \reals^d$ and $\fm$ has a density $f$ with respect to the Lebesgue measure we
say that $f$ is the intensity of the Poisson process.
 
 The following properties of the Poisson process are straightforward consequences of (I) and (II) above, and their proofs can be  found in the monographs \cite{King, SamT}.
\begin{lemma}
\label{LmPT} 
(see \cite{King}, \S\S 2.3 and 5.2)

(a) If $\{\Theta_j'\}\sim\Pois(\fX, \fm')$ and $\{\Theta_j''\}\sim\Pois(\fX, \fm'')$ are independent then
$$\{\Theta_j'\}\cup \{\Theta_j''\}\sim \Pois(\fX, \fm'+\fm'').$$

(b) If $\{\Theta_j\}\sim \Pois(\fX, \fm)$ and $f:\fX\to \fY$ is a measurable map then
$\{f(\Theta_j)\}\sim\Pois(\fY, f^{-1}\fm).$

(c) Let $\fX=\fY\times Z,$ $\fm=\nu\times \lambda$ where $\lambda$ is a probability measure on $Z.$
Then $\{(\Theta_j, \Gamma_j)\}\sim \Pois(\fX, \fm)$ iff $\{\Theta_j\}\sim \Pois(\fY, \nu)$ and 
$\Gamma_j$ are random variables independent from $\{\Theta_j\}$ and from 
each other and distributed according to $\lambda.$

(d) If in (c) $\fY=Z=\reals$ then
$\tTheta=\{\Gamma_j \Theta_j\}$ is a Poisson process. If  $\{\Theta_j\}$  has measure $f(\theta)d\theta$ then $\tTheta$ has  measure $\tf(\theta)d\theta$ with 
$$ \tf(\theta)=
\E_\Gamma \left(f\left(\frac{\theta}{\Gamma}\right)\frac{1}{|\Gamma|}\right) .$$

\end{lemma}

Next, recall \cite[Chapter XVII]{Fel} that the Cauchy distribution is the unique (up to scaling) symmetric distribution such that if $Z, Z'$ and $Z''$ are independent random variables with that distribution then $Z'+Z''$ has the same distribution as $2Z$.   This gives  the following representation of the Cauchy distribution.

\begin{lemma}
\label{LmPoissonCauchy}

(a) If $\lbrace \Theta_j \rbrace$ is a Poisson process on $\reals$ with measure $c \theta^{-2} d\theta$   then 
$$ \lim_{\delta\to 0} \frac{1}{\rho}\sum_{\delta<|\Theta_j|} \Theta_j $$
has a standard  Cauchy distribution, with $\rho={c\pi}$. 



(b) If $ \lbrace \Theta_j \rbrace$ is a Poisson process on $\reals$ with constant intensity $c$ and if $\Gamma_j$ are
iid random variables having a symmetric distribution with compact support then 
$$ \lim_{\breps \to 0} \frac{1}{\rho} \sum_{|\Theta_j|<\breps^{-1}} \frac{\Gamma_j}{\Theta_j} $$
has a standard Cauchy distribution with $\rho=c\E(|\Gamma|) \pi$.

\end{lemma}

We provide a
sketch of the proof to illustrate the idea of the argument. For more detailed presentation we refer the readers to
\cite[Theorem 1.4.2]{SamT}, or \cite[Appendix B]{DS}.

\noindent {\it Sketch of proof.} To see part (a),  let $\{U_j'\}, \{U_j''\}$ and  $\{U_j\}$ be independent Poisson processes with measure $c.$
We want to show that $\sum \frac{1}{U_j'}+\sum \frac{1}{U_j''}$ have the same distribution as $\sum \frac{2}{U_j}$. But 
$$ \sum \frac{1}{U_j'}+\sum \frac{1}{U_j''}=\sum_{y\in \{U_j'\}\cup \{U_j''\}} \frac{1}{y}$$
and we finish by observing, in light of Lemma \ref{LmPT} (a) and (b), that both $\{U_j'\}\cup \{U_j''\}$ and $\{\frac{U_j}{2}\}$ are Poisson processes with intensity $2c.$

The proof of part (b) follows the same idea as the proof of part (a), but now we use parts (b), (c) and (d) of Lemma \ref{LmPT}.

 \subsection{Siegel and Rogers identities} 
\label{SSRI}
For $d\in \N,$ $d\geq 2$, denote  the space of unimodular $(d+1)$-dimensional lattices  by 
$\cM_{d+1}= {\rm SL}_{d+1}(\R) / {\rm SL}_{d+1}(\Z)$ and let $\mu$ be the Haar measure on $\cM_{d+1}$. 
Denote    
\begin{equation}
\label{C1C2}
 \bc_1=\zeta(d+1)^{-1}, \quad \bc_2=\zeta(d+1)^{-2}, \quad \text{where }
\zeta(d+1)=\sum_{n=1}^\infty n^{-(d+1)} 
\end{equation}
is the Riemann zeta function.

The following identities (see \cite{Siegel, Rogers} as well as \cite{Mar1, V}) play an important role in our argument. 
Let $f, f_1, f_2$ be piecewise smooth functions with compact support on $\reals^{d+1}.$
For a lattice $\cL \subset \cM_{d+1}$, we say that a vector in $\cL$ is prime if it is not an integer multiple of another vector in $\cL$. Let
$$ F(\cL)=\sum_{\vv\in \cL, \text{ prime}} f(\vv), \quad
\brF(\cL)=\sum_{\vv_1\neq \pm \vv_2\in \cL, \text{ prime } } f_1(\vv_1) f_2(\vv_2). $$
$F$ is called {\it Siegel transform of $f$} so we will sometimes denote $F$ by $\cS(f).$
\begin{lemma} 
\label{LmRI} We have 

\begin{align*} (a)\quad \int_{\cM} F(\cL) d\mu(\cL)&=\bc_1 \int_{\reals^{d+1}} f(x) dx, \\
(b)\quad \int_{\cM} \brF(\cL) d\mu(\cL)&=\bc_2 \int_{\reals^{d+1}} f_1(x) dx \int_{\reals^{d+1}} f_2(x) dx. \end{align*}
(c) Consequently
\begin{multline*} \int_{\cM} F^2(\cL) d\mu(\cL)=\bc_1\int_{\reals^{d+1}} f^2(x) dx \\+\bc_1\int_{\reals^{d+1}} 
f(x)f(-x) dx+\bc_2 \left(\int_{\reals^{d+1}} f(x) dx\right)^2. \end{multline*}
\end{lemma}

 \section{  Negligible contribution of non-resonant terms} \label{sec3}

As we already mentioned,  the proof of the main Theorem~\ref{dimd} is obtained by applying the results of
\S \ref{SSPP} to a sum of resonant terms in the Fourier series of $\DOO(\xia,N)/(\ln N)^d$. 
But first we need to isolate 
the resonant terms that contribute to the limiting distribution. This will be done in the current section, the outcome of which is summarized in the Proposition \ref{prop.reduction} below. 
The proof of Proposition \ref{prop.reduction} is independent from the rest of the paper and can be skipped in a first reading.

 \subsection{} \label{res.reduction}  

Recall from \S \ref{subsec.intro}
the definition
$$X=\left\{(\a,x, u,\MA) \in \T^{d} \times \T^d \times ([v_1,w_1] \times \dots [v_d,w_d]) \times G_\eta \right\}$$
where $G_\eta$ is given by \eqref{GEta}.
For $\xia \in X$ and $k\in \Z^d$, we use the notation 
\begin{equation}
\label{sidedual}
\brk_i=a_{i,1} k_1+\dots+a_{i,d}k_d 
\end{equation}

Writing the Fourier series of the characteristic function of a box we get for the discrepancy $\DOO(\xia,N)$ defined in \eqref{def.Vn} 
$$\DOO(\xia,N)  = \sum_{k\in \Z^d-\{0\}} U_{k}(\xia,N)$$
where
$$U_{k}(\xia,N)=\cfourier \prod_i \left(\frac{\sin\left(2\pi \bar{k}_i u_i\right)}{\bar{k}_i} \right)
\frac{\sin(\pi N \langle k,\a \rangle)}{\sin(\pi \langle k,\a \rangle)}\cos(2\pi \langle k,x \rangle+\varphi_{k,\a, N})$$
and $\varphi_{k,\a,N}=  \dfrac{\pi(N-1)\langle k,\a \rangle}{2}$, $\cfourier=\dfrac{1}{\pi^d}$, 
$\displaystyle \langle k,x \rangle=\sum_{i=1}^d k_ix_i$.

Fix  a small number $\breps>0.$
For $y \in \R$ we use the notation $\|y\|$ for the closest distance of $y$ to the integers. In all this section, we will use the notation $u=O(v)$, or equivalently $u\ll v$, when  $|u|\leq C|v|$ for some constant $C$ that does not depend on $\breps$ or $N$.

 Define 
$W(\xia,N)=W(\a,(a_{i,j}),N)$
by
\begin{multline} \label{defW} W(\xia,N) := \left\{k\in \Z^d : \left|\prod_{i=1}^d \bar{k}_i\right|<N,   \right. \\ 
\left. \forall i=1,\ldots,d, \quad  |\bar{k}_i|\geq 1, \;\;
\left|\prod_{i=1}^d \bar{k}_i \right|\, \norm{\langle k,\a \rangle} \leq \frac{1}{\breps (\ln N)^d} \right\}  . \end{multline} 

Next,  we let
\begin{multline} \label{defZ} Z(\xia,N) = \left\{ k \in W(\xia,N): \bar k_1>0 \text{ and } \exists m \in \Z  {\rm \ such \ that \ }  \right. \\ \left.  k_1 \wedge \ldots \wedge k_d  \wedge m =1  {\rm \ and \ }  \norm{\langle k,\a \rangle}= |\langle k,\a \rangle+m| \right\} . \end{multline} 

Then define
\begin{equation}\label{eq.d7}
\bar{\holako}(\xia,N) = \sum_{k \in Z} \frac{\Gamma_{k}(\xia,N)}{\Omega_{k}(\xia,N)}  \end{equation}
where
\begin{equation}
\label{DefTheta}
\Omega_{k}(\xia,N)=\left(\prod_{i=1}^d \bar{k}_i\right) \norm{\langle k,\a \rangle}  (\ln N)^d,
\end{equation}
\begin{equation}
\label{GammaPhi}
 \Gamma_{k}(\xia,N)=\frac{2 \cfourier}{\pi} {\phi \left( \bar{k}_1 u_1,\ldots, 
 \bar{k}_d u_d, N\langle k,\a \rangle,\langle k,x \rangle+\varphi_{k,\a,N}\right)}, 
\end{equation} 
and 
\begin{equation}
\label{DefGamma}
\phi(\eta_1,\ldots,\eta_d,\eta_{d+1},\eta_{d+2})=
\end{equation}
$$ \sum_{j=1}^{\infty} \frac{\left[\Pi_{i=1}^d \sin(2 \pi j \eta_i)\right]  \sin(\pi j \eta_{d+1}) \cos (2\pi j \eta_{d+2})}{j^{d+1}}. $$

The purpose of this section is show that  $\left|\dfrac{\holako}{(\ln N)^d}-\bar{\holako}\right|$ is small in probability.

\begin{prop} \label{prop.reduction} 
For any $\upsilon>0$, if we take $\breps>0$ sufficiently small and then $N$ 
sufficiently large we have 
\begin{equation} \label{conv.dist} \lambda\left( \left\{ \xia \in X : \left| \frac{\holako(\xia,N)}{(\ln N)^d} 
- {\bar{\holako}(\xia,N)} \right|  \geq \upsilon \right\} \right)\leq \upsilon . \end{equation}
\end{prop}

\begin{remark}{ \it Differences with the discrepancies relative to convex  sets.} 
Proposition \ref{prop.reduction} identifies the normalization term $(\ln N)^d$ and the resonant terms in the Fourier series of the discrepancy function that contribute to its limiting distribution after normalization. These terms involve multiplicative 
small denominators of the form $\displaystyle \left(\prod_i \bar{k}_i  \right)\norm{\langle k,\a \rangle}$ with 
$\displaystyle \left|\prod_i \bar{k}_i\right| \leq N$. 
The fact that frequency vectors $k$ coming from many different scales
in the set $\displaystyle \left|\prod_i \bar{k}_i\right| \leq N$ 
yield  small denominators $\displaystyle \left(\prod_i \bar{k}_i\right)  \norm{\langle k,\a \rangle}$ 
of comparable strength, as well as the asymptotic independence of the small denominators
from different scales,
are key in proving that the limiting distribution of the normalized discrepancy turns out to be a classical stable law (see Section \ref{sec.poissonreduction} below). A similar analysis lies behind Kesten's proof in the case $d=1$. 

In contrast, the resonant terms in the Fourier series of the discrepancy function relative to a convex set involve small denominators
of the  form $|k|^{\frac{d+1}{2}}  \norm{\langle k,\a \rangle}$. 
The latter form of the small denominators  limits the contributing terms to frequencies $k$ 
of the same scale.
Namely it is shown in \cite{dbconvex} that the main contribution to the discrepancy comes from
frequencies $|k|$ that are of the order of $N^{\frac{1}{d}}$, which leads to the normalization factor 
$N^{\frac{d-1}{2d}}$. Using Dani correspondence and mixing of expanding translates of 
horoshperical subgroups, it was shown in \cite{dbconvex} that the limiting distribution
of the normalized discrepancy function is given by a Siegel transform of a certain function
on the space of marked lattices where the marked lattice is chosen at random.
 This is a special case of the distribution obtained by taking a certain function on an appropriate
moduli space and evaluating them at a random point.
Such distributions are not well studied in probabilistic literature, even though
recently it has been shown that they appear in many problems related to distribution
of ergodic averages of renormalizable systems (\cite{BF, Mar2, GM}). 
The first steps towards creating the general theory of these distributions are made in
\cite{CM, MV} but more work is needed in this direction. 
\end{remark}

We will need several lemmas to prove Proposition \ref{prop.reduction}. The lemmas will involve $L^2$ estimates with respect to the variables $(\a,x) \in \T^{2d}$ as well as an exclusion 
of a 
small measure of frequencies $\a$ where the discrepancy may go completely out of control due to the very small denominators  $\Pi_i \bar{k}_i  \norm{\langle k,\a \rangle}$. 

 \subsection{}   Let 
$$\holako_1(\xia,N)=  \sum_{|k_i| \leq N, \; k \neq (0,\ldots,0)} U_{k}(\xia,N).$$
\begin{lemma} \label{lemma.hol1} We have 
\begin{equation}\label{hol1} {\| \DOO-\holako_1\|}_2^2 =\cO(1). \end{equation}
\end{lemma}

The $L^2$ norm in \eqref{hol1} and below in Section \ref{sec3} 
is taken with respect to the variables $(\a,x) \in \T^{2d}$.

\begin{proof} Assume $\xia \in X$ given. 
Then for any $q \geq N$ and any $q_1,\ldots,q_{d-1}\in \N$ 
let $\hY(q, q_1,\dots q_{d-1})$ be the set of vectors $k\in \Z^d$ such that for some permutation 
of the indices $1,\ldots,d$
we have
 $|{\bar k}_{i_d}| \in [q,q+1]$ and $|\bar{k}_{i_j}| \in [q_j,q_{j}+1]$ for every $j\in [1,d-1].$ 
Note that the cardinality of $\hY(q, q_1,\dots q_{d-1})$ is uniformly bounded.

 Since for any $\omega \in \T$ and any $m \neq 0$, 
 $$\left| \frac{\sin\left(2\pi m\omega\right)}{m}\right| < \min\left(2\pi |\omega|, \frac{1}{|m|}\right)=
 \cO\left(\frac{1}{|m|+1}\right),$$ 
 the contributions of high frequencies 
  can be bounded as follows. 
 $${\| \DOO-\holako_1\|}_2^2$$
\begin{align*} 
&\ll \sum_{q \geq N, q_1,\ldots,q_{d-1} \geq 0} \frac{1}{q^2(q_1+1)^2 \ldots (q_{d-1}+1)^2} 
\sum_{k\in \hY(q, q_1,\dots q_{d-1})}
\int_{\T^d} {\left(\frac{\sin(\pi N \langle k,\a \rangle}{\sin(\pi \langle k,\a \rangle)}\right)}^2 d\a \\ 
&\ll  \sum_{q \geq N, q_1,\ldots,q_{d-1} \geq 0} \frac{1}{q^2(q_1+1)^2 \ldots (q_{d-1}+1)^2} N 
\ll 1. \qedhere
 \end{align*}
\end{proof} 

 \subsection{}  
\label{nosmallk} Define $S(\xia,N)=S(\MA,N) := \{k\in \Z^d : 
   |k_i|\leq N, |\bar{k}_i|\geq 1 \}$. Then let
$$\holako_2(\xia,N) =  \sum_{ k \in S(\xia, N)} U_{k}(\xia,N).$$
We want to  replace  $\holako_1$ by $\holako_2$. For a fixed matrix $\MA$, we want to bound the contributions of frequencies $k$ such that $|\bar{k}_{i_d}|< 1$ for at least one index $i_d \in [1,d]$. Observe first that since $\MA$ is close to Identity then $|\bar{k}_i|\leq 2N$ for every $i$. 
Moreover, there exists $C(d)$ such that for every $(q_1,\ldots,q_{d-1}) \in [0,2N]^{d-1}$ there is at most $C(d)$ vectors $k\in [-N,N]^d$ such that 
$|\bar{k}_{i_d}|\leq1$ and $|\bar{k}_{i_j}| \in [q_j,q_j+1]$ for every $j\in [1,d-1]$, where $i_j$ is some permutation of the indices $1,\ldots,d$. We call $\hY(q_1,\ldots,q_{d-1})$ the latter set of $k$. We then exclude the translation vectors $\a$ for which there exists 
$(q_1,\ldots,q_{d-1}) \in [0,2N]^{d-1}$ with at least one $k\in \hY(q_1,\ldots,q_{d-1})$ satisfying 
$\displaystyle \left| \prod_{i=1}^{d-1} (q_i+1) \right|  \norm{\langle k,\a \rangle} \leq \breps/(\ln N)^{d-1}$. The excluded set 
$E_N(\MA)$ has Lebesgue measure of order $\breps$. 
\begin{lemma} \label{lemma.hol2}
$\displaystyle
{\| \holako_2-\holako_1\|}^2_{L_2((\T^d-E_N) \times \T^d)} \ll  \frac{(\ln N)^{2(d-1)}}{\breps}. $
\end{lemma}

\begin{proof} Let 
\begin{multline*} B_p\left((q_1,\ldots, q_{d-1}),\MA\right)=  
\{ \a \in \T^d : \exists k \in 
\hY(q_1,\ldots,q_{d-1})(\MA),  \\  p \breps/(\ln N)^{d-1} \leq \left( \prod_{i=1}^{d-1} (q_i+1) \right)  \norm{\langle k,\a \rangle} \leq (p+1) \breps/(\ln N)^{d-1} \}. \end{multline*}
Then  
${\rm Leb}\left(B_p\left((q_1,\ldots, q_{d-1}),\MA\right)\right) \ll \frac{\breps}{(q_1+1)\ldots (q_{d-1}+1) (\ln N)^{d-1}}.$ 
Hence 
\begin{align*} &{\| \holako_2-\holako_1\|}_{L_2((\T^d-E_N) \times \T^d)}^2 \ll \\
&  \sum_{q_1,\ldots,q_{d-1} \in [0,2N]^{d-1}} \sum_{p\geq 1} \frac{\breps}{(q_1+1)\ldots 
(q_{d-1}+1) (\ln N)^{d-1}}\frac{(\ln N)^{2(d-1)}}{\breps^2 p^2} \\
&\ll \frac{(\ln N)^{2(d-1)}}{\breps} . 
\qedhere
\end{align*}
\end{proof}

 \subsection{} For $k\in \Z^d$, denote $K(\xia,k)= \Pi_{i=1}^d \bar{k}_i$. 
Let $$\bar{S}(\xia,N):= \{k\in \Z^d:   |K(k)|\leq N, |\bar{k}_i|\geq 1 \} \text{ and } 
\holako_3(\xia,N)  =  \sum_{ k \in \bar{S}} U_{k}(\xia,N).$$
\begin{lemma} \label{lemma.hol3}
\begin{equation}\label{hol3}{\| \holako_3- \holako_2\|}_2^2 =\cO\left(\left(\ln N\right)^{d-1}\right). \end{equation}
\end{lemma}

\begin{proof} 
$${\| \holako_3- \holako_2\|}_2^2 \leq   
\sum_{k\in S,   |K(k)|\geq N} \frac{1}{K(k)^2} \int_{\T^d} {\left(\frac{\sin(\pi N \langle k,\a \rangle)}{\sin(\pi \langle k,\a \rangle)}\right)}^2 d\a  
\leq  \sum_{k\in S,   |K(k)|\geq N} \frac{N}{K(k)^2} . $$

For $s\in \N$, let 
$A_s = \{ k \in S: |K(k)| \in [e^{s}N,e^{s+1}N] \} $
and observe that  $\Card(A_s) \ll e^s N (\ln N+s)^{d-1}$. Thus 
$$ {\| \holako_3- \holako_2\|}_2^2    \ll \sum_{s=0}^\infty  {e^s N (\ln N +s)^{d-1}}    \frac{N}{(e^sN)^2} 
\ll \ln N^{d-1} . \qedhere $$
\end{proof}


 \subsection{}  Recall the definition \eqref{defW} of $W(\xia,N)=W(\MA,\a,N)$ 
$$  W(\xia,N) = \left\{k \in \bar{S}((a_{i,j}),N) : 
|\Pi_{i=1}^d \bar{k}_i | \norm{\langle k,\a \rangle} \leq \frac{1}{\breps (\ln N)^d} \right\} 
$$
and let 
$$\holako_4(\xia,N) =  \sum_{k \in W(\xia, N)} U_{k}(\xia,N).$$

\begin{lemma} \label{lemma.hol4}
$\displaystyle
{\| \holako_4- \holako_3\|}_{L_2((\T^d-E_N) \times \T^d)} \ll \sqrt{\breps} \; (\ln N)^d. $
\end{lemma}

\begin{proof}  
Since $k \in \bar S$ and $\MA$ is close to Identity, we have that $1\leq |\bar{k}_i| \leq 2N$ for every $i$. 
Now, for every $q_1,\ldots,q_d \in [1,2N]^d$ there are at most $C(d)$  vectors $k\in [-N,N]^d$ such that 
 $|\bar{k}_{i}| \in [q_i,q_i+1]$. We denote the latter set of vectors $Y(q_1,\ldots,q_d)$.  We have that 

$${\| \holako_4- \holako_3\|}^2_{L_2((\T^d-E_N) \times \T^d)}  \ll     \sum_{(q_1,\ldots,q_d) \in [1,2N]^d} A_{Y(q_1,\ldots,q_d)}$$
where
$$A_{Y(q_1,\ldots,q_d)} = \sum_{k \in Y(q_1,\ldots,q_d)} \int_{\T^d}   \frac{1}{((\Pi_{i=1}^d q_i)  \norm{\langle k,\a \rangle} )^2} \one_{ \left((\Pi_{i=1}^d q_i) \; \norm{\langle k,\a \rangle}\right) \geq 1/ \breps (\ln N)^d } \;d\a.$$
Consider for each $k \in Y(q_1,\ldots,q_d) $ and $p\in \N$ the sets 
$$B_{k,p}= \left\{ \a \in \T^d  :  \frac{p}{\breps(\ln N)^d}  \leq  
\left(\prod_{i=1}^d q_i\right)   \norm{\langle k,\a \rangle}  < \frac{p+1}{\breps(\ln N)^d} \right\}.$$
We have that 
$${\rm Leb}_{\T^d}(B_{k,p})\leq \frac{1}{\breps (\Pi_{i=1}^d q_i) (\ln N)^d}.$$ 
Thus
$$ A_{Y(q_1,\ldots,q_d)} \leq   C \frac{1}{\breps  (\Pi_{i=1}^d q_i) (\ln N)^d}  \sum_{p=1}^\infty \frac{\breps^2(\ln N)^{2d}}{p^2} 
\leq  C \breps \frac{(\ln N)^d}{\Pi_{i=1}^d q_i} $$
and the claim follows as we sum over $(q_1,\ldots,q_d) \in [1,2N]^d$. \end{proof}

 \subsection{}    Since the  terms in $\holako_4$ satisfy $\norm{\langle k,\a \rangle} \leq \frac{1}{\breps (\ln N)^d}$, 
we can replace $U_{k}$ defined in \S \ref{res.reduction} by 
$$V_{k}(\xia,N)=\cfourier \prod_i \left(\frac{\sin\left( 2\pi \bar{k}_i u_i\right)}{\bar{k}_i}\right) 
\frac{\sin(\pi N \langle k,\a \rangle)}{\pi \norm{\langle k,\a \rangle}}\cos(2\pi \langle k,x \rangle+\varphi_{k,N,\a}))$$
and introduce 
 $$\holako_5(\xia,N) = 2\sum_{k \in W(\xia, N),k_1>0} V_k.$$
 Note that for small $\theta$ we have
 $|\sin\theta-\theta|\leq |\theta|^3,$ so
 $$ \left|\frac{1}{\sin\theta}-\frac{1}{\theta}\right|\leq 
 \left|\frac{\theta^3}{\theta\sin\theta}\right|\leq 2\theta. $$
Hence for $k \in W(\xia, N)$, 
 \begin{equation*} |U_k-V_k|\leq C \frac{\norm{\langle k,\a \rangle}}{  |\Pi_{i=1}^d \bar{k}_i|} \leq 
 \frac{  C }{  \breps (\ln N)^d \left(\left|\Pi_{i=1}^d \bar{k}_i\right|\right)^2}.
  \end{equation*}
Summing over $k\in W(\xia, N)$ we have thus obtained
\begin{lemma} \label{lemma.hol6}
\begin{equation}\label{hol6}|\holako_5(\xia,N) -\holako_4(\xia,N)|\leq \frac{C}{\breps (\ln N)^{d}}. \end{equation}
\end{lemma} 

 \subsection{}  
{\it \bg Proof of Proposition \ref{prop.reduction}.}
Putting together Lemmas \ref{lemma.hol1}--\ref{lemma.hol6}, we 
see that $(\holako_5(\xia,N)-\DOO(\xia,N)) / (\ln N)^d$ satisfies \eqref{conv.dist} if $\breps>0$ is sufficiently small and then $N$ is sufficiently large.

Recall the definition of $\bar{\holako}(\xia,N)$ given in \eqref{eq.d7}.
  The difference between $(\ln N)^d \bar{\holako}(\xia,N)$ and 
$D_5$ is that for $k \in Z(\xia, N)$,   
we comprise in $(\ln N)^d \bar{\holako}(\xia,N)$ all its multiples whereas in $D_5$ we take only multiples such that $p k \in W$. 
Since we have already shown in Lemmas \ref{lemma.hol1}--\ref{lemma.hol6} that the frequencies which are not in $W(\xia, N)$
make a negligible contribution after normalization by $(\ln N)^d$ as $\breps \to 0$ and $N \to \infty$
it follows that for each $\upsilon$ 
$$ \lambda\left( \left\{ \xia \in X :   \left|\frac{\holako_5(\xia,N)}{(\ln N)^d}- {\bar{\holako}(\xia,N)} \right|  \geq 
\frac{\upsilon}{2} \right\} \right)\leq \frac{\upsilon}{2} $$
provided that $\breps$ is sufficiently small and $N$ is sufficiently large.
This finishes the proof of Proposition \ref{prop.reduction}. \qed

 \section{  Poisson distribution of small divisors} \label{sec.poissonreduction}

In this section we reduce the Cauchy limit of the discrepancies to a Poisson limit theorem (Theorem \ref{ThPoisResd}) for the small divisors $\displaystyle \left(\prod_i \bar{k}_i\right)  \norm{\langle k,\a \rangle}$.

 \subsection{}   The following is the bulk of our proof of Theorem \ref{dimd}.


\begin{Main} \label{ThPoisResd} 
Assume that the distribution of $\xia \in X$ is absolutely continuous
with respect to the Lebesgue measure. For any $\bar{\eps}>0$, as $N \to \infty$, the process
$$\left\{(\ln N)^d  \left(\prod_i \bar{k}_i \right)  \norm{\langle k,\a \rangle}, N\langle k,\a \rangle \text{ mod } 2 , \{ \bar{k}_1 u_1\},\ldots,\{ \bar{k}_d u_d\},  \{\langle k,x \rangle\}  \right\}_{k \in Z(\xia, N)} $$
where $Z(\xia, N)$ is defined by \eqref{defZ},
converges to a Poisson process on 
$$\left[- \frac{1}{\breps}, \frac{1}{\breps}\right]\times (\R/ (2\Z)) \times \T^{d+1}$$ 
with intensity 
\begin{equation}
\label{MasterIntensity}
c=2^{d-1} \bc_1/d!, \quad \bc_1=1/\zeta(d+1). \end{equation}
\end{Main}

\begin{remark} \label{remark.smooth} It is sufficient to prove Theorem   \ref{ThPoisResd} when $\xia$ is distributed according to a smooth bounded density with respect to the normalized Lebesgue measure $\lambda$ on $X.$

Indeed,  let us suppose that Theorem \ref{ThPoisResd} is known for smooth measures
and assume now that $\xia$ is distributed according to an integrable  density $\fp(\xia).$ Let 
$\cK_1, \cK_2,\dots \cK_r$ be any partition of the target space 
$$ \left[-\frac{1}{\eps}, \frac{1}{\eps}\right]\times \R/(2\Z)\times \T^{d+1}.$$
Let $\cN_1(\xia, N), \dots \cN_r(\xia, N)$ be the number of points of our process inside
$\cK_1,\dots \cK_r$ respectively. We need to know that as $N\to\infty$,
$\cN_j(\cL, N)$ are asymptotically independent Poisson random variables with means
$\bc\hat\cK_j$ where $\hat\cK_j$ is the volume of $\cK_j.$
Equivalently we need to show that for each $r$ tuple $s_1, \dots s_r$
$$ \lim_{N\to \infty}\int \exp\left(i\sum_{j=1}^r s_j \cN_j(\xia, N)\right) \fp(\xia) d\lambda(\xia)
=\psi(s_1, \dots s_r) $$
where $\psi$ is characteristic function of multivariate Poisson, namely
$$ \psi(s_1,\dots s_r)=\exp\left(\bc \sum_{j=1}^r \hat\cK_j \left[e^{is_j}-1\right]
\right)$$
(the precise form of $\psi$ is not important for the argument below).
Fix $\upsilon>0.$ Take a smooth density $\bar\fp$ on $X$ such that 
\begin{equation}
\label{L1Close}
||\fp-\bar\fp||_{L^1(\lambda)}\leq \frac{\upsilon}{2}.
\end{equation}
Since we assume that Theorem \ref{ThPoisResd} holds for smooth densities, 
for large $N$ we have  
\begin{equation}
\label{CharClose}
 \left|\int \exp\left(i\sum_{j=1}^r s_j \cN_j(\xia, N)\right) \bar\fp(\xia) d\lambda(\xia)
-\psi(s_1, \dots s_r)\right|<\frac{\upsilon}{2}.  
\end{equation}
Combining \eqref{L1Close} with \eqref{CharClose} we obtain 
$$  \left|\int \exp\left(i\sum_{j=1}^r s_j \cN_j(\xia, N)\right) \fp(\xia) d\mu(\xia)
-\psi(s_1, \dots s_r)\right|<\upsilon. $$
Since $\upsilon$ as well as the partition $\cK_1,\dots \cK_r$ are arbitrary we obtain 
that Theorem \ref{ThPLat} holds for absolutely continuous initial distributions.

Therefore we assume henceforth  that the initial distribution 
of $\xia$ has smooth density on $X.$
\end{remark}

\begin{remark} 
\label{RkPoisResd} 
Observe that it does not change anything in the result nor in the proof to take in the last coordinate of the process $\{\langle k,x \rangle+\varphi_{k,\a,N}\}$ instead of $\{\langle k,x \rangle\}$ since the phase 
$\varphi_{k,\a,N}=\pi(N-1)\langle k,\a \rangle/2$ is independent of the variable $x$. It is with the phase $\varphi_{k,\a,N}$ that Theorem \ref{ThPoisResd} is used to prove Theorem \ref{dimd}.
\end{remark}

Here and below when we consider the Poisson process on a real line times a torus the intensity is always computed with
respect to the Lebesgue measure on the line times the Haar measure on the torus. This normalization is convenient since in Lemma \ref{LmPT}(c)
we need to have a probability measure on the second factor.

Sections \ref{sec.reduction}--\ref{SSDenom} are dedicated to the proof of Theorem \ref{ThPoisResd}.

Note that by standard properties of weak convergence the result remains valid in the limit $\breps=0.$ That is, we get the following result which is of independent interest.
\begin{coro}Let $(\MA, \a)$ have absolutely continuous distribution on \newline
$SL_{d}(\reals)\times \T^d.$ Let

{
\begin{multline}
\label{DefHZ}
 \hat Z(\xia,N) := \Big\{k\in \Z^d : \left|\prod_{i=1}^d \bar{k}_i\right|<N,   
\quad \forall i=1,\ldots,d, \quad  |\bar{k}_i|\geq 1, \\
\bar k_1>0 \text{ and } \exists m \in \Z  {\rm \ s. \ t. \ }  
k_1 \wedge \ldots \wedge k_d  \wedge m =1  {\rm \ and \ }  \norm{\langle k,\a \rangle}= |\langle k,\a \rangle+m| \Big\} . \end{multline}} 
Then as 
$N\to\infty$ the point process 
$$\left\{(\ln N)^d  \left(\prod_i \bar{k}_i \right)  \norm{\langle k,\a \rangle}
\right\}_{k \in { \hat Z}(\xia, N)} $$
converges to a Poisson process on $\reals$
with intensity 
$2^{d-1} \bc_1/d!.$ 
\end{coro}

{ Note that 
$$ Z(\xia, N)=\Big \{k\in \hat{Z}(\xia, N): \left|(\ln N)^d  \left(\prod_i \bar{k}_i\right)  \norm{\langle k,\a \rangle}\right|\leq \frac{1}{\breps} \Big\} $$
(compare \eqref{DefHZ} with \eqref{defZ}, \eqref{defW}).}

\begin{proof}
By definition of the weak convergence it is sufficient to prove that for each $\breps$ the point process restricted by the 
condition
$$ \left|(\ln N)^d  \left(\prod_i \bar{k}_i\right)  \norm{\langle k,\a \rangle}\right|\leq \frac{1}{\breps} $$
converges to the Poisson process on 
$\left[-\frac{1}{\breps}, \frac{1}{\breps}\right]\times (\R / (2\Z))
\times \T^{d}.$ 
Thus the corollary follows from Theorem \ref{ThPoisResd} and the invariance of Poisson processes
under projection (Lemma \ref{LmPT}(b)).
\end{proof}

{ 
 \subsection{Proof that Theorem \ref{ThPoisResd} implies Theorem \ref{dimd}} Fix $z \in \R$ and $\eta>0$. We want to show that for $N$ sufficiently large we have 
\begin{equation} \label{cauchy0} 
\left|\lambda \{\xia \in X   :  \frac{ D(\xia,N)}{(\ln N)^d} \leq z \} 
 - \fC(\brho z)\right|<\eta. \end{equation}
We first use the approximation of $ \frac{ D(\xia,N)}{(\ln N)^d}$ by $ { \bar{\holako}(\xia,N)}$ given by  Proposition \ref{prop.reduction}. Observe that unlike $ { D(\xia,N)}$, the definition of $\bar{\holako}{(\xia,N)}$ depends implicitly on some $\bar\eps>0$. Having fixed $z\in \R$, we know from Proposition \ref{prop.reduction}, that if we fix $\breps>0$ sufficiently small, and then consider $N$ sufficiently large the following holds
{
$$
\lambda \left\{\xia \in X   :  {\bar{\holako}(\xia,N)} \leq z-\frac{\eta}{4} \right\}-\frac{\eta}{2}$$
$$<
\lambda \left\{\xia \in X   :  \frac{ D(\xia,N)}{(\ln N)^d} \leq z \right\}< $$
$$ \lambda \left\{\xia \in X   :  {\bar{\holako}(\xia,N)} \leq z+\frac{\eta}{4} \right\}+\frac{\eta}{2}. $$}
Hence it suffices to prove that for $\breps$ sufficiently small
and  $N$ sufficiently large 
we have 
\begin{equation} \label{cauchy22} \left|\lambda \{\xia \in X   :  
{\bar{\holako}(\xia,N)} \leq z \} 
 - \fC(\brho z)\right|<\frac{\eta}{2} .
 \end{equation}



We now want to use the  representation of the Cauchy distribution in terms of a Poisson process recalled in \S \ref{SSPP}, and the Poisson limit of Theorem  \ref{ThPoisResd}. Recall that $\bar{\holako}(\xia,N) = \sum_{k \in Z} {\Gamma_{k}(\xia,N)}/{\Omega_{k}(\xia,N)}$ 
where ${\Omega_{k}(\xia,N)}$ and  ${\Gamma_{k}(\xia,N)}$ are given by \eqref{DefTheta} and \eqref{GammaPhi} respectively. 

First of all, Theorem \ref{ThPoisResd} asserts that the point process $\{\Omega_{k}(\xia,N)\}_{k\in Z(\xia,N)}$, that converges to  a Poisson process on $[- \frac{1}{\breps}, \frac{1}{\breps}]$  
with constant intensity 
$c=2^{d-1} \bc_1/d!$. Secondly, Theorem \ref{ThPoisResd}, with 
Remark \ref{RkPoisResd}, and Lemma 
{\color{blue} \ref{LmPT} (c)}
  tell us that $\{\Gamma_{k}(\xia,N)\}_{k\in Z(\xia,N)}$ behave asymptotically, as $N \to \infty$, like iid symmetric variables with compact support, that are independent of  
  $\{\Omega_{k}(\xia,N)\}_{k\in Z(\xia,N)}$. 
 
 Hence, the limiting distribution, as $N\to \infty$, of    $ \sum_{k \in Z} {\Gamma_{k}(\xia,N)}/{\Omega_{k}(\xia,N)}$ is approached by that of 
 $ \sum \Theta_k/\overline{\Gamma}_k$ where  $\{\Theta_k\}$ is a Poisson process on $[- \frac{1}{\breps}, \frac{1}{\breps}]$ with constant intensity $c$ and  $\{\overline{\Gamma}_k\}$ consisting of symmetric variables with compact support, that are independent of  $\{\Theta_k\}$. 

Hence, if $\breps$ was chosen sufficiently small, then $N$ sufficiently large, it follows from  Lemma \ref{LmPoissonCauchy} (b) that \eqref{cauchy22} holds with $ \brho=\pi c \E( |\overline{\Gamma}_k|),$ where the expectation is taken with respect to the distribution of the iid variables $\{\overline{\Gamma}_k\}$.  

We finish by explicitly computing this limit expectation. 
Using the definition of $\Gamma_k$ in \eqref{GammaPhi}
we get
\begin{equation}
\label{ExpGamma}
\E(|\overline{\Gamma}_k|)=\frac{2}{\pi^{d+1}} \int\dots\int 
\left| \phi(\eta_1, \dots \eta_{d+1})\right| d\eta_1\dots 
d \eta_{d+2}
\end{equation}
where $\phi$ is given by
\eqref{DefGamma}.

The formula  \eqref{DefBRho} for $\brho$ follows from \eqref{C1C2},  \eqref{MasterIntensity}, and
\eqref{ExpGamma}.
   Theorem \ref{dimd} is thus proved.  \carre }

 \section{  Reduction to dynamics on the space of lattices} \label{sec.reduction} 
 \subsection{  Notation.}
The goal of this section is to reduce the proof of Theorem \ref{ThPoisResd} to Theorem  \ref{ThPLat} 
which is a Poisson limit theorem for the diagonal action on  
 the space of lattices. {For this we rely on Dani's correspondance principle relating  Diophantine approximation problems to visits to a cusp  by orbits of the Cartan action on   the space of $d+1$ dimensional 
unimodular lattices $\cM_{d+1}=SL_{d+1}(\reals)/SL_{d+1}(\integers)$.    For more details on the specific approach adopted here we refer the reader to the surveys  \cite{marklof.symp} and  \cite{DFsurvey} and references therein.
 }

For simplicity  we drop the subindex and refer to $\cM_{d+1}$ as $\cM$. We recall that $\mu$ denotes the Haar measure on $\cM$. 

In all the sequel we associate to $N$ the integer  $M=[\ln N].$

Define 
\begin{align} 
\label{DefbPi2} 
\Pi_M &=\{\bt\in\N^{d}: \sum_{j=1}^d \bt_j\leq M{\bpurple -d}\}. \end{align}


Given $\breps>0$ let
$$I=(1,e], \quad J=[-e,-1)\cup (1, e], \quad K=\left[-\frac{1}{\breps}, \frac{1}{\breps}\right].$$

For $\xia=(\a, x, \MA)\in X$, we define 
\begin{equation}\label{def.lambda} 
\Lambda(\xia)=\left(\begin{array}{rrr}
\MA  &       0 \cr
\alpha & 1 
\end{array}  \right).
\end{equation}

We recall that in all the sequel, it is assumed that the distribution of $\xia \in X$ has smooth density 
with respect to the Lebesgue measure.

Consider the Cartan subgroup
\begin{equation}
\label{Cartan}
g^{\bt}=\textrm{diag}\left(e^{-\bt_1}, \dots e^{-\bt_d}, e^{\sum_{j=1}^d \bt_j}\right), \quad
\bt\in \R^d.
\end{equation}


 \subsection{  Reduction to a Poisson Limit Theorem for the Cartan action.}
{Let $\cL$ be a unimodular lattice in $\R^{d+1}.$ We decompose elements of $\cL$ as
$$ \vv=(x(\vv),  z(\vv)) \text{ where } x\in \R^{d}, \; z\in \R. $$
Let $\displaystyle \Pi(\vv)=\left(\prod_{j=1}^d x_j\right) z.$
Define 
\begin{equation}
\label{DefKeySet}
\cD=\{(x, z): x_1\in I, \; x_j\in J\text{ for }j=2,\dots, d \text{ and } M^d \Pi((x,z))\in K \}.
\end{equation}
Let $\Phi: \cM\to\reals$ be the Siegel transform of $\one_\cD$, 
as defined in \S \ref{SSRI}, 
that is 
\begin{equation}
\label{DefPhi}
 \Phi(\cL)=\sum_{\vv\in \cL\text{ prime}} \one_\cD(\vv)
\end{equation} 

The bulk of the derivation of the distribution of the small divisors of  Theorem \ref{ThPoisResd} from a Poisson limit theorem for the diagonal action on  
 the space of lattices is encapsulated in the following simple observation. Recall the definition of $Z(\xia,N)$ from \eqref{defZ}.

\medskip 

\noindent{\bg \sc Claim.  }  {\it  For $\xia$ as in \eqref{phileq1}, the following are equivalent for $\bt \in \Pi_M$: \begin{itemize}
 \item[(i)]  $\Phi(g^{\bt} \Lambda(\xia))=1$
 \item[(ii)]   There exists a unique $(k,m)   \in Z(\xia,N)$ such that 
$\displaystyle e^{\bt_{j}} <{ \left|\brk_j \right|}\leq e^{\bt_j+1}$ for $j=1, \dots,  d.$
\end{itemize}} 

\begin{proof} From the definition of $\Phi$ we have that 
\begin{align*} \Phi(g^{\bt} \Lambda(\xia))&=  \#\{(k,m) \in \Z^d\times \Z, k_1 \wedge \ldots \wedge k_d  \wedge m =1 :  g^{\bt} \Lambda(\xia) (k,m) \in \cD \}.\end{align*} 
Recall the notation \eqref{sidedual}: $\brk_i=a_{i,1} k_1+\dots+a_{i,d}k_d$. Then observe that the definition \eqref{def.lambda}  of $ \Lambda(\xia)$ implies that 
$$g^{\bt} \Lambda(\xia) (k,m)=\left(e^{-\bt_1}\bar{k}_1,\ldots,e^{-\bt_d}\bar{k}_d,e^{\sum_{j=1}^d \bt_j} (\langle k, \a \rangle+m)\right).$$
Hence, from the definitions of $\cD$ in \eqref{DefKeySet}, of $\Pi_M$ in \eqref{DefbPi2} and the definition of $Z(\xia,N)$ in \eqref{defZ}, we conclude that 
\begin{align*} \Phi(g^{\bt} \Lambda(\xia))&=  \#\{(k,m) \in Z(\xia,N): e^{\bt_{j}} <{ \left|\brk_j \right|}\leq e^{\bt_j+1}, j=1, \dots,  d\}.\end{align*} 
 We have thus proved the equivalence between {\bg (i)} and {\bg (ii)} of the claim. \bg 
\end{proof} 
\black }

Define also an $\reals\times (\R/2\Z)$ valued function on $\cM\times \reals$ 
\begin{equation}  \label{DefPsi} \Psi(\cL,b)=(\Psi_1(\cL), \Psi_2(\cL,b))= 
\sum_{\vv\in \cL\text{ prime}} \one_\cD(\vv)
(M^d \Pi (\vv), b z(\vv)  \text{ mod } 2). \end{equation}

Given $\bar \eps$ and $N$, suppose that $\xia \in X$ is such 
\begin{equation} \label{phileq1}  \forall \bt \in \Pi_M, \quad   
 \Phi(g^{\bt} \Lambda(\xia))\leq1 \end{equation}

Note that if (i) or (ii) of the Claim holds then 
$$\left((\ln N)^d  \left(\prod_i \bar{k}_i\right)   \norm{\langle k,\a \rangle}, N\langle k,\a \rangle \text{ mod } 2\right)
     = \Psi\left(g^{\bt} \Lambda(\xia), N e^{-\sum_{j=1}^d \bt_j}\right). $$

Thus, for $\xia$ satisfying \eqref{phileq1}, we have that the sequence 
$$ \left\{ \left((\ln N)^d  \left(\prod_i \bar{k}_i  \right) \norm{\langle k,\a \rangle}, N\langle k,\a \rangle \text{ mod } 2\right)
\right\}_{k\in Z(\xia,N)}$$
is exactly 
$$\left\{\Psi\left(g^{\bt} \Lambda(\xia), N e^{-\sum_{j=1}^d \bt_j}\right) 
\right\}_{\bt \in \Pi_M, \Phi(g^{\bt} \Lambda(\xia))=1}.$$ 

Hence, to show that the distribution of 
$$ \left\{ \left((\ln N)^d \left( \prod_i \bar{k}_i \right) \norm{\langle k,\a \rangle}, N\langle k,\a \rangle \text{ mod } 2\right)
\right\}_{k\in Z(\xia,N)}$$
converges as $N\to \infty$  to that of a Poisson process on 
$ [-\frac{1}{\breps},\frac{1}{\breps}] \times \R/(2\Z)$ 
with intensity  $2 \bc_1$ it is sufficient to prove: 

(a) that the set of $\xia$ that do not satisfy \eqref{phileq1} is small; 

(b) that the process 
$\left\{\Psi\left(g^{\bt} \Lambda(\xia), N e^{-\sum_{j=1}^d \bt_j}\right) 
\right\}_{\bt \in \Pi_M, \Phi(g^{\bt} \Lambda(\xia))=1}$
converges in probability to a Poisson distribution.  

This is the content of the following Theorem \ref{ThPLatLambda}.  
 
\begin{theo}
\label{ThPLatLambda}  Assume that the distribution of $\xia \in X$ has smooth density 
with respect to the Lebesgue measure.
 Let $\Lambda$ be the matrix $\Lambda(\xia)$ as defined in (\ref{def.lambda}).
Then, for any $\bar{\eps}>0$, we have 

(a) For any $\bt \in \bPi$, $\displaystyle{  \lambda(\Phi(g^\bt\Lambda)>1)=\cO(M^{-2d}).}$

(b)  $\left\{\left(\Psi\left(g^{\bt} \Lambda(\xia), N e^{-\sum_{j=1}^d \bt_j}\right), \frac{\bt}{M}\right) 
\right\}_{\bt \in \Pi_M, \Phi(g^{\bt} \Lambda(\xia))=1}$
converges as $N\to\infty$ to the Poisson process on $[-\frac{1}{\breps},\frac{1}{\breps}]  \times \R/(2\Z) \times \cP$ with
intensity $ \bc_1.$



\end{theo}

The notation $X=\cO(M^{-2d})$ means that $|X|\leq CM^{-2d}$ where 
$C$ may depend on other variables (such as $\bar \eps$) but not on $M$. 
 
In order to get the full Poisson limit in Theorem \ref{ThPoisResd} we will also need an additional effort to prove the independence and uniform distribution of the rest of the variables namely of
$$\left\{\{\brk_1 u_1\}, \dots , \{\brk_d u_d\}, \{\langle k,x \rangle\}\right\}_{k\in Z(\xia, N)}.$$
This issue is addressed below.

\begin{define} Let $L>0$. 
Consider a sequence $\{t^{(1)}, \ldots,t^{(s)}\}$ of points in $\Pi_M$ where 
$t^{(j)}=(t_1^{(j)}, \dots t_d^{(j)}).$
We say that this sequence 
is {\em $L$-split} if for any pair $i,j$ we have 
$$ |t_p^{(i)}-t_p^{(j)}|\geq L   \text{ for each } p \text{ and } 
|\max_p (t_p^{(i)})-\max_p(t_p^{(j)}) |\geq L , $$
and for any $i$ we have $\displaystyle \min_p(t_p^{(i)})>L.$
\end{define}

\begin{prop}
\label{ThLacun} 
Let  $R\in \R$ and $s\in \N$ be fixed.
Let $k^{(1)}(N) \dots,  k^{(s)}(N),$ $k^{(j)}(N)\in \R^d$,  be such that 
$$t^{(j)}=([\ln|\brk^{(j)_1}|],\dots [\ln |\brk^{(j)}_d|])$$ 
is $\sqrt{M}$-split ($[\cdot]$ denotes the integer part).

Suppose that $(u_1, \dots u_d, x_1\dots x_d)$ are distributed according
to a density $\rho_N$ such that 
\begin{equation}
\label{SmoothPhase}
||\rho_N||_{C^1}\leq R.
\end{equation}
Then the distribution of the numbers
$$ \{\brk^{(1)}_1 u_1\}, \dots, \{\brk^{(1)}_d u_d \},  \{\langle k^{(1)},x \rangle\}\dots, $$
$$ \{\brk^{(s)}_1 u_1\}, \dots, \{\brk^{(s)}_d u_d \},  \{\langle k^{(s)},x \rangle\}$$
converges to the uniform distribution on $\T^{(d+1)s}$ and the convergence is uniform with
respect to $N,$  $(\MA, \a),$  the choices of $s$ vectors satisfying  the splitness condition, 
and $\rho_N$ satisfying \eqref{SmoothPhase}.
\end{prop}

 \subsection{  Proof of Proposition \ref{ThLacun}.}
By Weyl equidistribution criterion 
we need to show that 
if $f_j: \T^{d+1}\to \mathbb{C}$ are exponentials
$$f_j(\theta_1, \dots, \theta_{d+1})=\exp\left(2\pi i \sum_{p=1}^{d+1} m_{jp} \theta_p\right) $$
and not all $m_{jk}$ are equal to zero then
\begin{equation}  
\label{SplitExp}
\int_{\T^{2d}} \prod_{j=1}^{s} 
f_j\left(\brk^{(j)}_1 u_1, \dots \brk^{(j)}_d u_d, \langle k^{(j)},x \rangle \right)  \rho_N(u, x) du dx \to 0 . 
\end{equation}
uniformly in the parameters involved. 
Suppose there exists some $p\leq d$ such that
 not all $m_{jp}$ are equal to zero. 
Then the coefficient in front of $u_p$ in the above product 
is large since, due to splitness, it is dominated by the contribution of the largest of 
$\brk^{(j)}_p$ for which $m_{jp}$ is non zero. In this case we 
show that the integral \eqref{SplitExp} is small by integrating by parts with respect to $u_p.$
Next suppose that all $m_{jp}$ with $j\leq s$ and $p\leq d$ are zero.
Let $\brj$ be such that
$\max_p (t_p^{(\brj)})=:t_{\bar{p}}^{(\brj)}$ 
is the largest among those indices for which $m_{j(d+1)}\neq 0.$ 
Note that 
$k_{\bar p}^{(\brj)}$ is of order $\exp(t_{\bar p}^{(\brj)}).$ 
Then,
 due to the splitness condition,  $\brk^{(\brj)}_{\bar p}$ 
dominates the coefficient in front of $x_{\bar p}$ and so we conclude that  
\eqref{SplitExp} is small by integrating by parts with respect to $x_{\bar p}.$
\carre

 \subsection{  Proof that Theorem \ref{ThPLatLambda} implies Theorem \ref{ThPoisResd}}
\hskip7cm

As demonstrated earlier, 
parts (a) and (b) of   Theorem  \ref{ThPLatLambda}  imply that 
$$ \left\{ \left((\ln N)^d  \left(\prod_i \bar{k}_i  \right) \norm{\langle k,\a \rangle}, N\langle k,\a \rangle \text{ mod } 2\right)
\right\}_{k\in Z(\xia,N)}$$
converges as $N\to \infty$  to a Poisson process on 
$ [-\frac{1}{\breps},\frac{1}{\breps}] \times   \R/(2\Z)$ with intensity  
$ \bc_1$. Next,  it follows from part (b) of Theorem \ref{ThPLatLambda}
that if $t^{(1)}, t^{(2)},\ldots \in \Pi$ are the points such that $\Phi(g_t \Lambda)=1$, listed in any order; then for any $s \in \N$, we have that  
\begin{equation}  \lambda(\{t^{(1)},t^{(2)}, \ldots \}  \text{ is }   \sqrt{M}-\text{split})  \to 1   \text{ as }N\to\infty. \label{partc}  \end{equation}
Indeed, given $\breps,$
$\tilde\eps$ we can 
choose $\delta$ such that the probability that
the Poisson process on 
$[-\frac{1}{\breps}, \frac{1}{\breps}]\times \R/(2\Z) \times \cP$ with intensity $\bc_1$ has two points within distance
$\delta$ from each other in projection on the last coordinate is less than $\tilde\eps.$  
Since $M^{-1/2}<\delta$ for large $M$ \eqref{partc} follows.
Therefore outside a set of small measure of $\xia \in X$,  the set $Z(\xia, N)$
satisfies the hypothesis of Proposition \ref{ThLacun}. 

Thus  $\left\{\left(\{\brk_1 u_1\}, \dots \{\brk_d u_d\}, \{\langle k,x \rangle\}\right)\right\}_{k\in Z(\xia, N)}$ 
converge to uniformly distributed iid's on $\T^{d+1}$ independent of 
$$ \left\{ \left((\ln N)^d  \left(\prod_i \bar{k}_i\right)  \norm{\langle k,\a \rangle}, N\langle k,\a \rangle \text{ mod } 2\right)
\right\}_{k\in Z(\xia,N)}$$
Lemma \ref{LmPT} hence yields the full Poisson limit of Theorem \ref{ThPoisResd}. \carre

 \subsection{  Modifying the initial distribution.}

Before we close this section we make
 a last observation that allows us to complete the reduction of our problem to a clear cut ergodic theory problem on the space of lattices, namely the following.
\begin{theo}
\label{ThPLat}
Assume that 
$\cL \in \cM$ is distributed according to a probability measure $\tmu$ that has a smooth bounded density
with respect to the Haar measure on $\cM$. Then 

(a) For any $\bt \in \bPi$, $\displaystyle{  \tmu(\Phi(g^\bt\cL)>1)=\cO(M^{-2d}).}$

(b)  $\left\{\left( \Psi\left(g^\bt \cL, N e^{-\sum_{j=1}^d \bt_j} \right),\frac{\bt}{M} \right)\right\}_{\Phi(g^\bt \cL)=1,\; \bt \in \Pi_M}$ 
converges in probability, as $N\to\infty$, to a Poisson process  on $ [-\frac{1}{\breps},\frac{1}{\breps}] \times   \R/(2\Z) \times \cP$ with
intensity $\bc=2^{d-1} \bc_1/d!.$ 

\end{theo}

\begin{remark} As in Remark \ref{remark.smooth}, it is easy to see that the statement with a smooth bounded density of Theorem \ref{ThPLat}, implies the same results with a merely bounded density. Moreover,
part 9b) of Theorem \ref{ThPLat} holds assuming that the initial distribution of $\cL$ is absolutely continuous
with respect to the Haar measure.
\end{remark}

\noindent {\it Proof that Theorem  \ref{ThPLat} implies Theorem  \ref{ThPLatLambda}.}

Let $\eta>0$ and define for an interval $A=[a,b]$ the intervals 
$$A^+=[a(1-\eta),b(1+\eta)]\quad\text{and}\quad A^-=[a(1+\eta),b(1-\eta)].$$ Fix an interval $\bar K \subset K$.
Let 
$\bar{\Phi}^\pm$ be defined as in (\ref{DefPhi})  with the intervals $I^\pm,J^\pm,\bar{K}^\pm$ instead of $I,J,{K}$. 
Next, given $\Lambda=\Lambda(\xia)$ for some $\xia \in X$, define 
$$\tLambda=\left(\begin{array}{cccc}
(1+\sigma_1) & \dots & 0 & 0 \cr
\dots & \dots & \dots & \dots \cr
0 &\dots & (1+\sigma_d) & 0 \cr
0 & \dots & 0 & \left(\prod_{j=1}^d (1+\sigma_j)\right)^{-1} 
\end{array}  \right) 
\left(\begin{array}{cccc}
1 & \dots &0 & c_1 \cr
\dots & \dots & \dots & \dots \cr
0 & \dots &  1 & c_d \cr
0 & \dots & 0 & 1 
\end{array}  \right) 
\Lambda $$
where $\sigma_1,\dots, \sigma_d,$ $c_1, \dots, c_2$ are distributed according to any smooth density on $[-\eta^2,\eta^2]^{2d}$. This guarantees that when $\xia$ is distributed according to a smooth density on $X$, the lattice $\tLambda$ has a smooth bounded distribution with respect to the Haar measure.
Thus, the implication of Theorem \ref{ThPLatLambda} from  Theorem \ref{ThPLat} 
stems from the straightforward observation that  if $M$ is sufficiently large, 
then for any $n \in \N$ it holds that 
\begin{multline*} \bar{\Phi}^-(g^{(\bt_1+\ln(1+\sigma_1),\dots ,\bt_d+\ln(1+\sigma_d))} \tLambda) 
\geq n \implies \Phi(g^\bt \Lambda)  \geq n  \\ \implies \bar{\Phi}^+(g^{(\bt_1+\ln(1+\sigma_1),\dots, \bt_d\ln(1+\sigma_d))} \tLambda) \geq n. \;\text{ \carre} \end{multline*}

The proof of Theorem \ref{ThPLat} occupies Sections \ref{sec.poisson.abstract}--\ref{SSDenom}.

 \section{  Poisson limit theorem for almost independent rare events} \label{sec.poisson.abstract}

To prove Theorem \ref{ThPLat} we will start with an abstract result that establishes a Poisson limit theorem for $M^d$ variables that behave similarly to iid variables with expectation of order $1/M^d$.  
The variables to which the abstract Poisson limit theorem must be applied to imply Theorem \ref{ThPLat} will be defined precisely in Section \ref{SSDenom} (see  \eqref{xi}--\eqref{zeta}). Essentially, we will be considering a  counting variable $\xi_\bt$ that takes integer values and that corresponds to $\Phi(g^\bt \cL)$ and two related variables $\nu_\bt$ and $\zeta_\bt$ 
that give the value of $\Psi_1$ and $\Psi_2$ in 
$\Psi\left(g^\bt \cL, N e^{-\sum_j \bt_j} \right)$ when $\xi_\bt=1$.

 \subsection{  Setting and results} \label{secp}
$ \ $ 
\medskip 

{\bg $\bullet$} Let $(\Omega, \Prob)$ be a probability space.
We denote by $\E$ the expectation with respect to $\Prob.$

{\bg $\bullet$} Let $\bPi$ be a bounded domain in $\R^d$ with piecewise smooth boundary.
For $M \in \N$ we let $$\Pi_M= \{ \bt \in \N^d : \bt/M \in \bPi\}.$$ 

{\bg $\bullet$} We are given an inhomogeneous non-constant linear form on $\R^d$, 
$$\lambda_1(\bt)=\sigma_0+\sum_{j=1}^d \sigma_j \bt_j. $$

{\bg $\bullet$} We let $(\fX,\fm)$ and  $(\tfX,\tfm)$ 
be two probability spaces. 
Let $\cQ$ be a countable collection 
of finite partitions of  $\fX$
and $\tcQ$ be a countable collection 
of finite partitions of  $\tfX.$
We assume that $\cQ$ and $\tcQ$ converge to the point partitions of  $(\fX,\fm)$ and  $(\tfX,\tfm)$ respectively.

{\bg $\bullet$}  For every $M$ we consider a sequence $\{\xi_\bt^M\}_{\bt \in \Pi_M}$ of  
random variables taking values in non-negative integers 
and a sequence $\{\nu_\bt^M\}_{\bt\in \Pi_M}$ of $\fX$ valued random variables
on $\Omega.$
 
{\bg $\bullet$} For each fixed partition
 $Q=(\Ka_1,\ldots,\Ka_P)\in \cQ$  we suppose that $\xi_\bt^M$ can be decomposed as
 $\xi_\bt^M=\sum_{p=1}^P \xi_{\bt, p}^M$ where 
 $\xi_{\bt,p}^M$ take values in non-negative integers and
on the set $\{\xi_\bt^M=1\}$, it holds that $\xi_{\bt,p}^M=\mathbbm{1}_{\nu_\bt^M\in \Ka_p}.$ 

We define 
$$\eta^M_\bt=\xi_\bt^M \mathbbm{1}_{\xi^M_\bt=1}, \quad \eta^M_{\bt,p}=
\xi^M_{\bt,p} \mathbbm{1}_{\xi^M_{\bt,p}=\xi^M_\bt=1}. $$
(Note that, in fact, $\eta^M_\bt=\mathbbm{1}_{\xi^M_\bt=1},$  and $\eta^M_{\bt,p}=\mathbbm{1}_{\xi^M_{\bt,p}=\xi^M_\bt=1}$
but we use a more complicated definition above because condition (h2)  below will ensure 
that with probabiilty close to 1 we have
$\eta^M_\bt=\xi^M_\bt,$ and $\eta^M_{\bt,p}=\xi^M_{\bt,p}.$)

{\bg $\bullet$} Since all the variables depend on $M$, we will omit sometimes the superscript or subscript $M$ and denote  $\Pi_M$ simply by $\Pi$, $\xi_\bt^M$ by $\xi_\bt$ etc. 

{\bg $\bullet$} In all this section, when we use the notation $Y=\cO(X)$ or equivalently $Y \ll X$, it corresponds to $|Y|\leq C |X|$, where the implicit constant $C$ is allowed
to depend on $Q, \tilde{Q}$ but not on $M$, $\bt$, and $\bdelta$ that will be introduced later in the section. 

{\bg $\bullet$} We assume that for every fixed $M$,  a sequence  of partitions $F_\bt$, $\bt \in \Pi$ of $(\Omega,\Prob)$ is given. 
For $\omega\in \Omega$ we denote by $F_\bt(\omega)$ the element of $F_\bt$ containing 
$\omega.$ We will denote by $\cF_\bt$ the $\sigma$-algebra generated by $F_\bt.$ 
We assume that the following hypotheses hold: there exists $R>0$ (that does not depend on $M$) and a set $E$ such that $$\Prob(E^c) = \cO(M^{-100d})$$ and

\begin{itemize}
\item[(h1)]   For any $\bt \in \Pi$, 
$$\displaystyle{\E( \xi_\bt) = \cO(M^{-d})};$$
\item[(h2)]  For any $\bt \in \Pi$, 
$$ \displaystyle{\Prob(\xi_\bt>1) =\cO(M^{-2d}) };$$ 
\item[(h3)]   For $\bt,\bt' \in \Pi$, $\bt \neq \bt'$,  
$$\Prob( \xi_\bt \geq 1, \xi_{\bt'} \geq 1) =\cO(M^{-2d});$$
\item[(h4)]   For     $ \bt, \bt' \in \Pi$ with  
$\lambda_1(\bt)\geq \lambda_1(\bt')+R\ln M$, for any $p \in [1,P]$
and for any $\omega\in E;$ 
$$\displaystyle{\text{(h4a)} \quad {\E(\xi_{\bt}|\cF_{\bt'})(\omega)} =\frac{\bbc \fm(\fX)}{M^d} +\cO\left(M^{-2d}\right)},$$
$$\displaystyle{\text{(h4b)}\quad \E(\xi_{\bt ,p}|\cF_{\bt'})(\omega) 
=\frac{\bbc \fm(\Ka_{p})}{M^d} +\cO\left(M^{-2d}\right)},$$
$$\displaystyle{\text{(h4c)}\quad \E(\eta_{\bt ,p}|\cF_{\bt'})(\omega) =
\frac{\bbc \fm(\Ka_{p})}{M^d} +\cO\left(M^{-2d}\right)};$$

\item[(h5)]   For $ \bt,\bar\bt \in \Pi$ with  
$\lambda_1(\bar\bt)>\lambda_1(\bt)+R\ln M$,  for any $p \in [1,P]$, for
any $\omega\in E$
$$\xi_{\bt,p} \text{ is  constant on } F_{\bar\bt}(\omega) ; $$

\item[(h6)]
The algebras $\{\cF_\bt\}$ have a filtration like property in the sense that 
for $ \bt,\bar\bt \in \Pi$ with  $\lambda_1(\bar\bt)>\lambda_1(\bt)+R\ln M$,  for
any $\omega\in E$
$$ F_{\bar\bt}(\omega)\subset F_\bt(\omega). $$
\end{itemize}

\begin{theo} \label{th.poisson1} Under conditions (h1)--(h6), the sequence of point processes
$$\left\{\nu_\bt^M,\frac{\bt}{M}\right\}_{\xi_\bt^M=1, \bt \in \Pi_M}$$ 
converges as $M \to \infty$ to a Poisson process  with intensity $\bbc$  on 
$$(\fX\times \bPi,  \fm\times Leb).$$
\end{theo}

Assume now 
that there is another form $\hat{\l}(\bt)=\hat\sigma_0+\sum_{j=1}^d \hat\sigma_j \bt_j$ 
such that  $\hlambda(\bt)>\lambda_1(\bt)$ on $\text{Int}(\bPi).$ {Assume that there is another domain $\tcP\supset \cP$
  such that denoting $\tPi_M=\{\bt\in \N^d: \bt/M\in \tcP\}$ we have  that 
  that (h1)--(h6) 
  above are satisfied for $t\in \tPi_M$ and that $\max_{\tcP}\l_1>\max_{\cP}\hat\l.$}

Suppose that for each $M$ we have a sequence of $\zeta_\bt^M$ of $\tfX$-valued random variables and assume that 
for any fixed element $\tQ=(\tKa_1,\ldots, \tKa_{J})\in \tcQ$
the following conditions are satisfied.

\begin{itemize}
\item[(h7)]   There exists a sequence $\upsilon_M\to 0$ as 
$M\to\infty,$
and $R>0$ such that
 if $\bt,\bt'\in {\tPi_M}$ satisfy $ \hat{\l}(\bt) \geq \lambda_1(\bt')+R \ln M  \geq \lambda_1(\bt)+2R \ln M$, then for any $\omega\in E$ such that $\xi_\bt(\omega)=1$
$$| \Prob( \zeta^M_{\bt} \in \tKa_j | \cF_{\bt'})(\omega)  -  \tfm(\tKa_j)| \leq \upsilon_M. $$
\item[(h8)]
For $ \bt,\bar\bt \in {\tPi_M}$ with  $\lambda_1(\bar\bt)>\hlambda(\bt)+R\ln M$,  for any $j \in [1,J]$, for
any $\omega\in E$ such that $\xi_\bt(\omega)=1$ 
$$\mathbbm{1}_{\zeta_{\bt}\in \tKa_j} \text{ is  constant on } F_{\bar\bt}(\omega). $$
\end{itemize}

Then we have the following strengthening of Theorem \ref{th.poisson1}.

\begin{theo} \label{th.poisson2}
Under hypothesis (h1)--(h8), the sequence of point processes 
$$\left\{\nu_\bt^M,\zeta_\bt^M,\frac{\bt}{M}\right\}_{\xi_\bt^M=1, \bt \in \Pi_M}$$ 
converges as $M \to \infty$ to a Poisson process with intensity $\bbc$ on 
$$(\fX\times \tfX\times \bPi, \fm\times \tfm\times Leb).$$
\end{theo}

 \subsection{  Proof of Theorem \ref{th.poisson1}} \label{sec.poisson1} $ \ $

 Divide $\cP$ into arbitrarily small subsets 
$\bPi_1, \bPi_2,\dots \bPi_{S}$ of positive volume. 
Fix from the sequence $\cQ$ a partition $Q=(\Ka_1,\ldots,\Ka_P)$  
that is arbitrarily close to the point partition. 

Given any double sequence $l_{p,s} \in \N, (p,s) \in [1,\ldots,P] \times [1,\ldots,S]$, 
we define the event 
\begin{multline}
\label{DefcA}   
\cA = \left\{ \forall (p,s) \in [1,\ldots,P] \times [1,\ldots,S],  \right. \\ 
\text{there are exactly }l_{p,s}\text{ points } \bt \in \Pi \text{ satisfying }   \\ \left.
\frac{\bt}{M}\in \bPi_s \text{ and }  \xi_\bt=\xi_{\bt,p}=1 \right\}.
\end{multline}
For a set $B \subset \R^d$, we denote by $\hat{B}$ the volume of $B$.

To prove Theorem \ref{th.poisson1} it suffices to see that 
\begin{equation} \label{eq.poisson} \lim_{M \to \infty} \mathbb P(\cA) = \prod_{p,s} \left[ \frac{(\bbc \fm(\Ka_p)(\hat{\bPi}_s))^{l_{p,s}}}{l_{p,s}!}  \exp\left(-{ \bbc \fm(\Ka_p)(\hat{\bPi}_s)}\right)\right].\end{equation}

This section is devoted to the proof of \eqref{eq.poisson}. Fix  an arbitrarily | small number $\bar{\delta}>0.$

In the sequel we say that a function $h$ which depends on $M, \brdelta$, and maybe some
other variables such as $\bt\in \Pi_M$, is 
$o_{\bdelta}(1)$ if the following holds : given $ \nu>0$, we can find $\hat\delta(\nu)$ such that if
$\bdelta<\hat\delta(\nu)$, then there exists $\bar M(\brdelta)$ such that
 for $M \geq \bar M(\bdelta)$,  
we have  $|h| \leq \nu$ (uniformly in all the additional parameters). 
 $O_{\brdelta}(1)$ has  a similar meaning.
%

Partition $\Pi$ into cubes $C_1, \ldots, C_H$ of side size $\bar \delta M$ with one of the faces parallel to $\Ker(\lambda_1).$

\begin{define}   We say that a $k$-tuple $\{(S_l,i_l)\}$,  
$S_1, S_2,\dots S_k 
\subset \{C_1, \ldots, C_H \}$,  $i_1,\ldots,i_k \in  \{1\dots P \}$, 
realizes the event $\cA$  
if the event $\cA$ is realized,  and if the $l_{p,s}$ 
points $\bt^{(1)},\ldots,\bt^{(k)}$ of $\cA$ are such that $\bt^{(l)} \in S_l$ 
and $ \xi_{\bt^{(l)}}=\xi_{\bt^{(l)},i_l}=1$. 
\end{define}

\begin{define} \label{define.generic} We call a collection of cubes  $\bdelta$-generic if the  images of any two of them under $\lambda_1$ are 
distant by more than $3\bdelta M$. We say that a $k$-tuple $\{(S_l,i_l)\}$ is  $\bdelta$-generic if the 
cubes $S_l$ are  $\bdelta$-generic.
\end{define}

To obtain \eqref{eq.poisson}, we shall need the following.

\begin{prop} \label{prop.generic} 
$$ \mathbb P(\cA) =  \mathbb 
P(\cA \text{ is realized by a $\bdelta$-generic collection of cubes}) +   o_{\bdelta}(1).$$
\end{prop}

\begin{prop}  \label{prop.generic2}  Given a $\bdelta$-generic $k$-tuple $\{(S_l,i_l)\}$ 
we have that 
\begin{multline*}  \mathbb P(\cA \text{ is realized by } \{(S_l,i_l)\}) = \bbc^k \brdelta^{dk} 
\left(\prod_{l=1}^k \fm(\Ka_{i_l}) \right) \exp(- \bbc \fm(\fX) \hat{\cP}) (1+o_{{\bdelta}}(1)). \end{multline*}
\end{prop}

\begin{proof}[Proof of  \eqref{eq.poisson}]
By Proposition \ref{prop.generic} we can restrict to the contribution of generic collections of squares. 

Now, there are $n_s= \frac{(\hat{\bPi}_s)}{\brdelta^d}(1+\cO(\bdelta))$ cubes 
in $M \bPi_s$. The number of possible generic choices of $k$-tuples $\{(S_l,i_l)\}$ 
that realize $\cA$ is thus 
$$ \prod_{p,s} \left(\begin{array}{l} n_s  \cr l_{p,s} \end{array}\right) (1+o_{\bdelta}(1)).$$

Applying Proposition \ref{prop.generic2} we get that 
\begin{align*} \mathbb P(\cA)=& \prod_{p,s} \left[ \left(\begin{array}{l} n_s  \cr l_{p,s} \end{array}\right) \left(\bbc \fm(\Ka_p) \brdelta^d\right)^{l_{p,s}}
    \right] \exp\left(- \bbc \fm(\fX) \hat{\cP} \right)  (1+o_{\bdelta}(1))
\\ &= \prod_{p,s} \left[ \frac{(\bbc \fm(\Ka_p)(\hat{\bPi}_s))^{l_{p,s}}}{l_{p,s}!}  \exp\left(-{ \bbc \fm(\Ka_p)(\hat{\bPi}_s)}\right)\right]  (1+o_{\bdelta}(1))\end{align*}
since $\hat{\cP}= \sum (\hat{\bPi}_s)$, and $\fm(\fX)=\sum \fm(\Ka_i)$.  
\end{proof}

We now turn to the proofs of Propositions \ref{prop.generic} and  \ref{prop.generic2}. Proposition \ref{prop.generic} is a direct consequence of Lemma \ref{LmDensity}(2) below.
\begin{lemma}  We have \label{LmDensity}
\begin{itemize} 
\item[(1)] $\Prob(\exists \bt\in \Pi_M: \xi_\bt>1)=\cO(M^{-d}).$

\item[(2)] $\Prob(\exists \bt', \bt''\in \Pi_M: \xi_{\bt'}\geq 1, \xi_{\bt''}\geq 1$ and 
$|\lambda_1(\bt')-\lambda_1(\bt'')|\leq 3\bdelta M)\leq C \bdelta.$
\end{itemize} 
\end{lemma}

\begin{proof} Parts (1) and (2)  follow by  summation 
of (h2) and (h3) respectively.  \end{proof}
\begin{proof}[Proof of Proposition \ref{prop.generic2}] By Lemma \ref{LmDensity} $(1)$ we know that up to excluding a small probability set we have that each $\xi_\bt, \bt\in \Pi_M$ 
is either equal to $1$ or $0$. Hence, 
given a $\bdelta$-generic $k$-tuple $\{(S_l,i_l)\}$ we have that 
\begin{multline*}  \mathbb P(\cA \text{ is realized by } \{(S_l,i_l)\}) =
\Prob\left( \exists  (\bt^{(1)},\ldots,\bt^{(k)}) \in S_1 \times \ldots \times S_k: \right. \\ \left. 
\eta_{\bt^{(l)},i_l}=1,    \forall l \in [1,k],  \eta_\bt=0 \text{ for } \bt\notin \{\bt^{(1)},\ldots,\bt^{(k)}\} \right)(1+o_{{\bdelta}}(1)).
\end{multline*}
We will thus finish if we prove the following 
{ \begin{lemma}\label{lemma.generic}  Given a $\bdelta$-generic $k$-tuple $\{(S_l,i_l)\}$ 
we have that 
\begin{multline*}  
\Prob\left( \exists  (\bt^{(1)},\ldots,\bt^{(k)}) \in S_1 \times \ldots \times S_k: \right. \\ \left. 
\eta_{\bt^{(l)},i_l}=1,    \forall l \in [1,k],  \eta_\bt=0 \text{ for } 
\bt\notin \{\bt^{(1)},\ldots,\bt^{(k)}\} \right)     \\ 
= \bbc^k \bdelta^{dk} \left(\prod_{l=1}^k \fm(\Ka_{i_l}) \right) \exp(- \bbc \fm(\fX) \hat{\cP}) (1+o_{{\bdelta}}(1)). \end{multline*}
\end{lemma}}

The proof of Lemma \ref{lemma.generic} will be given in the next subsection.
\end{proof}

 \subsection{  Compatible strips. Proof of Lemma \ref{lemma.generic}.}
\label{SSIndStrip}

Divide $\Pi=\Pi_M$ into strips  parallel to $\Ker\lambda_1$, of the form $\lambda_1^{-1} [s_{j-1}, s_j]$,  of width $\bar{\delta} M$. These strips have common boundaries.
To create some independence between the strips we let $\brs_j=s_{j-1}+\sqrt{M}$ and define the strips
$\Pi^j=\lambda_1^{-1} [\brs_j, s_j],$ for $j \geq 1$. We still denote their width by ${\bdelta}M$ 
(up to changing the definition of  $\bdelta$ to $\bar{\delta}-1/\sqrt{M}$).   
Let $L$ be the total number of the strips (observe that $L$ is of order $\bdelta^{-1}$). 
We can naturally assume that the partition into cubes $C_1,\ldots,C_H$ is such that every cube 
is completely included in a strip $\Pi^j$.
Note that, similarly to the proof of Proposition \ref{prop.generic},
  the probability that there is a point in a buffer zone is negligible.

\begin{define}[Type A and B strips] \label{definition.typeA} Given a $\bdelta$-generic $k$-tuple $\{(S_l,i_l)\}$
we call a strip $\Pi^j$ which contains a square $S_l$ a  {\it type A} strip. 
The remaining strips (they are a majority) are called {\it type B} strips. If $\Pi^j$ is 
of type B we say that
it is compatible  if $\eta_\bt=0$ for all $t\in\Pi^j.$ 
If $\Pi^j$ is of type A we say that
it is compatible if for $q$ such that  $S_q \subset \Pi^j$, 
there exists $\bt^{(q)} \in S_q$  such that $\eta_{\bt^{(q)},i_q}=1$  and 
$\eta_{\bar\bt}=0$ for $\bar\bt\in   \Pi^j-\{\bt^{(q)}\}.$  \end{define}

Denote $p_0=1,$ and for $j>0$
$$ p_j=\Prob(\Pi^l \text{ are compatible for } l\leq j). $$

Lemma \ref{lemma.generic} becomes thus equivalent to showing that (recall that $L$ is the total number of strips) 
$$p_L=\bbc^k \bdelta^{dk} \left(\prod_{q=1}^k \fm(\Ka_{i_q}) \right) \exp(- \bbc \fm(\fX) \hat{\cP}) (1+o_{\bdelta}(1)).$$
The latter is derived immediately by an iterative application of  \eqref{RecCompA} or 
\eqref{RecCompB} of the following lemma, according to whether a strip is of type A or B respectively. 

\begin{lemma} \label{lemma.compatible}  If $\Pi^{j+1}$ is of type A, with $S_q \in \Pi^{j+1}$, then 
\begin{equation}
\label{RecCompA}
p_{j+1}= \bbc \fm(\Ka_{i_q}) \bdelta^d p_j(1+o_{\bdelta}(1))
\end{equation}
and if $\Pi^{j+1}$ is of type B then
\begin{equation}
\label{RecCompB}
p_{j+1}=
p_j\left(1-\bbc \fm(\fX) \hat{{\bf \Pi}}^{j+1}(1+o_{\bdelta}(1))\right)
\end{equation}
with ${\bf \Pi}^{j+1}={\Pi}^{j+1}/M$.
\end{lemma}

\medskip

\noindent{\it Proof of Lemma \ref{lemma.compatible}.} We first prove \eqref{RecCompA}. So, we assume $\Pi^{j+1}$ is of type A,  with $S_q \in \Pi^{j+1}$. 
Let $\tilde{\bt}$ be such that 
\begin{equation} \label{Fj}  \min_{\bt \in \Pi^{j+1} } \lambda_1({\bt})-R \ln M 
\geq \lambda_1(\tilde{\bt})  \geq  \max_{\bt \in \Pi^j } \lambda_1({\bt})+R \ln M . \end{equation}
We define $\cF_j:= \cF_{\tilde{\bt}}.$

We will need a control on the simultaneous occurrences of 
$\eta_\bt=1$ and $\eta_{\bt'}=1$ for pairs $(\bt,\bt')$.  Denote
$$\displaystyle V_j=\sum_{\bt, \bt'\in \Pi^j, \bt\neq \bt':  |\lambda_1(\bt)-\lambda_1(\bt')|\leq 3R \ln M} \eta_\bt \eta_{\bt'}. $$
\begin{lemmata} 
\label{SLOff}
There is an $\cF_j$ measurable set $E_j \subset E$ such that $\Prob(E_j^c) \leq \frac{C\ln M}{\sqrt{M}}$ 
and for $\omega\in E_j$
\begin{equation}  \label{R3} \E(V_{j+1}|\cF_j)(\omega) \leq1/\sqrt{M} \end{equation} 
and for $\bt, \bar{\bt} \in \Pi^{j+1}$ such that  $\lambda_1(\bar{\bt})  \geq \lambda_1({\bt})+3R \ln M$ and $\omega\in E_j$ we
have
\begin{equation} \E( \eta_{\bt} \eta_{\bar\bt}|\cF_{j})(\omega) \ll   M^{-2d}. \label{re} \end{equation}
\end{lemmata}

The proof is technical and involves all the properties (h3)--(h6). 
We differ it to the Appendix \ref{appB}.
 
 As a consequence we get the following. 
 \begin{lemmata} \label{lemma69} We have for $\Pi^{j+1}$ of type A,  with 
 $S_q \subset \Pi^{j+1}$, and $\omega \in E_j$ 
\begin{equation} \label{eqprob} 
\fp:=\Prob(\Pi^{j+1} \text{ is compatible }|\cF_j)(\omega) = \bbc \fm(\Ka_{i_q}) \bdelta^d(1+o_\brdelta (1)). \end{equation}
\end{lemmata}
 \begin{proof} Recall that by definition,
\begin{equation}\label{ppp}  \fp=\Prob(\exists \bt\in S_q :  \eta_{\bt,i_q}=1, \text{ and }
\eta_{\bar\bt}=0 \text{ for }\bar\bt\neq \bt, \bar\bt \in \Pi^{j+1}|\cF_j)(\omega).\end{equation}
We omit in the rest of this proof the mention of $\omega$, that is assumed to be fixed in $E_j$. Since,  for a fixed $\bt \in S_q$, Bonferroni inequalities imply that 
\begin{multline*} \E\left(\eta_{\bt,i_q}- \sum_{\bt' \in \Pi^{j+1}, \bt'\neq \bt} 
\eta_{\bt} \eta_{\bt'}|\cF_j\right)  \leq \\   
\Prob(\eta_{\bt,i_q}=1, \text{ and }\eta_{\bt'}=0 \text{ for }\bt' \neq \bt, \bt' \in \Pi^{j+1}|\cF_j) \leq  
\E(\eta_{\bt,i_q} |\cF_j). \end{multline*}
Hence, because $\omega\in E_j$,
$$
 \left| \fp - \sum_{\bt \in S_q} 
\E (\eta_{\bt,i_q} |\cF_j  ) \right|\leq\sum_{{ \bt, \in S_q,} \, 
\bt'\neq \bt\in \Pi^{j+1}} 
\E(\eta_{\bt} \eta_{\bt'}|\cF_j) \label{eta1}    
\leq     \cO(1/\sqrt{M}) + \cO \left(\bdelta^{d+1}\right)  $$
where we used \eqref{R3} for the terms with $|\lambda_1(\bt)-\lambda_1(\bt')|\leq 3R \ln M$, 
and  \eqref{re} for 
the terms with $|\lambda_1(\bt)-\lambda_1(\bt')|> 3R \ln M$.

On the other hand (h4c) implies that for $\omega\in E$
\begin{equation*}
\label{ProbExp}
\sum_{ t\in S_q} 
\E(\eta_{t,i_q}|\cF_j)=\bbc \fm(\Ka_{i_q}) \bdelta^d (1+o_{\bdelta}(1)). 
\end{equation*}
proving \eqref{eqprob}.
\end{proof}

\medskip 
 
\noindent \begin{proof}[Proof of \eqref{RecCompA}.]
With the set $E_j$ from Sublemma \ref{SLOff} we have
$$p_{j+1}= \Prob((\Pi^1,\dots,\Pi^{j+1} \text{ are compatible})\cap E_j)  + \cO(\ln M/\sqrt{M})$$
$$= \E(\mathbbm{1}_{\Pi^1,\dots,\Pi^j \text{ are compatible}} \mathbbm{1}_{E_j} 
\Prob(\Pi^{j+1} \text{ is compatible }|\cF_j)) + \cO(\ln M/\sqrt{M})$$
where the last step relies on the fact that $E_j\subset E$ and so, by (h5), the event 
$\mathbbm{1}_{\Pi^1,\dots,\Pi^j \text{ are compatible}}$ is constant on $F_j(\omega)$ for
$\omega\in E_j.$ We finish using \eqref{eqprob} and one more time the fact that 
$\Prob(E_j^c) \leq \frac{C\ln M}{\sqrt{M}}$.
 \end{proof}

\noindent \begin{proof}[Proof of \eqref{RecCompB}.]
\eqref{RecCompB} follows from Sublemma \ref{lemma692} below 
exactly in the same way as 
 \eqref{RecCompA} follows from Sublemma \ref{lemma69}. 
 \begin{lemmata} \label{lemma692}  We have for $\Pi^{j+1}$ of type B and $\omega \in E_j$ 
\begin{equation} \label{eqprob2} 
\fp':=\Prob(\Pi^{j+1} \text{ is compatible }|\cF_j)(\omega) = 1- \bbc \fm(\fX) \hat{{\bf \Pi}}_{j+1} (1+o_{\bdelta}(1)). \end{equation}
\end{lemmata}
\begin{proof} 
We omit in the rest of this proof the mention of $\omega$, that is assumed to be fixed in $E_j$. Observe that 
$$\fp'=\Prob\left(\eta_\bt=0 \text{ for all }  \bt \in \Pi^{j+1}|\cF_j \right) 
= 1-\bar{\fp}-\hat{\fp}
$$
with 
\begin{align*}\bar{\fp}&= \Prob\left(\text{there exists a unique } \bt  \in \Pi^{j+1}  \text{ such that } \eta_\bt=1|\cF_j \right) \\ 
\hat{\fp}&= \Prob\left(\text{there exists at least a pair } (\bt,\bt')  \in \Pi^{j+1}  \text{ such that } \eta_\bt=\eta_{\bt'}=1|\cF_j \right) .
\end{align*}
 
From the proof of \eqref{RecCompA} (taking the sum over all the squares of $\Pi^{j+1}$ and $i_q \in [1,P]$ in \eqref{eqprob} and \eqref{ppp}) we have that 
$$ \bar{\fp}=  \bbc \fm(\fX) \hat{{\bf \Pi}}_{j+1} (1+o_{\bdelta}(1)).$$
On the other hand
$$ \hat{\fp}\leq  \sum_{(\bt,\bt') \in \Pi^{j+1}, \bt'\neq \bt} 
\E(\eta_{\bt} \eta_{\bt'}|\cF_j)    
\leq     \cO(1/\sqrt{M}) + \cO (\bdelta^2)$$
where we used \eqref{R3} for the terms with $|\lambda_1(\bt)-\lambda_1(\bt')|\leq 3R \ln M$, 
and  \eqref{re} for 
the terms with $|\lambda_1(\bt)-\lambda_1(\bt')|> 3R \ln M$.

Since $\hat{\fp}=o_{\bdelta}(\bar{\fp})$, \eqref{eqprob2} follows. 
\end{proof}

The proof of Lemma \ref{lemma.compatible}, and therefore  
of Lemma \ref{lemma.generic} and Theorem \ref{th.poisson1}  is now complete. \end{proof}

 \subsection{  Proof of Theorem \ref{th.poisson2}}

The proof of Theorem \ref{th.poisson2} is similar to the proof of Theorem \ref{th.poisson1}.
In addition to using (h4), (h5) and (h6) to control $\nu_\bt$ we use (h7) and (h8) to control 
$\zeta_\bt.$
Let us briefly describe the necessary modifications.

We need a stronger notion of generic collections of squares than that of Definition  \ref{define.generic}. 

\begin{define}
\label{define.generic2} We call a collection of cubes 
$S_1,\dots,S_k$ strongly $\bdelta$-generic if the distance of any two among $2k$ intervals
$$ \lambda_1(S_1), \lambda_1(S_2), \dots, \lambda_1(S_k), 
\hlambda(S_1), \hlambda(S_2), \dots, \hlambda(S_k) $$
is at least $3 \bdelta M.$ We say that a $k$-tuple 
$\{(S_l,i_l,j_l)\}$, where $i_l \in [1,P], j_l \in [1,J]$, is strongly $\bdelta$-generic 
if $S_1,\dots,S_k$ is strongly $\bdelta$-generic.
\end{define}

The same argument as in the proof of Theorem \ref{th.poisson1}
shows that the contribution of non strongly generic configurations becomes negligible as 
$\bdelta\to 0.$


For the proof of Theorem \ref{th.poisson2} 
we need the following generalization of Lemma \ref{lemma.generic}.

\begin{lemma}  Given a strongly $\bdelta$-generic  $k$-tuple $\{(S_l,i_l,j_l)\}$, we have that 
\label{LmDensity2}

\begin{multline*} \Prob\left( \exists  (\bt^{(1)},\ldots,\bt^{(k)}) \in 
S_1 \times \ldots \times S_k: \right. \\ \left. 
\eta_{\bt^{(l)},i_l}=1, \zeta_{\bt^{(l)}} \in \tKa_{j_l}   \forall l \in [1,k], \eta_\bt=0 \text{ for } \bt\notin
 \{\bt^{(1)},\ldots,\bt^{(k)}\} \right)\\
= \bbc^k \bdelta^{dk} \left(\prod_{l=1}^k \fm(\Ka_{i_l}) \tfm(\tKa_{j_l}) \right) \exp(- \bbc \fm(\fX) \hat{\cP}) (1+o_{\bdelta}(1)). \end{multline*}
\end{lemma} 

The derivation of Theorem \ref{th.poisson2} from Lemma \ref{LmDensity2} 
is exactly the same as the derivation of Theorem \ref{th.poisson1} from Lemma \ref{lemma.generic}. 

We thus focus on explaining the difference in the proof of Lemma \ref{LmDensity2} from that of  Lemma \ref{lemma.generic}. 

We need to replace Definition \ref{definition.typeA} of type A and B strips by the following definition that takes into account a third kind of strips. We keep the division $\{\Pi^j\}_{j=1,\ldots,L}$ of $\Pi=\Pi_M$ into the  parallel strips of width $\bar{\delta} M$. 
\begin{define}[Type A, B and C strips]  \label{definition.typeC} For $j\in [1,L]$, we say that $\Pi^j$ is of type A if it contains a cube 
$S_q$ from our configuration. For such a pair $(\Pi^j,S_q)$ we say that  a strip $\Pi^p$ is associated to $S_q$ if 
$\lambda_1(\Pi^p)\cap \hlambda(S_q)\neq\emptyset.$ 
Note that since   $\hlambda(\bt)>\lambda_1(\bt)$ on $\text{Int}(\bPi)$, we have that $p>j$.
We call the union of all strips associated to a given $S_q$ a type C strip. Note that type C strips are a union of  a uniformly bounded number of consecutive strips $\Pi^i$, so their width is still $\cO(\brdelta).$   The strips from the partition $\{\Pi^j\}$ which are neither type A nor belong to a type C strip 
will be called type B strips.  For convenience, after slightly decreasing the total number of strips to $L'\leq L$, we still call the collection of strips that we get $\{\Pi^j\}_{j=1,\ldots,L'}$ (the fact that the type C strips are slightly larger than the others will have no consequence on the subsequent proofs).
\end{define}

Observe that there may exists $\bar k < k$ such that for  $q \in \{\bar k+1,\ldots,k\}$ it holds that  
$\lambda_1(\Pi_M)\cap \hlambda(S_q)=\emptyset.$ If not, we just take $\bar k =k$. Thus, we have $k$ type A strips, $\bar k$ type C strips and $L-k -\bar k$ type B strips. {It will be convenient for the proof below to add $\Pi^{L'+1},\ldots,\Pi^{L'+k-\bar k}$
  strips in $\tPi_M$ where for every $j\in [1, k-\bar k]$
$$ \Pi^{L'+j}=\{\bt\in \tPi_M: \exists \bt'\in S_{\brk+j}: |\lambda_1(\bt)-\hlambda(\bt') |\leq R\ln M\}. $$
We call $\Pi^{L'+j}$
  type C strips associated to $S_{\bar k +j}$.} {The fact that these additional strips are well defined is due to our assumption that $\tcP$ is sufficiently large so that $\max_{\tcP}\l_1>\max_{\cP}\hat\l.$}

Denote the total number of strips $L=L'+k-\bar k$. 



The definition of compatibility for type A and type B strips remains the same as in Definition \ref{definition.typeA}. 
Given a cube $S_l$ of our collection, and a strip $\Pi^j$ of type $A$ such that $S_l \subset \Pi^j$ and suppose that $\Pi^j$ is compatible. This gives  $\bt^{(l)}\in S_l$ satisfying the 
conditions of Definition \ref{definition.typeA} ($\eta_{\bt^{(l),i_l}}=1$ and $\eta_\bt=0$ 
for all $\bt \in \Pi^j-\{\bt^{(l)}\}$). Assume now that $\Pi^p$, $p>j$ is a type C strip  
associated to $S_l$. Then, $\Pi^p$ is called compatible if $\zeta_{\bt^{(l)}}\in \tKa_{j_l}$
and
{$\eta_\bt=0$ for all $\bt \in \Pi^p\cap \Pi_M$.
  In particular, for type $C$ strips which are disjoint from $\Pi_M$ the only requirement that we have is that
  $\zeta_{\bt^{(l)}}\in \tKa_{j_l}$.} 

Let $ p_j=\Prob(\Pi^l \text{ are compatible for } l\leq j).$ 
Then Lemma \ref{LmDensity2} becomes thus equivalent to showing that 
$$ p_{L}= \bbc^k \bdelta^{dk} \left(\prod_{l=1}^k \fm(\Ka_{i_l}) \tfm(\tKa_{j_l}) \right) \exp(- \bbc \fm(\fX) \hat{\cP}) (1+o_{\bdelta}(1)). $$

Then, as in the proof of Lemma \ref{lemma.generic}, Lemma 
\ref{LmDensity2} follows inductively from 
\begin{lemma} \label{lemma.compatible2}  If $\Pi^{j+1}$ is of type A, with $S_q \in \Pi^{j+1}$, then 
\begin{equation}
\label{RecCompA2}
p_{j+1}= \bbc \fm(\Ka_{i_q}) \bdelta^d p_j(1+o_{\bdelta}(1));
\end{equation}
if $\Pi^{j+1}$ is of type B then
\begin{equation}
\label{RecCompB2}
p_{j+1}=
p_j\left(1-\bbc \fm(\fX) \hat{{\bf \Pi}}^{j+1}(1+o_{\bdelta}(1))\right);
\end{equation}
and  if $\Pi^{j+1}$ is a type C strip  associated  to a square $S_q$
then
\begin{equation} \label{RecCompC} p_{j+1}=
p_j \tfm(\tKa_{j_q}) (1+o_{\bdelta}(1)).
\end{equation}
\end{lemma}

\begin{proof}  We consider the same $\sigma$-algebra $\cF_j$ as in the proof of Lemma \ref{lemma.compatible}, that is 
$\cF_j:=\cF_{\tilde{\bt}}$
with $\tilde{\bt}$  such that 
$$  \min_{\bt \in \Pi^{j+1} } \lambda_1({\bt})-R \ln M 
\geq \lambda_1(\tilde{\bt})  \geq  \max_{\bt \in \Pi^j } \lambda_1({\bt})+R \ln M.$$ 

We then have from (h5) and (h8) that, regardless of the type of the strips, for $\omega \in E$
$$\mathbbm{1}_{\Pi^1,\dots,\Pi^j \text{ are compatible}} \text{ is constant on } F_j(\omega).$$
Indeed, events of the form $\mathbbm{1}_{\eta_{\bt^{(l)}}\in \Ka_{i_l}}$, with 
$\lambda_1(\bt^{(l)}) \in  \Pi^1 \cup \ldots \cup \Pi^j$ and 
$\mathbbm{1}_{\zeta_{\bt^{(l)}}\in \tKa_{j_l}}$ with  
$\hlambda(\bt^{(l)}) \in \lambda_1( \Pi^1 \cup \ldots \cup \Pi^j)$ are constant on 
$F_j(\omega)$ by (h5) and (h8).

Observe that \eqref{RecCompA2} and \eqref{RecCompB2}  as well as their proofs are then identical to   \eqref{RecCompA} and \eqref{RecCompB} of Lemma \ref{lemma.compatible}. As for \eqref{RecCompC}, it follows from
\begin{lemmata} \label{lemma69C} If $\Pi^{j+1}$ is a type C strip associated to a square $S_q$, and $\omega \in E$
 {is such that on $\omega,$ $\Pi^l$ are compatible for $l\leq j$,} we have  
\begin{equation} \label{eqprobC} 
\Prob(\Pi^{j+1} \text{ is compatible }|\cF_j)(\omega) =   \tfm(\tKa_{j_q}) (1+o_\brdelta (1)).   \end{equation}
\end{lemmata}
 \begin{proof} By definition there exists some $\bt^{(q)} \in S_q$ such that $\hlambda(\bt^{(q)}) \in \lambda_1(\Pi^{j+1})$. 
 By the definition of $\tilde{\bt}$, this implies that $\hlambda(\bt^{(q)}) > 
 \lambda_1(\tilde{\bt})+R \ln M.$ 
{ 
Next
$$ \Prob( \zeta^M_{\bt^{(q)}} \in \tKa_{j_q} | \cF_j)(\omega)-\Prob(\exists \bt\in \Pi^{j+1}: 
\eta_\bt\neq 0|\cF_j)(\omega) $$
$$ \leq 
\Prob(\Pi^{j+1} \text{ is compatible }|\cF_j)(\omega)  \leq 
\Prob( \zeta^M_{\bt^{(q)}} \in \tKa_{j_q} | \cF_j)(\omega)
 $$
Since $\omega\in E$, $\bt^{(q)}$ is constant on $\cF_j(\omega).$
Hence (h7) gives
$$\Prob( \zeta^M_{\bt^{(q)}} \in \tKa_{j_q} | \cF_j)(\omega)=
\tfm(\tKa_{j_q}) +o_{\bar \delta}(1)
$$
while (h4c) gives
$$ \Prob(\exists \bt\in \Pi^{j+1}: \eta_\bt\neq 0|\cF_j)(\omega)=\cO(\bdelta). $$
Combining the last three estimates we obtain
\eqref{eqprobC}.  }
 \end{proof}

We have completed the proof of Lemma \ref{lemma.compatible2} and thus of  Theorem \ref{th.poisson2}.
\end{proof}

 \section{  Rate of equi-distribution of unipotent flows}
\label{SSSLMix}

 \subsection{  Notation.}
Recall that $\{g^\bt\}_{\bt\in \R^d}$ is the $d$ parameter subgroup given by \eqref{Cartan}.
We let 
\begin{equation}
\label{Entropy}
T(\bt)=\sum_{j=1}^d \bt_j.
\end{equation}
In the proof of Theorem \ref{ThPLat} we will need to show 
asymptotic independence between the moments 
of time such that $\Phi(g^\bt\cL)\geq1$ as well as randomness of the values taken then by 
$\Psi(g^\bt \cL, N e^{-T(\bt)} )$. 
For this we will rely on the fact that  the action of $g^\bt$ on $\cM$ is partially hyperbolic in the sense that
$$ T\cM=E_0\oplus \sum_{q<p}^{d+1} \left(E_{qp}^+\oplus E_{qp}^- \right)$$
where
$E_0$ is tangent to the orbit of $g$ and $E_q^\pm$ are invariant one dimensional distributions.
The corresponding Lyapunov exponents are 
$\pm\lambda_{qp}$ with $q<p\leq d$ and $\pm \lambda_q$ (corresponding to $p=d+1$)
where
$$ \lambda_{qp}=\bt_q-\bt_p, \quad \lambda_q=T(\bt)+\bt_q.$$
$E_{qp}^\pm$ are tangent to foliations $W_{qp}^\pm$ which are orbit foliations for groups 
$h_{qp}^\pm$
where $h^+_{qp}(u)$ is the matrix with ones on the main diagonal,
$u$ in the $p$-th column of the $q$-th row, and zero in all other places, while
$h_{qp}^-(u)$ are transposes of $h_{qp}^+(u).$
Below we shall abbreviate $E_{1(d+1)}^+,W_{1(d+1)}^+,h_{1(d+1)}^+$ with  
$E_1,W_1,h_1$. 

We also use the notation $\mu(A) =  \int_{\cM} A(x) d\mu (x)$, and  for $g \in SL_{d+1}(\R)$, 
we denote by $A(g \cdot)$ the function whose value at $x$ is $A(gx).$

The Sobolev norm of index $\bs$ will be denoted by $||\cdot||_\bs.$ We will always assume that
$\bs$ is an integer.

 \subsection{  Mixing for smooth functions.}
Here we recall the mixing properties of homogeneous flows.
To fix the notation we discuss only the subgroup $h_1(u)$
(which is the only subgroup used in the proof of Theorem \ref{ThPLat}
given in Section \ref{SSDenom}), however similar results holds for all other unipotent subgroups
including $h_{pq}^\pm.$

 By \cite{KM}, Theorem 2.4.5  there exists $\bs$ and constants $C,\kappa>0$ such that if $A,B \in H^\bs$ 
  then 
\begin{equation}
\label{SLEM}
 |\mu(A(\cdot) B(g^\bt \cdot))-\mu(A)\mu(B)|\leq C ||A||_\bs ||B||_\bs e^{-\kappa \max|\bt_j|} 
\end{equation}

We recall that this implies that there exists $C>0$ such that 
\begin{equation}
\label{SLPM}
 |\mu(A(\cdot) B(h_1(u) \cdot))-\mu(A)\mu(B)|\leq C ||A||_\bs ||B||_\bs u^{-\kappa} .
\end{equation}
Indeed let $\theta \in [0,\pi]$ be such that $\cos \th = -e^{-t}$ and let 
$$R(t) =  \left(\begin{array}{ccccc}     \cos\theta & 0 & \dots & 0 & \sin\theta   \\
                                  
                                                            0 & 1 & \dots & 0 & 0 \\
                                                            \dots & \dots & \dots & \dots & \dots \\
                                                            0 & 0 & \dots & 1 & 0 \\
                                                            -\sin\theta & 0 &\dots & 0 & \cos\theta \\
\end{array}\right), $$
$$\fg_t=\textrm{diag}(e^{-t}, 0, \dots 0, e^t).$$
It is then immediate to observe that 
\begin{equation*}
 M(t):= h_1(e^t) R(t)^{-1} \fg_{-t}
 \end{equation*}
is uniformly bounded for $t>0$. By invariance of $\mu$ we get 
$$\mu(A(\cdot) B(h_1(e^t) \cdot)) = \mu(A(\cdot) B(M(t) \fg_t R(t) \cdot))=
\mu(A_t(\cdot) B_t(\fg_t \cdot))$$
with $ A_t(\cdot):= A ( R(t)^{-1} \cdot),$ $B_t(\cdot):= B (M_t \cdot).$ Note that
 $||A_t||_\bs \leq ||A||_\bs,$ $||B_t||_\bs \ll    ||B||_\bs$ {\color{blue} 
 uniformly in $t \geq 0$} Hence, \eqref{SLPM} follows from \eqref{SLEM} applied to $A_t$, $B_t$ and $\fg_t$ with $t=\ln u.$ 

 \subsection{  Equidistribution.}
\label{SS-EU}

The functions we are going to work with in Section \ref{SSDenom} are not smooth but they can be well approximated
by the smooth functions. This motivates the following definition. 

\begin{define} \label{sr}
Given $\bs, \br \geq 0$, we say that a function $A: \cM \to \reals$ is in $H^{\bs,\br}$ with $||A||_{\bs, \br}=K$ if given $0<\eps\leq 1$ there are 
$H^\bs$-functions $A^-\leq A\leq A^+$ such that
 $$ ||A^+-A^-||_{L^1(\mu)}\leq \eps \text{ and } ||A^\pm||_\bs\leq K \eps^{-\br}$$
where $\mu$ is the Haar measure on $\cM$ and $||\cdot||_\bs$ denotes the Sobolev norm of
index $\bs.$
\end{define}


We say that $\gamma$ is a $W_1$ curve of size $L>0$ if for some $y \in \cM$ we have 
$\gamma = \{ h_1(\tau) y : \tau \in [0,L]\}$. For a  function $A: \cM \to \reals$ we use the notation  
 $$\int_{\gamma} A = \frac{1}{L} \int_0^L A(h_1(s)y) ds.$$
 
 
\begin{define} \label{def.representative} Fix $\kappa_0 >0$.  Let $L>0$ and $\cP$ be a partition of $\cM$ into $W_1$-curves of length $L$ and denote by $\gamma(x)$ the element of $\cP$ 
containing $x$. Given a finite or infinite sequence of integers $(k_n)$ and a function $A \in H^{\bs,\br}$, we say that 
$\cP$ is $\kappa_0$-{\it representative} with respect to $((k_n), A)$ if for any $n$ 
\begin{equation}
\label{ShiftMix}
\mu\left(x \in \cM : \; \left|  \int_{g^{k_n}\gamma(x)} A - \mu(A) \right| \geq \cK_A 
L_n^{-\kappa_0}
\right)\leq L_n^{-\kappa_0}
\end{equation}
where  $\cK_A=||A||_{\bs, \br}+1$,  and $L_n = Le^{\lambda_q(k_n)}$ is the length of 
$g^{k_n} \gamma(x)$.

We call the points $x$ such that 
$$\text{for every }  n : \;    \left|  \int_{g^{k_n}\gamma(x)} A- \mu(A) \right|\leq \cK_A 
L_n^{-\kappa_0}$$
{\it representative} with respect to $(\cP, (k_n), A)$. Observe that if
$\cP$ is $\kappa_0$-{\it representative} with respect to $((k_n), A)$ and
$$ \sum_n \left(L_n\right)^{-\kappa_0}\leq \eps $$
then the set of representative points has measure larger than $1-\eps$.
\end{define}

The goal of this section is to show the following. 

\begin{prop} \label{mixing} There exists $\bs, \kappa_0, \eps_0>0$ such that 
for any $0\leq \br \leq \bs$, 
$0<\eps \leq \eps_0$, any function $A \in H^{\bs,\br}$, any  $ L$ and any sequence 
$\{ k_n \}$ satisfying 
$$ \sum_n \left(L e^{\lambda_q(k_n)}\right)^{-\kappa_0}\leq \eps, $$
there exists a partition $\cP$ of $M$ into $W_1$-curves of length $L$ that is  $\kappa_0$-{\it representative} with respect to $((k_n), A)$.

(b) If $\cL \in \cM$ is distributed according to a probability measure $\tmu$ that has a bounded
density with respect to the Haar measure $\mu$, then the result of part (a) holds with \eqref{ShiftMix} in the definition of representative partitions replaced by 
\begin{equation*}
\tmu \left(x \in \cM : \; \left|  \int_{g^{k_n}\gamma(x)} A - \mu(A) \right| \geq \cK_A 
L_n^{-\kappa_0}
\right)\leq \tC L_n^{-\kappa_0}
\end{equation*}
where $\tC$ is the maximum of the density of $\tmu$ with respect to $\mu.$
\end{prop}

\begin{remark} 
  The requirement that $\br \leq \bs$ will only serve to maintain the exponent $\kappa$ in the speed of equidistribution in (\ref{ShiftMix}) bounded from below. Any upper bound on $\br$ would yield a lower bound on $\kappa$  but it will be sufficient for us in the sequel to consider functions in $H^{\bs,\bs}$, since we will have to deal with characteristic functions of
  nice sets (cf. \S \ref{sec.norms}).
\end{remark}

\begin{proof}
It suffices to prove part (a). Part (b) follows from (a) since for any set $\Omega$ we have
$\tmu(\Omega)\leq\tC\mu(\Omega).$

Without loss of generality we will work with functions $A$ having zero average, that is $\mu(A)=0$. We will first prove Proposition 
\ref{mixing} for $A \in H^\bs$ and then generalize it to $A \in H^{\bs,\br}$.

Now, assuming that $\mu(A)=0$,  
\eqref{SLPM} implies that
$$ |\mu(A(\cdot) A(h_1(u)\cdot))|\leq C \cK_A^2 u^{-\kappa}$$
with $\cK_A =  ||A||_\bs$, thus for $ S_L(\cdot)=\frac{1}{L} \int_0^L A(h_1(u)\cdot) du$ we have
$$\mu(S_L^2)\leq C L^{-\kappa} \cK_A^2.$$
This implies that for $\kappa_0:=\kappa/3,$ we have 
\begin{equation}
\label{MES}
\mu(x \in \cM : |S_L(x)|> \cK_A L^{-\kappa_0})\leq C L^{-\kappa_0} .
\end{equation}

Next let $\hcP$ be an arbitrary partition of $M$ into $W_1$-curves of length $L$ and for $u\in [0,1]$
let 
$\hcP^u=h_1(Lu) \hcP.$ Then by \eqref{MES} 
$$ \brmu\left((x,u) \in \cM \times [0,1] : 
\left| \int_{\gamma(x,u)} A \right| > \cK_A L^{-\kappa_0}\right)
\leq C L^{-\kappa_0}$$
where $\gamma(x,u)$ denotes the piece of $\hcP^u$ that goes through $x$ and $\brmu$ denotes the product of $\mu$ and the Lebesgue measure on $[0,1]$.
 Thus, we can choose $u$ so that $\hcP^u$ satisfies
\begin{equation}
\label{SMix}
\mu\left(x \in \cM : 
\left| \int_{\gamma(x)} A \right| > \cK_A L^{-\kappa_0} \right)
\leq C L^{-\kappa_0}.
\end{equation}

If $L$ is large we can drop the constant $C$ if we let $\kappa_0$ be slightly smaller than $\kappa/3$. Likewise,  if $(k_n)$ is a finite or infinite sequence with 
$$ \sum_n \left(L e^{\lambda_u(k_n)}\right)^{-\kappa_0}\leq \eps $$
then there exists a partition $\cP$ that is representative with respect to $((k_n),A)$ as in Definition \ref{def.representative}.

To extend \eqref{SMix} to functions in $H^{\bs,\br}$ (that may have infinite $H^\bs$-norm), we use a standard approximation 
argument. Note first that (\ref{MES}) still holds for non zero mean $H^\bs$-functions if we replace $S^A_L$ by  $ {\bar S}^{A}_L(\cdot)=S^A_L-\mu(A). $

Now for $\eps>0$ let $A, A^+, A^-$ be as in Definition \ref{sr} where we assume that $\mu(A)=0$. Let $\cK_A=||A||_{\bs,\br}+1$.  Since 
$0\leq \mu(A^+)\leq \eps$, we have that 
\begin{equation}
\label{SBrS}
\mu\left(x: S_L^{A}(x) > 2 \cK_A L^{-\tilde{\kappa}}\right) \leq 
\mu\left(x: \bar{S}_L^{A^+}(x) >  2\cK_A L^{-\tilde{\kappa}} -\eps \right) . 
\end{equation}
So, if we choose $\eps$ and $\tilde{\kappa}$ such that   ${\eps} = \cK_A L^{-\tilde{\kappa}} \sim \cK_A \eps^{-\br} L^{-\kappa_0}$, that is  $\eps \sim L^{- \tilde{\kappa}}$ and $\tilde{\kappa} = \kappa_0/(r+1)$ we get from \eqref{SBrS}
using \eqref{MES} that 
 $$ \mu\left(x : S_L^A(x)> 2\cK_A L^{-\tilde{\kappa}}\right)\leq   \mu\left(x : \bar{S}_L^{A^+}(x)>||A||_{\bs, \br}
  L^{-\kappa_0}\right) \leq  L^{-\kappa_0}.$$
 Using $A^-$ to bound  $\mu\left(x : S_L^A(x)\leq -2 \cK_A L^{-\tilde{\kappa}}\right)$ we see that \eqref{SMix} and thus the rest of the proof extends to $H^{\bs, \br}$ functions, provided the exponent $\kappa_0$ is reduced. 
\end{proof}

If $\cA$ is a finite collection of functions we say that $\cP$ is representative with respect to $((k_n), \cA)$
if for each $A\in \cA,$ $\cP$ is representative with respect to $((k_n), A).$ 

 \subsection{  Mixing for approximately smooth functions.}
\S \ref{SS-EU}
controls the deviations of ergodic sums for $H^{\bs,\br}$-functions. We also need a bound on the 
rate of mixing for diagonal flows. Namely, let $A$ be a bounded $H^\bs$-function.

\begin{lemma}
\label{LmEM-sr}
There is a constant $C$ such that for any $H^{\bs,\br}$ function $B$ we have
$$  |\mu(A(\cdot) B(g \cdot))-\mu(A)\mu(B)|\leq C \left(||A||_{\bs}+||A||_{L^\infty}\right)
||B||_{\bs, \br} e^{-\frac{\kappa \max|t_j|}{r+1}}  $$
where $\kappa$ is the constant from \eqref{SLEM}.
\end{lemma}

\begin{proof}
Without the loss of generality we may assume that 
\begin{equation}
\label{BZM}
\mu(B)=0.
\end{equation}

Given $\eps$, let $B_\pm$ be the functions such that
$$ B_-\leq B\leq B_+, \quad \mu(B_+-B_-)\leq \eps \text{ and }
||B_\pm||_{\bs} \leq ||B||_{\bs, \br} \eps^{-r}. $$
Assume first that $A$ is positive. Then
\begin{equation}
\label{CorB-Bpm}
 \mu(A (\cdot) B_-(g\cdot)) \leq \mu(A (\cdot) B(g\cdot)) \leq \mu(A (\cdot) B_+(g\cdot)). 
 \end{equation}
Next, by \eqref{SLEM}
$$ |\mu(A (\cdot) B_+(g\cdot))|\leq \mu(A)\mu(B_+)+C||A||_\bs\; ||B||_{\bs, \br} \eps^{-r} e^{-\kappa \max|t_j|}. $$
Note that due to \eqref{BZM}
$$\mu(B_+)=\mu(B_+-B)\leq \mu(B_+-B_-)\leq \eps. $$
Hence
$$ |\mu(A(\cdot) B_+(g\cdot)|\leq C ||A||_\bs\; \left[\eps+||B||_{\bs, \br} \eps^{-r} e^{-\kappa \max|t_j|}\right]. $$
Choosing $\eps$ so that
$ \eps^{-(r+1)} e^{-\kappa \max|t_j|}=1$ we get
$$  |\mu(A(\cdot) B_+(g \cdot))|\leq C ||A||_\bs\; ||B||_{\bs, \br} e^{-\frac{\kappa \max|t_j|}{r+1}} . $$
Likewise
$$  |\mu(A(\cdot) B_-(g \cdot))|\leq C  ||A||_\bs \; ||B||_{\bs, \br} e^{-\frac{\kappa \max|t_j|}{r+1}} . $$
The last two inequalities together with \eqref{CorB-Bpm} prove the lemma for non negative $A.$ 

In the general case decompose
 $A=A_1-A_2$ where
$$ A_1=2||A||_{L^\infty}, \quad A_2=A_1-A.$$
Since both $A_1$ and $A_2$  are non-negative we have
$$  |\mu(A_j(\cdot) B(g \cdot))-\mu(A_j)\mu(B)|\leq C \left(||A||_\bs+||A||_{L^\infty}\right)
 ||B||_{\bs, \br} e^{-\frac{\kappa \max|t_j|}{r+1}}  $$
proving the lemma in the general case.
\end{proof}

 \section{  Poisson limit theorem for the diagonal action} \label{SSDenom}

 \subsection{  Overview of the proof of Theorem \ref{ThPLat}.} 
The goal of this section is to prove Theorem \ref{ThPLat} using the abstract Theorem \ref{th.poisson2} 
and the polynomial rate of uniform distribution of long pieces of horocycles given in Section \ref{SSSLMix}.

We fix the probability  space $(\Omega,\Prob)$ to be the space $(\cM,{\tmu})$ where  $\tmu$ is the measure from Theorem \ref{ThPLat} that is assumed to have a smooth bounded density with 
respect to the Haar measure.

In all this section, the expectation with respect to $\tmu$ of a variable $X$ will be denoted $\E(X)$.

{Let 
\begin{equation}\label{DefbPi} \bPi=\{\bt\in \R^d: \bt_j>0, \; T(\bt) \leq 1\},\end{equation} 
where $\displaystyle T(\bt)=\sum_{j=1}^d \bt_j.$ In this section we will denote by $\Pi_M=M\cP.$}
{(Note that, in fact, \eqref{DefbPi2} gives $\Pi_M=(M-d)\cP$ but we ignore "$-d$" in order to simplify
the notation. Since the result holds for arbitrary $M$, this does not cause any loss of genelity).
}
Let
$$\tcP=\{\bt\in \R_+^d : T(\bt) < 10 \} \text{ so that }
\tilde\Pi_M=\{\bt\in \R_+^d : T(\bt) < 10 M \}.$$
Recall the definition of $\lambda_1(t)= \bT(\lambda)+t_1$, and let $\hat{\l}(t)= t_1+M$. 
Observe that $\hat{\l}(t)> \l_1(t)$ on $\Pi_M$
(even though $\hat{\l}(t)$ can be equal to $\l_1(t)$ on the boundary of $\Pi_M$, that is, if 
$T(\bt)=M$).

 We take $(\fX,\fm)$ and  $(\tfX,\tfm)$ to be the spaces $K=\left[-\frac{1}{\breps }, \frac{1}{\breps}\right]$
  and $\R / 2 \Z$ equipped with their normalized Lebesgue measures.

 Fix any $P \geq 1$ and divide $K$
  into a finite number of intervals $K_1, K_2\dots K_{P}$ and  let  $Q$ be the partition  of $(\fX,\fm)$ into the intervals $K_1, K_2\dots K_{P}$. 

Similarly, fix any $J \geq 1$ and divide $[0,2)$ into a finite number of intervals  $\tilde K_1,\dots, \tilde K_{J}$, and  let  $\tilde{Q}$ be the partition  of $(\tfX,\tfm)$  into the intervals $\tilde K_1,\dots, \tilde K_{J}$.

  Recall the definitions of $\Phi, \Psi$ given in \eqref{DefPhi} \eqref{DefPsi} 
  and introduce 
$\Phi_p$, $p\in [1,P]$ that are defined by  formula  \eqref{DefPhi} with $K_p$ in place of $K$. 
Let \begin{align} \label{xi}
 \Phi^\bt&=\Phi(g^\bt \cL)
 \\ \Phi^\bt_p&=\Phi_p(g^\bt\cL)  \label{def.xi} \\
\nu_\bt&=\Psi_1(g^\bt \cL) \\
\zeta_\bt&= \Psi_2\left(g^\bt \cL,N e^{-T(\bt)}\right) \label{zeta} \end{align}

Fix $R$ to be a large number (the precise conditions on $R$ are described later in this section).

By Theorem \ref{th.poisson2}, if we prove that       
$\Phi^\bt,\{\Phi^\bt_p\}_{p \leq P}, \zeta_\bt, \nu_\bt$, for $\bt\in \Pi_m$ satisfy (h1)--(h8) 
then we get the Poisson limit for  
$$\{\Psi(g^{\bt} \Lambda(\xi), N e^{-T(\bT)}) \}_{\bt \in \Pi_M, \Phi(g^{\bt} \Lambda(\xi))=1}$$ 
required in Theorem \ref{ThPLat}.

Proposition \ref{LmMultSol}, proven in \S \ref{sec.mult}, 
shows that ${\Phi^\bt}$ satisfies (h1)--(h3). The proof relies on Rogers' identities given in Lemma \ref{LmRI}.
\S \ref{sec.norms} contains estimates of $||\cdot||_{\bs, \bs}$ norms of the functions $\Phi_i$ and $\Psi_j.$ 
Then in \S \ref{sec.partition} we show, using Proposition \ref{mixing},
the existence of  the partitions $F_\bt$ and the set $E$ 
such that (h4)--(h8) hold. 

 \subsection{  Multiple solutions.}  \label{sec.mult}

The following proposition asserts that ${\Phi^\bt}$ satisfies (h1)--(h3). 
\begin{prop}  
Recall the notations $\Phi^\bt= \Phi \circ g^\bt$, $\Phi_p^\bt= \Phi_p \circ g^\bt$.
\label{LmMultSol} 
Then uniformly in $ \bt\in \Z^d, \; \bt' \in \Z^d-\{\mathbf{0}\}$ we have
\begin{align*}
&(a) \quad \quad \quad {\tmu}(\Phi^\bt)= \cO( M^{-d}); \\
  &(a') \ \  \quad \quad {\tmu}(\Phi^\bt_p )=2^{d-1} \bc_1  |K_p|  M^{-d} +\cO(M^{-100d}) \\
 &  \text{provided that} \ \min_j (\bt_j) \geq R\ln M  \text{ and } R \text{ is sufficiently large;} \\
 &(b) \quad \quad \quad \tmu\left(\left(\Phi^\bt\right)^2-\Phi^\bt\right)=
 \cO\left(M^{-2d}\right);   \\
 &(c) \quad \quad \quad  \tmu( \{\cL \in \cM : \Phi^\bt (\cL)\neq 0 \text{ and } \Phi^\bt(g_{\bt'}\cL)\neq 0 \})=\cO\left(M^{-2d}\right). 
\end{align*}
\end{prop}
Note that (b) implies via Markov inequality that
\begin{equation}
\tmu( \{\cL \in \cM : \Phi^\bt(\cL)> 1\}) =\cO\left(M^{-2d}\right).
\end{equation}

\begin{proof} Without loss of generality, we can assume in the proof of the inequalities (a), (b), (c), that $\cL$ is distributed according to the Haar measure on $M$, and by invariance of the Haar measure take $t=0$. The inequalities then follow from Rogers' 
equalities of Lemma \ref{LmRI}. Namely, part (a)  of Lemma \ref{LmRI} implies  that 
$\mu(\Phi )=\bc_1 \Vol(\cD)=2^{d-1} \bc_1 |K| M^{-d}$, where $\cD$ is the set defined by \eqref{DefKeySet}.  
Indeed
$$\Vol(\cD)=\frac{|K|}{M^d} \int_{\R^d}  \mathbbm{1}_I(x_1(\vv)) 
\left( \prod_{j=2}^d \mathbbm{1}_J(x_j(\vv)) \right) \frac{dx}{\prod_{j=1}^d |x_j|}
= 2^{d-1} |K|M^{-d}. $$  
On the other hand, letting $f=\one_\cD$
we get,
since $I$ is an interval of positive numbers, that  
\begin{equation} 
\label{PhiSq-Phi}
\Phi^2(\cL) -\Phi(\cL) =    \sum_{\vv_1\neq \vv_2\in L \text{ prime}} f(\vv_1) f(\vv_2)  
=\sum_{\vv_1\neq \pm \vv_2\in L \text{ prime}} f(\vv_1) f(\vv_2) 
\end{equation}
so part (b) follows by Lemma \ref{LmRI} (b).
As for (c) observe that if we define, for $\vv=(x,z) \in \cL$, 
$$\tf(\vv)= \mathbbm{1}_{e^{-\bt_1'}I}(x_1(\vv)) 
\left( \prod_{j=2}^d \mathbbm{1}_{e^{-\bt_j'} J} (x_j(\vv)) \right) \mathbbm{1}_{e^{T(\bt')}K} (M^d \Pi(\vv) ), $$
$$  \mu(\Phi \Phi^{\bt'})=\int_{\cM} \sum_{\vv_2\neq \pm \vv_1\in \cL \text{ prime}} f(\vv_1) \tf(\vv_2)
 d\mu(\cL) $$
where the contribution of $\vv_2=-\vv_1$ vanishes because both  $I$ and $e^{-t'_1}I$ are positive intervals, while the contribution of $\vv_2=\vv_1$ vanishes since either 
$I$ and $e^{-\bt'_1} I$ are disjoint or $J$ and $e^{-\bt'_j} J$ for some $j=2, \dots d$
 are disjoint. Applying Lemma \ref{LmRI}(b) we get $(c)$.  
 
 Since 
 $\mu(\Phi_p )=2^{d-1} \bc_1 |K_p| M^{-d}$, ($a'$) follows by exponential mixing of the geodesic flow
 (Lemma \ref{LmEM-sr}) and Lemma \ref{LmNorms} from \S ~\ref{sec.norms}.
\end{proof}

 \subsection{  Estimates of norms.} \label{sec.norms}
Before we construct the partition $\cF_\bt$, we first  state  estimates on the $H^{\bs,\bs}$ norms of $\Phi$ and $\Phi_p$. We also obtain an estimate on the norm of $\Phi^2$ after making an appropriate cutoff. The results stated in Lemma \ref{LmNorms} below
 are proven in \S \ref{AppNS}.

Let $\fh_{1, \delta}$ be a smooth cutoff function supported on the set of lattices with 
a short vector of size $\cO(\delta).$ The existence of such function is guaranteed by Lemma
\ref{LmCutOff} from Appendix \ref{AppNorms}. We set $\fh_{2, \delta}=1-\fh_{1,\delta}.$

Let $\hK=\{z: d(z, \partial K)\leq M^{-1000d}\}$ and set $\hPhi=\cS(1_{\hat\cD})$ where $\hat\cD$ is defined similarly to $\cD$
(see equation \eqref{DefKeySet}) with $K$ replaced by $\hK.$
Define similarly $\hK_p,$ $\cD_p,$ and $\hPhi_p$ for $p\in [1,P]$.

\begin{lemma} For any $s \geq 0$ we have that 
\label{LmNorms}

(a) $\displaystyle{||\Phi||_{\bs, \bs}=\cO(1)},$ $\displaystyle{||\hPhi||_{\bs, \bs}=\cO(1)}.$
Also for each $p\in [1, P]$ 
$$||\Phi_p||_{\bs, \bs}=\cO(1),\quad ||\hPhi_p||_{\bs, \bs}=\cO(1).$$
 
(b) For each $\delta>0,$  $\displaystyle{||\Phi \fh_{1, \delta}||_{\bs, \bs}=\cO(1)}.$

(c) For each $\delta>0,$  $\displaystyle{||(\Phi^2-\Phi) \fh_{2, \delta}||_{\bs, \bs}=
\cO(\delta^{-2(d+1)})}.$
  
(d) $\displaystyle{  \mu(\Phi_i)=2^{d-1}  \bc_1 |K_i|, \quad  \mu(\hPhi_i)=\cO(M^{-1000d}). }$

(e) $\displaystyle \mu(\Phi \fh_{1, \delta})=\cO\left(\delta^{(d+1)/2}\right). $

(f) $\displaystyle \mu((\Phi^2-\Phi) \fh_{2, \delta})=\cO\left(M^{-2d}\right).$
\end{lemma}

 \subsection{  The partition $\cF_t$ and the proof of (h4)--(h8)} \label{sec.partition}

Given $\bt\in {\tilde\Pi}$ we denote by $\Pi^+(\bt) $ the set of $\bar\bt\in {\tilde\Pi}$ such that $ \lambda_1(\bar\bt)>\lambda_1(\bt)+R\ln M.$

Consider the following collection of functions
$$ \bPhi=\{\Phi, \Phi_1\dots \Phi_P, 
\hPhi,\hPhi_1,\ldots,\hPhi_P, \Phi \fh_{1, M^{-1000d}}, (\Phi^2-\Phi) \fh_{2, M^{-1000d}} \}. $$

Let  $F_\bt$ be a partition of $\cM$ into $W_1$-curves of  size 
$L_\bt=(e^{\lambda_1(t)} M^{1000d})^{-1}$, which is 
$\kappa_0$-representative 
with respect to $(\Pi^+(\bt) , \bPhi)$ (that is, representative  for all $\bar\bt \in \Pi^+(\bt)$). 
Such a partition exists
due to Proposition \ref{mixing}, Lemma \ref{LmNorms}  and the fact  that  
$$\sum_{ \bt\in \Pi,\;\;\bar{t} \in \Pi^+(\bt)} \left(L_\bt e^{\lambda_1(\bar\bt)}\right)^{-\kappa_0}=
\cO\left(M^{-{10^{100d}}}\right)$$ 
if $R$ is sufficiently large.
Moreover, if we let $E_1$ be the set 
of $\cL$ such that for any $\bt\in \Pi,$ 
$\cL$ 
is representative with respect to $(\cF_\bt, \Pi^+(\bt), \bPhi)$ then 
$$\mu(E_1^c)=\cO(M^{-100d}).$$

\begin{prop} \label{proph4h8} There exist  sets $E$ with 
$${\tmu}(E^c) \ll    M^{-100d}, $$
such that the variables ${\Phi^\bt},\eta_t,\Phi^\bt_p,\eta_{\bt,p}$ and the partitions $F_\bt$ satisfy the properties (h4)--(h8) of Section 
 \ref{sec.poisson.abstract}.
\end{prop}

\begin{proof}

\noindent {\it Property (h4).}  We prove that any $\cL\in E_1$ satisfies (h4)
(with $\omega$ replaced by $\cL$).

Properties (h4a) and (h4b) follow from parts (a) and $(a')$ of 
Proposition~\ref{LmMultSol} and the definition of representative points.

To check (h4c) note that $\Phi^\bt_p$ are integer valued and so
$$ \Phi^\bt_p-\eta_{\bt,p}\leq {\Phi^\bt_p}^2-\Phi^\bt_p\leq {\Phi^\bt}^2-{\Phi^\bt}. $$
Since also $0\leq \xi_{\bt, p}-\eta_{\bt,p} \leq \Phi^\bt_p\leq {\Phi^\bt}$ we get
$$ 0\leq \Phi^\bt_p-\eta_{\bt,p}\leq \hxi_\bt $$
where
$$ \hxi_\bt=\left[(\Phi^2-\Phi) \fh_{2, M^{-1000d}}+\Phi \fh_{1, M^{-1000d}}\right]\circ g^\bt. $$
Accordingly for $\cL\in E_1$ 
$$ 0\leq \E(\Phi^\bt_p-\eta_{\bt,p}|\cF_{\bt'})\leq \E(\hxi_{\bt}|\cF_{\bt'})\leq \frac{C}{M^{2d}}$$
where the last inequality relies on parts (e) and (f) of Lemma \ref{LmNorms} and the fact that $\cL$ is representative with respect
to $(F_{\bt'}, \bt, \bPhi).$ The last display implies that
$$ \E(\Phi^\bt_p|\cF_{\bt'})-\frac{C}{M^{2d}}\leq \E(\eta_{\bt,p}|\cF_{\bt'})\leq \E(\Phi^\bt_p|\cF_{\bt'}) .$$
Hence (h4c) follows from (h4b).
\medskip $ \ $ 

\noindent {\it Property (h5).} 
{
Consider $\bt \in \Pi,$ $\bar\bt \in \Pi^+(\bt)$, and $\gamma_{\bar\bt} \in F_{\bar\bt}.$ 
If ${\Phi^\bt}\equiv 0$ on $\gamma_{\bar\bt}$ then (h5) clearly holds on $\gamma_{\bar\bt}.$
On the other hand suppose that there exists $\bar{\cL}\in \gamma_{\bar\bt}$ and $p$ such that
$\xi_{p, \bt}(\bar\cL)\neq 0.$ If follows that $g^\bt(\bar\cL)$ contains a vector 
$(\bar x, \bar z)\in I\times J^{d-1}\times K_p.$ Note that
$$ g^{\bt} \gamma=\{h_\tau g^\bt(\bar\cL)\}_{\tau\in \cI} $$
where $\cI$ is an interval containing zero of length which less than
$M^{-10^9\;d}.$ Observe that $h_\tau g^\bt(\bar\cL)$ contains a vector 
$(x_\tau, z_\tau)$ where $x_\tau=\bar x,$ $z_\tau=\bar z+\tau x_1.$
In particular $|z_\tau-z|<10^{-1000 d}.$
Hence if $g^\bt \gamma_{\brt}$ does not intersect  
$$\hat\Ka_p=\{z: d(z, \partial \Ka_p)\leq M^{-10000d}\}$$ then it is completely contained in 
$\Ka_p.$ 
The measure of $\cL$ such that  $g^\bt \gamma_{\bar\bt}$ intersects $\hat\Ka_p$ for some $p$ is thus bounded by 
$\cO(M^{-1000d})$ from Lemma \ref{LmNorms}(d). 
Taking the complement to the union of all these exceptional $\cL$ for all $\bt \in \Pi, \bar\bt \in \Pi^+(\bt)$ we get a set $E_2$ such that $\tmu(E_2^c)=\cO(M^{-999d})$ and
(h5) holds for $\cL\in E_2.$}
\medskip $ \ $

\noindent {\it Property (h6).} Since the size of the pieces of $F_\bt$ is $L_\bt=(e^{\lambda_1(\bt)} M^{1000d})^{-1}$ 
and the size of the pieces of $F_{\bar\bt}$ is 
$\displaystyle L_{\bar\bt}=\left(e^{\lambda_1(\bar\bt)} M^{1000d}\right)^{-1}$ 
if we let
$$ E_3=\{\cL: F_{\bar\bt}(\cL)\subset F_\bt(\cL) \text{ for all } \bt\in\Pi, \bar\bt\in \Pi^+(\bt)\}$$
then we have $\tmu(E_3^c)=\cO(M^{-100d}).$
\medskip $ \ $

We let $E_4:=E_1 \cap E_2\cap E_3$ and observe that $\tmu(E_4^c)=O(M^{-100d})$ 
and (h4), (h5), and (h6) hold on $E_4$.  
\medskip $ \ $ 

\noindent {\it Property (h7).}  
Let $\tKa_1,\ldots, \tKa_J$ be a partition of $\R / 2\Z$ with $J$ intervals. 
We will show that $\cL\in E_4$ satisfies (h7). Then we further refine $E_4$ to ensure (h8). 

Assume $\bt,\bt'$ are such that 
$ \hat{\l}(\bt) \geq \lambda_1(\bt')+R \ln M \geq \lambda_1(\bt)+2R \ln M$.  
We need to show that if $\cL\in E$ is such that ${\Phi^\bt}=\Phi^\bt_p=1$ then for any $j \in [1,J]$

\begin{equation} \label{equid}  \tmu\left(   \Psi_2\left( g^\bt \cL,\frac{N}{e^{T(\bt)}}\right)\in \tKa_j
   \big|\cF_{\bt'}   \right)(\cL) =|\tKa_j| (1+o(1)). \end{equation}
 Let $\gamma_{\bt'}=F_{t'}(\cL).$ Then $\gamma_{\bt'}$ is of the form
$$\gamma_{\bt'}=\{h_\tau^1 \brcL\}_{0\leq \tau \leq (e^{\lambda_1(\bt')} M^{1000d})^{-1}}$$
for some $\brcL\in \cM.$ 
By property (h5) $\Phi^\bt_p=1$ on $\gamma_{\bt'}.$ 
In particular, $\Phi_p(g^\bt\brcL)=1,$ that is  $g^\bt  \brcL$ contains a vector 
$(x,z)\in \cD_p.$ 
Since
$g^\bt h_\tau = h_{e^{\l_1(\bt)} \tau} g^\bt$ it follows that
$g^\bt h_\tau \brcL$
contains the vector $(x_\tau,z_\tau)\in \cD_p.$ 
Namely, $x_\tau=x,$  $z_\tau=z+e^{\lambda_1(\bt)} \tau x_1 . $
 Now
\begin{equation} 
\label{HoroImage}
\Psi_2\left( g^\bt h^1_\tau \brcL, \frac{N}{e^{T(\bt)}}  \right)= 
\frac{N}{e^{T(\bt)}} z_\tau \text{ mod } 2
= \frac{N}{e^{T(\bt)}} z + e^{\hat{\l}(\bt)} \tau x_1 \text{ mod } 2. 
\end{equation}
and since $\tau$ varies on an interval of length $ (e^{\lambda_1(\bt')} M^{1000})^{-1}$, 
the uniform distribution (\ref{equid}) follows from the fact that  
$ \hat{\l}(\bt) \geq \lambda_1(\bt')+R \ln M$ 
provided that $R$ is sufficiently large.
\medskip $ \ $ 

\noindent {\it Property (h8).}  $F_{\bar\bt}(\cL)$ is of the form 
$$ \{h_\tau^1 \brcL\}_{0\leq \tau \leq (e^{\lambda_1(\bar\bt)} M^{1000})^{-1}}$$
for some $\brcL\in \cM.$ By (h5), ${\Phi^\bt}=1$ on $F_{\bar\bt}(\omega).$
\eqref{HoroImage} shows that if
$\mathbbm{1}_{\zeta_\bt\in \tK_j}$ is not constant on $F_{\bar\bt}(\cL)$ then
$\Psi_2\left( g^\bt h^1_\tau \brcL, \frac{N}{e^{T(\bt)}}  \right)$ lies in a $\cO(M^{-1000d})$ neighborhood 
of $\partial \tKa_j.$ Let

\begin{multline*}  E=\left\{\cL\in E_4 : \forall  \bt, \bar\bt \in \Pi \text{ with }  \lambda_1(\bar\bt)\geq \hlambda(\bt)+R\ln M \text{ it holds that :}  \right. \\  
\left.  \forall j\in [1,J], \quad \mathbbm{1}_{\zeta_\bt\in \tKa_j} \text{ is constant on } F_{\bar\bt}(\cL)\right\} .\end{multline*} 
Then (h8) holds on $E.$ On the other hand,
the same argument as for property (h5) shows that
$\tmu(E-E_4)$ $=\cO(M^{-999d})$ as needed.
\medskip $ \ $ 

We have thus checked (h4)--(h8) for     ${\Phi^\bt},\{\Phi^\bt_p\}, \{\zeta_\bt\}$ for ${\bt\in  \Pi}$ and $p\in [1,P]$, which completes the proof of Proposition \ref{proph4h8}.  \end{proof}

\begin{proof}
[Proof of Theorem \ref{ThPLat}.]  Theorem \ref{ThPLat}(a) is exactly property (h4a) that we proved in Proposition \ref{proph4h8}.

{To prove Theorem \ref{ThPLat}(b), we note that it
follows
from Theorem \ref{th.poisson2} and from properties (h1)--(h8) that we proved in Propositions \ref{LmMultSol}  and  \ref{proph4h8}, that the process 
$$\left\{\left(\Psi\left(g^\bt \cL, N e^{-\sum_{j=1}^d \bt_j} \right), \frac{\bt}{M} \right) 
\right\}_{\Phi(g^\bt \cL)=1,\; \bt \in \Pi_M}$$
converges in probability, as $N\to\infty$, to a Poisson process  on 
$$ \left[-\frac{1}{\breps},\frac{1}{\breps}\right] \times   \R/(2\Z) \times \cP $$ with
intensity $2^{d-1} \bc_1 .$ Since $\hat\cP=\frac{1}{d!}$
the claim of Theorem \ref{ThPLat}(b) follows from the invariance of Poisson processes by projection
given by Lemma \ref{LmPT}(b). 
}
\end{proof}

 \section{  Small boxes}
\label{ScSmall}
One can also consider the visits to small boxes 
$\cC_N=\prod_j \left[-\frac{u_j}{N^\gamma}, \frac{u_j}{N^\gamma}\right].$ 
The case $\gamma=0$ is treated in Theorem \ref{dimd} while the case $\gamma=1/d$
was studied in \cite{Mar1}. For $\gamma>1/d$ most orbits do not visit $\cC_N$ so we
consider the remaining case $0<\gamma<\frac{1}{d}.$ Recall \eqref{DefBRho}.

\begin{theo}
\label{ThSmBox}
Under the assumptions of Theorem \ref{dimd}, $\frac{D(x, \a,{\cC_N},N)}{\brho ((1-d\gamma)\ln N)^d}$
converges to  the standard Cauchy distribution.
\end{theo}

The proof of Theorem \ref{ThSmBox} is very similar to the proof of Theorem \ref{dimd}
so we just describe the necessary changes.

Similarly to Theorem \ref{dimd} the proof consists of two parts: showing that non-resonant
terms are negligible and establishing the Poisson limit for the resonant terms.

To describe the first part let
$$ Z(\xia, N, \gamma)=\{k\in Z(\xia, N): |\brk_i|>N^\gamma\}, $$
$$ \bar{\holako}(\xia, N, \gamma)=\sum_{k\in Z(\xia, N, \gamma)}
\frac{\Gamma_k(\xia, N, \gamma)}{\Omega_k(\xia, N)} $$
where $\brk_i$ are defined by \eqref{sidedual},
$$ \Gamma_k(\xia, N, \gamma)=\frac{2A}{\pi}
\phi(\brk_1 u_1 N^{-\gamma}, \dots, \brk_d u_d N^{-\gamma},
N(k,\alpha), \langle k,x \rangle+\varphi_{k, \alpha, N}) $$
and $\Omega_k$ and $\phi$ are defined by \eqref{DefTheta} and \eqref{DefGamma}
respectively. We have the following analogue of Proposition \ref{prop.reduction}. 

\begin{prop} \label{prop.reduction.small}
  For any $\upsilon>0$, if we take $\breps>0$ sufficiently small and
  $N$ sufficiently large we have that 
  \begin{equation*} 
  \lambda\left( \left\{ \xia \in X :
    \left| \frac{\bar{\holako}(\xia,N,\gamma)}{(\ln N)^d} -
    \frac{D(x, \a,{\cC_N},N)}{(\ln N)^d} \right|  \geq \upsilon \right\} \right)\leq \upsilon . 
    \end{equation*}
\end{prop}

Proposition \ref{prop.reduction.small} 
follows from the estimates of Section \ref{sec3} with the 
exception that \S \ref{nosmallk} has to be modified to take into account that now we remove more frequencies
(namely, we now discuard the frequncies with $|\brk_i|<N^\gamma$).

Define $U_k(\xia, N, \gamma)$ similarly to $U_k(\xia, N)$ with $u_j$ replaced by $u_j N^{-\gamma}$
and let
$$ \holako_1(\xia, N, \gamma)
=\sum_{0<|\brk_i|<N} U_k(\xia, N, \gamma), 
\quad
 \holako_2(\xia, N, \gamma)=\sum_{N^\gamma<|\brk_i|<N} U_k(\xia, N, \gamma). $$
We claim that
\begin{equation}
\label{D12gamma}
\frac{\holako_1-\holako_2}{(\ln N)^d} \text{ converges to }0 \text{ in probability as } N\to\infty.
\end{equation}
To this end  fix a small $\teps>0$ and let
$$ \tY(q_1, \dots q_d)=\{k\in \Z^d: |\brk_j|\in [q_j, q_j+1]\}, \quad
Q(q_1, \dots q_d)=\prod_{j=1}^d \max(q_j, N^\gamma),   
$$
 $$ \cQ=\{(q_1,\dots, q_d)\in \Z^d: \min(|q_j|)<N^\gamma, \max(|q_j|)<2N\}, $$
$$\tE=\{\a: \exists (q_1,\dots , q_d)\in \cQ\text{ and }k\in \tY(q_1,\dots q_d): 
||\langle k,\a \rangle||\leq \teps Q(q_1, \dots q_d)\}. 
$$
Then 
\begin{equation}
\label{TESmall}
\mes(\tE)\leq \teps.
\end{equation}
On the other hand, since for $k\in \tY(q_1,\dots q_d)$ we
have
$$ |U_k(\xia, N, \gamma)|\leq \frac{|\cos(2\pi\langle k,x \rangle+\phi_{k,\alpha, N})|}
{Q(q_1, \dots q_d)||\langle k,\a \rangle||} $$
we get, repeating the arguments of Subsection \ref{nosmallk}, that 
\begin{equation}
\label{L2TEComp}
 ||\holako_2-\holako||_{L^2((\T^d-\tE)\times \T^d)}\leq C \frac{(\ln N)^{d-1}}{\sqrt{\teps}}. 
\end{equation}
Combining \eqref{TESmall} and \eqref{L2TEComp} we obtain \eqref{D12gamma}. 
Combining \eqref{D12gamma} with the estimates of Section \ref{sec3} we obtain 
Proposition \ref{prop.reduction.small}.

Next, Theorem \ref{ThPoisResd}  has to be modified as follows.

\begin{prop} \label{ThPoisResSmall} 
Assume $\xia \in X$ is distributed according to the normalized Lebesgue measure $\lambda$. For any $\bar{\eps}>0$, as $N \to \infty$, the process
$$\left\{\left((\ln N)^d  \left(\prod_i \bar{k}_i\right)  \norm{\langle k,\a \rangle}, \left(N\langle k,\a \rangle \text{ mod } 2\right) , 
\{ \bar{k}_1 u_1 N^{-\gamma} \},\ldots,\{ \bar{k}_d u_d N^{-\gamma} \}, \right. \right.$$
$$\left. \left.  \{\langle k,x \rangle\}\right)\right\}_{k \in Z(\xia, N,\gamma)} $$
converges to a Poisson process on $[- \frac{1}{\breps}, \frac{1}{\breps}]\times (\R/ (2\Z)) \times \T^{d+1}$ with intensity 
$2^{d-1} \bc_1 (1-\gamma d)^d /d!$ where $\bc_1=1/\zeta(d+1)$ is the constant from Lemma~\ref{LmRI}. 
\end{prop}

The Poisson limit for the pair 
\begin{equation}
\label{FirstTwo}
\left\{(\ln N)^d  \left(\prod_i \bar{k}_i\right)  \norm{\langle k,\a \rangle}, \left( N\langle k,\a \rangle \text{ mod } 2\right) \right\} 
\end{equation}
follows from the abstract Theorem \ref{th.poisson2} with
$\bPi=\{t: t_j>0, \;\sum_j t_j<M\}$ 
replaced by
$\bPi_\gamma=\{t: t_j>\gamma M,\; \sum_j t_j<M\}. $ Note that
the change of variables $\brt_j=t_j-\gamma M$ transforms $\bPi_\gamma$ to the simplex
$$\{\brt: \brt_j>0, \; \sum_j \brt_j<(1-\gamma d) M\}$$ so that 
$\Vol(\bPi_\gamma)=(1-d\gamma)^d \; \Vol(\bPi). $ This explains the extra factor $(1-\gamma d)^d$ 
in Theorem \ref{ThSmBox}.

The asymptotic independence of the remaining coordinates in Proposition \ref{ThPoisResSmall}  from the 
pair \eqref{FirstTwo} is a consequence of the following modification of Propoition \ref{ThLacun}.

Given $s\in\N$ consider a sequence of $s$-tuples 
$(k^{(1,N)} , \dots k^{(s,N)})$ where $k^{(j,N)}\in \Z^d.$
Let $\brk^{(j, N)}$ be vector with components
$$ \brk^{(j, N)}_p=a_{p,1} k^{(j,N)}_1+\dots+a_{p,d} k^{(j,N)}_d. $$

\begin{prop}
\label{ThLacunGamma} 
Suppose that
\begin{equation}
\label{1Split}
\forall j\in\{1,\dots, s\} \quad |\brk^{(j,N)}|>N^\gamma e^{\sqrt{\ln N}}, 
\end{equation}
and
\begin{equation}
\label{2Split}
\forall j'\neq j''\in\{1,\dots, s\} \quad \left|\ln|\brk^{(j',N)}|-\ln|\brk^{(j'',N)}|\right|>
e^{\sqrt{\ln N}} 
\end{equation}
Let $(x,u)$ be distributed according
to a density $\rho_N$ on $\T^d\times T^d$ such that 
\begin{equation}
\label{SmoothPhaseD}
||\rho_N||_{C^1}\leq R.
\end{equation}
Then the distribution of the $s$ $(d+1)$-tuples
$$\left( \{(\brk^{(1,N)}, u_1 N^{-\gamma})\},\dots \{(\brk^{(1,N)}, u_d N^{-\gamma})\},
\{(k^{(1,N)}, x)\}\right.,$$
$$\dots$$
$$\left.\{(\brk^{(d,N)}, u_1 N^{-\gamma})\},\dots \{(\brk^{(d,N)}, u_d N^{-\gamma})\},
\{(k^{(d,N)}, x)\}\right) $$
converges to the uniform distribution on $\T^{(d+1)s}$ as $N\to\infty$ and the convergence is uniform with
respect to the matrix $(a_{pq})$  the choices of $\{k^{(j,N)}\}_{j=1}^s$ satisfying  \eqref{1Split} and \eqref{2Split}, 
and $\rho_N$ satisfying \eqref{SmoothPhaseD}.
\end{prop}

The proof of Proposition \ref{ThLacunGamma}  
is similar to the proof of Proposition \ref{ThLacun}  so we omit it.
 
Apart from the modifications described above the proof of Proposition \ref{ThPoisResSmall}  is identical to the proof of Theorem \ref{ThPoisResd}.

Also, the derivation of Theorem \ref{ThSmBox} from Propositions \ref{prop.reduction.small} and \ref{ThPoisResSmall}
is the same as the derivation of Theorem \ref{dimd}  from Proposition \ref{prop.reduction} and Theorem \ref{ThPoisResd}.

 \section{  Continuous time}
\label{ScCont}
In this section we discuss briefly the behavior of the discrepancy function in the case of  linear flows on the torus. Given a set $\cC$ we consider the continuous time discrepancy function  
$$\bD(v,x,\cC,T)  = \int_{0}^{T} \one_{\cC}(S_v^t x)dt - T \Vol({\cC})$$
where $S_v^t=x+vt.$ 

In the case of balls, it was shown in \cite{dbconvex} that for $d\geq 4$, the continuous time discrepancy function has a similar behavior as the discrete time discrepancy, namely  it converges in distribution after normalization by a factor  $T^{(d-3)/2(d-1)}$. 

Curiously,  for balls in dimension $d=3$, the continuous time discrepancy behaves 
similarly to the discrete discrepancy of boxes and gives rise to a Cauchy distribution after normalization by $\ln T$. This will be proved in \S  \ref{secballs} below. 

It was also shown in \cite{dbconvex} that for balls in dimension $d=2$ the  continuous time discrepancy  converges, without any normalization, in distribution. In \S \ref{seccont} we will show that this is 
also the case in any dimension $d\geq 2$ for the continuous time discrepancy for boxes.

 \subsection{  Boxes.} \label{seccont} Let $\cC=\MA (\prod_j (0, u_j)).$
We assume that the triple $(\MA,x, v)$ is distributed according to 
a smooth density of compact support and
that $\MA \in \rm{SL}_d(\R)$ is such that $||\MA-I||\leq \eta$ where $\eta$ is sufficiently small.

\begin{theo}
\label{ThCTRB}
As $T\to\infty,$ $\bD(v,x,\cC,T)$ converges in distribution. 
\end{theo}

\begin{proof}
We have
$$ \bD(v,x,\cC,T)=4^d \sum_{k} \left[ \prod_j \left(\frac{\sin\left(2\pi \brk_j u_j\right)}{\brk_j} \right) \right]
\frac{\sin(\pi\langle k,vT \rangle)}{\pi\langle k,v \rangle} \cos (2\pi\langle k,x \rangle+\phi_{k, T, v}) .$$
where 
$\brk_j$ is given by \eqref{sidedual}.
We claim that for almost all $\MA, v$ there exist
a constant $C(\MA, v)$ such that
$$||\bD(v,x,\cC,T)||_{L^2_x} \leq C(\MA,v) $$
and moreover for each $\eps$ there exists $N=N(\MA,v)$ such that
$$\left\Vert \sum_{|k|>N} \left[ \prod_j \left(\frac{\sin\left(2\pi \brk_j u_j\right)}{\brk_j} \right)\right]
\frac{\sin(\pi\langle k,vT \rangle)}{\pi\langle k,v \rangle} \cos (2\pi\langle k,x \rangle+\phi_{k, T, v}) \right\Vert_{L^2_x}\leq \eps. $$ 
To this end it suffices to demonstrate that for almost every $(\MA,v)$ 
$$ \sum_k \left(\left(\prod_j  \brk_j \right) \langle k,v \rangle\right)^{-2} <\infty. $$
Since $\det(\MA)\neq 0$ there exists $\delta(\MA)$ such that for each 
$k$ there is $l\in\{1\dots d\}$ such that
$|\brk_l|>\delta |k| .$ Accordingly it suffices to check that for each $l$
$$\sum_k \Gamma_k(\MA,v) <\infty \quad \text{where} \quad
\Gamma_k(\MA,v)=\left(\left(\prod_{j\neq l} \brk_j \right) \langle k,v \rangle |k| \right)^{-2}. $$
All sums have the same form so we consider the case $l=d.$ 
Given numbers $s_1,\dots s_{d-1}, s_d$ and $\eps>0$ denote
$\Omega(k, s_1\dots s_d)=$
$$\{(\MA, v): |\brk_j| \in [|k|^{s_{j}}, |k|^{s_{j}+\eps}] 
\text{ for }j=1,\dots, d-1\text{ and } 
|\langle k,v \rangle| \in [|k|^{s_{d}}, |k|^{s_{d}+\eps}]\}. $$
Then
$$\Prob(\Omega(k, s_1\dots s_d)) \ll    |k|^{s+d\eps-d} $$
where $s=\sum_{j=1}^d s_j.$ We draw two conclusions from this estimate.
First, for almost all $(\MA,v)$ we have
$$ \left|\left(\prod_{j=1}^{d-1} \brk_j \right) \langle k,v \rangle\right|>|k|^{-2d\eps} $$
provided that $|k|$ is sufficiently large.

Second, for $s\geq -2d\eps$ we have
$$ \E(\mathbbm{1}_{\Omega(k, s_1\dots s_d)}(\MA,v) \Gamma_k(\MA,v))\leq  
C |k|^{d\eps-[(d+2)+s]}.$$
Hence
$$ \E\left( \sum_k \mathbbm{1}_{\Omega(k, s_1\dots s_d)}(\MA,v) \Gamma_k(\MA,v)\right) <\infty . $$
Summing over all $d$-tuples $(s_1\dots s_d)\in (\eps\mathbb{Z})^d$ such that
$$ s_j \leq 1, \quad s=\sum_{j=1}^d s_j> -2 d \eps $$
we get 
$\displaystyle \E\left( \sum_k  \Gamma_k(\MA,v)\right) <\infty  $
proving our claim. 

The claim implies that for large $N$ the distribution of $\bD(v,x,\cC,T)$ is close to the distribution of
$$ \bD_{ N}^-(v,x,\cC,T)=4^d \sum_{|k|\leq N} \prod_j \left(\frac{\sin\left(2\pi \brk_j u_j\right)}{\brk_j} \right)
\frac{\sin(\pi\langle k,vT \rangle}{\pi\langle k,v \rangle} \cos (2\pi\langle k,x \rangle+\phi_{k, T, v}) .$$
Hence it remains to prove that $\bD_{N}^-(v,x,\cC,T)$ converges in distribution as
$T\to\infty.$ This convergence follows easily from the fact that as $T\to\infty$
$\{vT\}$ becomes uniformly distributed on $(\R/2\Z)^d.$
\end{proof}

A similar argument shows that randomness in $\cC$ is not necessary. Namely, we have the following result.

\begin{theo}
\label{ThCTFB}
Let $\cC=\prod_j (0, u_j)$. Suppose that the pair $(x,v)$ has a smooth distribution of compact support.
Then 
$\bD(v,x,\cC,T)$ converges in distribution as $T\to\infty.$  
\end{theo}

The proof of Theorem \ref{ThCTFB} is similar to the proof of Theorem \ref{ThCTRB} with the additional 
simplifications since now $|\bar k_j|\geq 1$ and so 
only $\langle k,v \rangle$ may possibly be small. Therefore we leave the proof to the reader.

 \subsection{  Balls.}  \label{secballs} In this section,  $\cC$ is assumed to be a ball of radius $r$ in $\T^3$. We suppose that $v$ is chosen according
to a smooth density $p$ whose support is compact and does not contain the origin, $r$ is uniformly 
distributed on some segment $[a,b],$  $x$ is uniformly distributed on $\T^3$ and $v,$ $r$ and $x$
are independent.

\begin{theo}
\label{ThCauchyBalls}
There exists a constant $\trho$ such that
$ \frac{\bD(v,x,{B(0,r)},T)}{\trho r \ln T}$ converges as $T\to\infty$ to the standard Cauchy distribution.
\end{theo}

\begin{proof}
The proof is similar to the proof of Theorem \ref{dimd}  so we just outline the main steps. We have
$$\bD(v, x,{B(0,r)}, T)=\sum_{k\in\Z^3} f_k(r,v,x,T)=\sum_{k\in\Z^3, k \text{ prime}}  g_k $$
where
$f_k = c_k \frac{\cos[2\pi\langle k,x \rangle+\pi\langle k,Tv \rangle] \sin(\pi\langle k,Tv \rangle)}{\pi \langle k,v \rangle},$
$g_k=\sum_{p=1}^\infty f_{kp} $ and 
$$c_k\sim \frac{r}{\pi |k|^2} \sin(2\pi r |k|). $$
Similarly to Section \ref{sec3}  (see also \cite[Section 3 and \S 6.4]{dbconvex})                                                                                                                                                                                                                                                                                                                                                                                                                        we show that the main contribution to the discrepancy comes from
the harmonics where 
$$\frac{\eps}{\ln T}<|\langle k,v \rangle| |k|^2< \frac{1}{\eps \ln T}\quad\text{and}\quad
|k|<T.$$ Therefore the key step in proving Theorem \ref{ThCauchyBalls} is the following.
\begin{prop} 
\label{BallPoisson}
The point process
$$ \left\{ |k|^2 \langle k,v \rangle \ln T, \langle k,Tv \rangle \text{\rm{ mod }}2, \{\langle k,x \rangle\}, \{r|k|\}\right\}_{|k|\leq T, \eps k^2 |\langle k,v \rangle| \ln T<1, k \text{ prime} } $$
converges as $T\to\infty$ to a Poisson process on
$[-\frac{1}{\eps}, \frac{1}{\eps}]\times (\R/2\Z)\times (\R/\Z)^2$ with constant intensity.
\end{prop}
The proof of Proposition \ref{BallPoisson} is similar to the proof of Theorem \ref{ThPoisResd} and consists of the following steps.

(a) We prove the Poisson limit for $\{|k|^2 \langle k,v \rangle \ln T\}$ using the argument of Section \ref{SSDenom}. We first normalize one of the coordinates, say $v_3$, of the vector $v$ to $1$, which reduces the study of the Poisson limit for $\{|k|^2 \langle k,v \rangle \ln T\}$ to the study of the visits to the cusp in $\cM={\rm SL}_3(\reals)/{\rm SL}_3(\integers)$ of $g_t \Lambda$ with 
$$ g_t=\left(\begin{array}{lll} e^{t} & 0 & 0 \\
                                                            0 & e^{t} & 0 \\
                                                            0 & 0 & e^{-2t} \\
\end{array}\right)
\text{ and }
\Lambda(v_1,v_2)=\left(\begin{array}{lll} 1 & 0 & 0 \\
                                                            0 & 1 & 0 \\
                                                            v_1 & v_2 & 1 \\
\end{array}\right).$$
More precisely, the relevant neighborhood in the cusp is defined {\it via} the function 
$$ f(x, y, z)=\mathbbm{1}_I(x^2+y^2) \mathbbm{1}_K((x^2+y^2) z) $$
where
$I=[1,e)$, $K=[-\frac{1}{\breps \ln T},\frac{1}{\breps \ln T}]$.  We then define 
 $\Phi(\cL)=\cS(f).$

The Poisson limit of $\{|k|^2 \langle k,v \rangle \ln T\}$  is obtained from a Poisson limit for 
$$\{\Phi(g_{t} \Lambda) \}_{t \in [0,\ln T], 
}.$$ 
In this setting, the manifold determined by $\Lambda(v_1,v_2)$ consists of the full strong unstable foliation of $g_t$ and there is no need for extra parameters to establish the Poisson limit. 

(b) We prove that $\langle k,Tv \rangle \text{\rm{ mod }}2$ is asymptotically independent of $|k|^2\langle k,v \rangle \ln T$ using the fact
that their values are determined at different scales (cf. proof of (h7) in Section \ref{SSDenom}).

(c) We show that $\langle k,x \rangle$ and $\{r|k|\}$ are independent of the previous data using the superlacunarity 
of the sequence of small denominators (cf. Proposition \ref{ThLacun}). 
\end{proof}  

\appendix
 \section{  Norms}
\label{AppNorms}
 \subsection{  Preliminaries.}
It is well known that the fluctuation of ergodic integrals depends strongly on the regularity properties of the observables.
To gauge such regularity we will need several norms on the space of lattices.
 
Let $C^\bs(\cM_{d+1})$ denote  the space of smooth functions on $\cM_{d+1}.$ 
Let $$\fU_1, \fU_2, \dots ,\fU_{(d+1)^2-1}$$ 
be a basis in the space of left invariant vectorfields on $\cM_{d+1}.$ We let
$$ ||\Phi||_{C^\bs}=\max_{0\leq k\leq \bs} \max_{i_1, i_2\dots i_k} 
\max_{\cL\in \cM_{d+1}} \left|\partial_{\fU_{i_1}} \partial_{\fU_{i_2}} \dots \partial_{\fU_{i_k}} \Phi(\cL) \right|.
$$
Let $H^\bs$ denote the Sobolev space of index $\bs.$ It is equipped with the norm
$$ ||\Phi||_s^2=\sum_{0\leq k\leq \bs} \sum_{i_1, i_2\dots i_k} \int
\left|\partial_{\fU_{i_1}} \partial_{\fU_{i_2}} \dots \partial_{\fU_{i_k}} \Phi(\cL) \right|^2 d\mu(\cL).
$$

Let $\ba(\cL)$ denote the length of the shortest nonzero vector in $\cL.$
\begin{lemma}
\label{LmCutOff}
For each $\bs$ there are constants $C_1, C_2$ such that for each $\delta\leq 1$ there is a function
$\fh_{1, \delta}:\cM\to\reals$ such that
\begin{itemize}
\item $0\leq \fh_{1, \delta}\leq 1,$
\item $\fh_{1,\delta}(\cL)=1$ if $\ba(\cL)\leq \delta,$
\item $\fh_{1,\delta}(\cL)=0$ if $\ba(\cL)\geq C_1\delta,$
\item $||\fh_{1,\delta}||_{C^\bs(\cM)} \leq C_2.$
\end{itemize}
\end{lemma}

\begin{proof}
This lemma is a special case of  \cite[\S 4.2]{KM2}. For completeness, we reproduce the formula from 
\cite{KM2}. Let $\Upsilon$ be a nonnegative function on $SL_{d+1}(\R)$ with integral one supported on the set
$$ ||g||^2\leq C_1, \quad ||g||^{-2}\leq C_1. $$
Then one can set
$$ \fh_{1, \delta}(\cL)=\int_{SL_{d+1}(\R)}  \Upsilon(g) \; \mathbbm{1}_{\ba(g\cL)\leq C_1 \delta}\; d\mu(g).
 \qedhere
$$
\end{proof}

We also need a space $H^{\bs, \br}$ of functions on $\cM_{d+1}$ which can be well approximated by $H^\bs$ functions
(see Definition \ref{sr}). Similar norms can be introduced on the Euclidean space $\reals^{d+1}.$
We note the following inequalities for $\Phi,\Psi \in C^\bs(\cM_{d+1})$
\begin{equation}
\label{Hs-CsLat}
||\Phi||_{\bs}\leq C_3 ||\Phi||_{C^\bs},
\end{equation}
\begin{equation}
\label{Hs-CsProd}
|||\Psi \Phi||_{\bs}\leq C_4 ||\Psi||_{C^\bs} ||\Phi||_{\bs}
\end{equation}
where the constants $C_3$ and $C_4$ depend on $\bs.$
Accordingly if $\Psi \in C^\bs(\cM_{d+1})$ is positive and $\Phi \in H^{\bs, \br}$ we get
\begin{equation}
\label{Hsr-CsrProd}
|||\Psi \Phi||_{\bs, \br}\leq C_5 ||\Psi||_{C^\bs} ||\Phi||_{\bs, \br}
\end{equation}
(In fact, \eqref{Hsr-CsrProd} holds for arbitrary smooth $\Psi$ since $\Psi$ can be represented as a difference
of two smooth positive functions, but we will only need \eqref{Hsr-CsrProd} for positive $\Psi$.)

We also need a space $C^{\bs, \br}(\R^{d+1})$ which is defined similarly to $H^{\bs, \br}.$

Namely, given $\bs, \br \geq 0$, we say that a function $f: \R^{d+1} \to \reals$ is in $C^{\bs,\br}$ 
with $||f||_{C^{\bs, \br}}=K$ 
if given $0<\eps\leq 1$ there are 
$C^\bs$-functions $f^-\leq f\leq f^+$ such that
$$ ||f^+-f^-||_{L^1(\R^{d+1})}\leq \eps \text{ and } ||f^\pm||_{C^\bs(\R^{d+1})} \leq K \eps^{-\br}.$$

\begin{lemma} 
\label{CrHs-Cs}
 For each integer $\bs$ and each $R$ there is a constant $C=C(R, \bs)$ such that:

(a) If $f$ is a $C^\bs(\reals^{d+1})$ function supported in the ball of radius $R$ about the origin then
\begin{equation}
\label{Hs-Cs}
 ||\cS(f)||_{H^\bs(\cM_{d+1})}\leq C ||f||_{C^\bs(\reals^{d+1})}. 
\end{equation}
and  if $f$ is a $C^{\bs, \br}(\reals^{d+1})$ function supported in the ball of radius $R$ about the origin then
\begin{equation}
\label{Hsr-Csr}
 ||\cS(f)||_{H^{\bs, \br}(\cM_{d+1})}\leq C ||f||_{C^{\bs, \br}(\reals^{d+1})}; 
\end{equation}
(b) Let $\fh_{2, \delta}=1-\fh_{1, \delta}$ where $\fh_{1, \delta}$ is a function from Lemma \ref{LmCutOff}.
Let $f$ be a $C^\bs(\reals^{d+1})$ function supported in the ball of radius $R$ about the origin. 
Then
$$ ||\cS(f) h_{2, \delta}||_{C^\bs(\cM_{d+1})} \leq C ||f||_{C^\bs(\R^{d+1})}\; \delta^{-(d+1)}. $$
\end{lemma}

\begin{proof}
Given a left invariant vectorfield $\fU$ on ${\rm SL}_{d+1}(\R)$ let
$\brfU$ be the corresponding left invariant vectorfield on $\reals^{d+1}.$
That is 
$$ (\partial_\brfU f)(x)=\frac{d}{dt}\big|_{t=0} f(g(t) x) $$
where $g(t)$ is a one parameter subgroup of ${\rm SL}_{d+1}(\R)$ such that $g'(0)=\fU.$

Since
$\partial_{\fU} \cS(f)=\cS(\partial_{\brfU} f),$
\eqref{Hs-Cs} follows from Lemma~\ref{LmRI}(c).

\eqref{Hsr-Csr} follows from \eqref{Hs-Cs} since $f^-\leq f \leq f^+$ implies 
$$\cS(f^-)\leq \cS(f) \leq \cS(f^+).$$

Next
$$ \left| \cS (f)(\cL) \fh_{2, \delta}(\cL)\right|\leq 
\mathbbm{1}_{\ba(\cL)\geq \delta} \sum_{\vv\in\cL, \text{ prime}} |f(\vv)|.$$
However if the shortest vector in $\cL$ is longer than $\delta$ there are at most $\cO(\delta^{-(d+1)})$
terms contributing to this sum. Accordingly 
$$||\cS(f) h_{2, \delta}||_{C^0(\cM_{d+1})}\leq C||f||_{C^0(\R^{d+1})}\; \delta^{-(d+1)}. $$
The higher derivatives are estimated similarly.
\end{proof}

 \subsection{  Proof of results of Section \ref{sec.norms}}
\label{AppNS}

\begin{proof}[Proof of Lemma \ref{LmNorms}]
(a) Let $\phi$ be a $C^\infty$ function such that $\phi(z)=1$ for $z\leq 0,$
$\phi(z)=0$ for $z\geq 1$ and $0\leq \phi(z)\leq 1$ for $0\leq z\leq 1.$
Given an interval $K=[k_1, k_2]$ let 
\begin{align*} 
\phi^+_{K,\eps}(z)&=\frac{1}{2}\left[\phi\left(\frac{z-k_2}{\eps}\right) -\phi\left(\frac{z-k_1+\eps }{\eps}  \right)\right], \\
\phi_{K,\eps}^-(z)&=\frac{1}{2}\left[\phi\left(\frac{z-k_2+\eps}{\eps}\right) -\phi\left(\frac{z-k_1 }{\eps}  \right)\right] . \end{align*}
Consider the following functions on $\reals^{d+1}$
$$ f_\eps^\pm(x, z)= \phi_{I, \eps}^\pm(x_1) 
\left(\prod_{j=2}^d \phi_{J, \eps}^\pm(x_j) \right) \phi_{K_M, \eps}^\pm(\Pi(\vv)),$$
where $K_M=[-\frac{1}{\breps M^2}, \frac{1}{\breps M^2}],$ and define as in \eqref{DefPhi} 
the Siegel transforms $\Phi_\eps^\pm$ of $f_\eps^\pm$ instead of  
$$f=\mathbbm{1}_I(x_1) \left(\prod_{j=2}^d \mathbbm{1}_J(x_j)\right) \mathbbm{1}_{K_M}(\Pi(\vv)).$$ 
Since
$ f_\eps^-\leq f \leq f_\eps^+ $
we conclude that 
$||f||_{C^{\bs, \bs}(\R^3)}=\cO(1).$ Now \eqref{Hsr-Csr} shows that $||\Phi||_{\bs, \bs}=\cO(1).$
The norms of $\hPhi, \Phi_p$ and $\hPhi_p$ are
estimated similarly.

Part (b) follows from part (a) and \eqref{Hsr-CsrProd}.

Next, \eqref{PhiSq-Phi} shows that 
$$(\Phi_\eps^-)^2-\Phi_\eps^-\leq \Phi^2-\Phi\leq 
(\Phi_\eps^+)^2-\Phi_\eps^+.$$
Thus
$$
 \fh_{2, \delta} \left[(\Phi_\eps^-)^2-\Phi_\eps^-\right] \leq \fh_{2, \delta} \left[\Phi^2-\Phi\right] 
 \leq
\fh_{2, \delta} \left[(\Phi_\eps^+)^2-\Phi_\eps^+\right]. $$

We have
\begin{align*} 
 \mu\left(\fh_{2, \delta} 
\left(\left[(\Phi_\eps^+)^2-\Phi_\eps^+\right]-\left[(\Phi_\eps^-)^2-\Phi_\eps^-\right]\right)\right) &  \\
\leq 
 \mu
\left(\left[(\Phi_\eps^+)^2-\Phi_\eps^+\right]-\left[(\Phi_\eps^-)^2-\Phi_\eps^-\right]\right)&=\cO(\eps). 
\end{align*}
where the last step relies on Lemma \ref{LmRI}(c). 

Next, similarly to the proof of Lemma \ref{CrHs-Cs}
(b) we get
$$ ||((\Phi_\eps^\pm)^2-\Phi_\eps^\pm)h_{2, \delta} ||_{s,s}=\cO\left(\delta^{-2(d+1)}\right) $$
proving part (c).

Part (d) follows directly from Lemma \ref{LmRI}(d). 

To prove part (e)  we note that
$$ \mu(\Phi \fh_{1, \delta})\leq \sqrt{\mu(\Phi^2) \mu(\fh_{1, \delta}^2)}\leq
C \sqrt{ \mu(\fh_{1, \delta}^2)}\leq \brC \delta^{(d+1)/2} $$
where the estimate of $\mu(\Phi^2)$ follows from Lemma \ref{LmRI}(c) and the estimate of
$\mu(\fh_{1, \delta}^2)$ follows from the fact that
$$ \fh_{1, \delta}\leq \cS(\mathbbm{1}_{x^2+z^2\leq (C_1 \delta)^2}) $$
and Lemma \ref{LmRI}(c).

Finally part (f) follows from Proposition \ref{LmMultSol}(b) since $\fh_{2, \delta}\leq 1.$
\end{proof}

 \section{  Independence} \label{appB}
\begin{proof}[Proof of Sublemma \ref{SLOff}]
We start with \eqref{R3}. Observe that 
$$\E(\E(V_{j+1}|\cF_j))=\E(V_{j+1})= \cO\left(\frac{\ln M}{M}\right)$$
where the last step comes directly from (h3) and the fact that there is $\cO(M^{2d-1} \ln M)$ 
terms in $V_{j+1}$.  Let 
\begin{equation}
\label{EPrime}
E_j'=\left\{\omega\in E: \E(V_{j+1}|\cF_j)(\omega)\leq \frac{1}{\sqrt{M}}\right\}.
\end{equation}
Then $\Prob(E_j'^c)=\cO(\ln M/\sqrt{M})$ by 
Markov inequality while \eqref{R3} holds for $\omega \in E_j'$.

We turn now to \eqref{re}. Let $\bt, \bar{\bt} \in \Pi^{j+1}$ such that  
$\lambda_1(\bar{\bt})  \geq \lambda_1({\bt})+3R \ln M$. 

Let now $\hat{\bt}$ be such that 
\begin{equation}\label{hatt} 
\l_1(\bar{\bt}) -R \ln M \geq   \l_1(\hat{\bt}) \geq \l_1(\bt) +R \ln M. 
\end{equation} 
Recall that $\tilde{\bt}$ is such that 
\begin{equation*}   \min_{\bt \in \Pi^{j+1} } \lambda_1({\bt})-R \ln M 
\geq \lambda_1(\tilde{\bt})  \geq  \max_{\bt \in \Pi^j } \lambda_1({\bt})+R \ln M . \end{equation*}
and that $\cF_j= \cF_{\tilde{\bt}}.$
We let  $F_j(\omega):=F_{\tilde{\bt}}(\omega).$ We then consider the partition 
$\tF_\bt=F_\bt\wedge F_j$ and let
$\tcF_\bt$ denote the $\sigma$-algebra generated by $\tF_\bt.$

The idea in addressing \eqref{re} is essentially the following: 
\begin{align*}  \E( \eta_{\bt} \eta_{\bar\bt} | \cF_{j})&\overset{\text{by (h6)}}{\sim}& 
\E \left( \E(\eta_{\bt} \eta_{\bar\bt}|  \tcF_{\hat\bt}) | \cF_{j}\right)
&\overset{\text{by (h5)}}{\sim}& \E \left( \eta_{\bt}  \E( \eta_{\bar\bt}|  \tcF_{\hat\bt}) | \cF_{j}\right) \\
&\overset{\text{by (h4)}}{\sim} & \frac{\bbc \fm(\fX)}{M^d} \E ( \eta_{\bt}  | \cF_{j})  
&\overset{\text{by (h4)}}{\sim} & \frac{\bbc \fm(\fX)}{M^d}  \frac{\bbc \fm(\fX)}{M^d}. \end{align*}

The fact that (h4)--(h6) only hold outside of some exceptional set, makes the argument above incomplete and we now complete it. Due to (h6), we have that  
$F_{\hat{\bt}} (\omega)\subset F_j(\omega)$ for $\omega\in E.$ 
Hence, for $\omega\in E$ we have 
\begin{align*}\E( \eta_{\bt} \eta_{\bar\bt}|\cF_{j})&=&
\E\left(\E\left(\eta_{\bt}|\tcF_{\hat\bt}\right) \eta_{\bar\bt}|\cF_j\right)+\cR_1(\bt, \bar\bt) 
\\
&=&\E\left(\E\left(\eta_{\bt}|\tcF_{\hat\bt}\right) \E\left(\eta_{\bar\bt}|\tcF_{\hat\bt}
\right)|\cF_j\right)+\cR_1  \\
&=&\E\left(\E\left(\eta_{\bt}|\tcF_{\hat\bt}\right) 
\E\left(\eta_{\bar\bt}|\cF_{\hat\bt}\right)|\cF_j\right)+\cR_1 +\cR_2 \\
&=&\E\left(\eta_{\bt} \E\left(\eta_{\bar\bt}|\cF_{\hat\bt}\right)|\cF_j\right)+\cR_1 +\cR_2 +\cR_3 \\
&=&\E\left(\eta_{\bt} \mathbbm{1}_E \E\left(\eta_{\bar\bt}|\cF_{\hat\bt}\right)|\cF_j\right)
+\cR_1 +\cR_2 +\cR_3 
+\cR_4 \\
&\leq&\frac{C}{M^d} \E\left(\eta_\bt|\cF_j\right)+\cR_1 +\cR_2 +\cR_3 +\cR_4 \\
&\leq& \frac{C}{M^{2d}} +\cR_1 +\cR_2 +\cR_3 +\cR_4 \\
\end{align*}
where the inequalities follow from (h4c) and for $i=1,\ldots,4$, $\cR_i=\cR_i(\bt, \bar\bt)$ 
are given by 
\begin{align*}\cR_1&=\E\left(\left[\eta_\bt-\E\left(\eta_\bt|\tcF_{\hat\bt}\right)\right]
\eta_{\bar\bt}\Big|\cF_j\right), \\ 
\cR_2&=\E\left(\E\left(\eta_{\bt}\Big|\tcF_{\hat\bt}\right) \left[\E\left(\eta_{\bar\bt}\Big|\tcF_{\hat\bt}\right)-
\E\left(\eta_{\bar\bt}|\cF_{\hat\bt}\right)\right]\Big|\cF_j\right), \\ 
\cR_3&=\E\left(\left[\E\left(\eta_{\bt}\Big|\cF_{\hat\bt}\right)-
\eta_\bt\right]\E\left(\eta_{\bar\bt}|\cF_{\hat\bt}\right)|\cF_j\right),\\ 
\cR_4&=\E\left(\eta_{\bt} \E\left(\eta_{\bar\bt}|\cF_{\hat\bt}\right)|\cF_j\right)-
\E\left(\eta_{\bt} \mathbbm{1}_E \E\left(\eta_{\bar\bt}|\cF_{\hat\bt}\right)\Big|\cF_j\right). 
\end{align*}
Note that $\cR_l$ are $\cF_j$ measurable for $l\in\{1\dots 4\}.$
We claim that all  of them have $L^1$ norm of order $\cO\left(M^{-100d}\right).$
Indeed, first,
$$ \E\left(\left|\cR_4\right|\right)\leq \Prob(E^c)=\cO\left(M^{-100d}\right). $$
Next
$$ \E\left(\left|\cR_1\right|\right)\leq \E\left(\left|\eta_\bt-
\E\left(\eta_\bt|\tcF_{\hat\bt}\right)\right|\right)$$
$$=
\E\left(\left|\eta_t-\E\left(\eta_t|\tcF_\Ht\right)\right|\mathbbm{1}_E\right)+\cO\left(M^{-100d}\right)=
\cO\left(M^{-100d}\right)
$$
since on $E,$ $\tF_{\hat\bt}(\omega)=F_{\hat\bt}(\omega)$ due to (h6) and hence 
$\E\left(\eta_\bt|\tcF_{\hat\bt}\right)=\E\left(\eta_\bt|\cF_{\hat\bt}\right)=\eta_t$
due to (h5). $\E\left(\left|\cR_3\right|\right)$ 
and $\E\left(\left|\cR_2\right|\right)$ are estimated similarly.

Let now
$$E_j''=\{\omega\in E: \forall \bt, \bar\bt\in \Pi^{j+1}  \text{ and } \hat{\bt} \text { as in \eqref{hatt} }  \forall l\in \{1\dots 4\} \; |\cR_l(\bt, \bar\bt)|\leq M^{-2d}\}. $$
Then by Markov inequality $\Prob((E''_j)^c)=\cO\left(M^{-90d}\right)$ and 
\eqref{re} holds for $\omega\in E_j''.$

Letting $E_j=E_j'\cap E_j''$ where $E_j'$ is given by \eqref{EPrime},
 we finish the proof of Sublemma \ref{SLOff}.
\end{proof}

\bigskip 
\noindent {\sc \bg Acknowledgments. }  We thank the referees for many comments and suggestions that were useful in improving the presentation of the first versions of this paper.
The second author was supported by the grant ANR-15-CE40-0001.

\end{document}